\documentclass[11pt]{article}
\usepackage{epsfig,epsf,fancybox}
\usepackage{amsmath}
\usepackage{mathrsfs}
\usepackage{amssymb}
\usepackage{graphicx}
\usepackage{color}
\usepackage{multirow}
\usepackage{paralist}
\usepackage{verbatim}
\usepackage{galois}
\usepackage{algorithm}
\usepackage{algpseudocode}
\usepackage{boxedminipage}
\usepackage{booktabs}
\usepackage{accents}
\usepackage{stmaryrd}
\usepackage{subcaption}
\usepackage{wrapfig}
\usepackage[bottom]{footmisc}
\usepackage{natbib}
\usepackage[normalem]{ulem}
\usepackage{xspace}
 \usepackage{tikz}
 \usepackage{wrapfig}
 \usepackage{color}
\usetikzlibrary{fadings}
\usetikzlibrary{patterns}
\usetikzlibrary{shadows.blur}
\usetikzlibrary{shapes}

\usepackage{array}		
\usepackage{booktabs}		
\usepackage[inline,shortlabels]{enumitem}		
\setenumerate{itemsep=0pt,parsep=\smallskipamount,topsep=\smallskipamount,left=\parindent}

\usepackage[colorlinks,linkcolor=magenta,citecolor=blue, pagebackref=true,backref=true]{hyperref}
\renewcommand*{\backrefalt}[4]{%
    \ifcase #1 \footnotesize{(Not cited.)}%
    \or        \footnotesize{(Cited on page~#2.)}%
    \else      \footnotesize{(Cited on pages~#2.)}%
    \fi}

\long\def\comment#1{}
\usepackage{nicefrac}

\newtheorem{theorem}{Theorem}[section]

\newtheorem{lemma}[theorem]{Lemma}
\newtheorem{proposition}[theorem]{Proposition}

\newtheorem{definition}{Definition}[section]
\newtheorem{example}{Example}[section]

\newtheorem{remark}{Remark}[section]
\newtheorem{assumption}{Assumption}[section]

\numberwithin{equation}{section}

\newcommand{\st}{\textnormal{s.t.}}

\newcommand{\Exp}{\textnormal{Exp}}

\newcommand{\grad}{\textnormal{grad}}
\newcommand{\subg}{\textnormal{subgrad}}

\newcommand{\x}{\mathbf x}
\newcommand{\y}{\mathbf y}

\newcommand{\MCal}{\mathcal{M}}

\newcommand{\SCal}{\mathcal{S}}
\newcommand{\XCal}{\mathcal{X}}
\newcommand{\YCal}{\mathcal{Y}}

\newcommand{\proj}{\textnormal{proj}}

\newcommand{\St}{\textnormal{St}}

\newcommand{\RMMO}{RMMO\xspace}

\newcommand{\br}{\mathbb{R}}

\newcommand{\bn}{\mathbb{N}}
\newcommand{\ba}{\begin{array}}
\newcommand{\ea}{\end{array}}

\newcommand{\EE}{{\mathbb{E}}}

\newcommand{\NCal}{\mathcal{N}}

\newcommand{\bigO}{O}

\newcommand{\mydefn}{:=}

\begin{document}

\begin{center}

{\bf{\LARGE{First-Order Algorithms for Min-Max Optimization \\ [.2cm] in Geodesic Metric Spaces}}}

\vspace*{.2in}
{\large{
\begin{tabular}{c}
Michael I. Jordan $^{\diamond, \dagger}$ \and Tianyi Lin $^\diamond$ \and Emmanouil V. Vlatakis-Gkaragkounis$^\diamond$ \\
\end{tabular}
}}

\vspace*{.2in}

\begin{tabular}{c}
Department of Electrical Engineering and Computer Sciences$^\diamond$ \\
Department of Statistics$^\dagger$ \\ 
University of California, Berkeley \\
\end{tabular}

\vspace*{.2in}

\today

\vspace*{.2in}

\begin{abstract}
From optimal transport to robust dimensionality reduction, a plethora of machine learning applications can be cast into the min-max optimization problems over Riemannian manifolds. Though many min-max algorithms have been analyzed in the Euclidean setting, it has proved elusive to translate these results to the Riemannian case.~\citet{Zhang-2022-Minimax} have recently shown that geodesic convex concave Riemannian problems always admit saddle-point solutions. Inspired by this result, we study whether a performance gap between Riemannian and optimal Euclidean space convex-concave algorithms is necessary. We answer this question in the negative---we prove that the Riemannian corrected extragradient (RCEG) method achieves last-iterate convergence at a linear rate in the geodesically strongly-convex-concave case, matching the Euclidean result. Our results also extend to the stochastic or non-smooth case where RCEG and Riemanian gradient ascent descent (RGDA) achieve near-optimal convergence rates up to factors depending on curvature of the manifold.  
\end{abstract}

\end{center}

\section{Introduction}\label{sec:introduction}
Constrained optimization problems arise throughout machine learning, in classical settings such as dimension reduction~\citep{Boumal-2011-RTRMC}, dictionary learning~\citep{Sun-2016-Complete-I, Sun-2016-Complete-II}, and deep neural networks~\citep{Huang-2018-Orthogonal}, but also in emerging problems involving decision-making and multi-agent interactions. While simple convex constraints (such as norm constraints) can be easily incorporated in standard optimization formulations, notably (proximal) gradient descent~\citep{Raskutti-2015-Information,Giannou-2021-Survival,Giannou-2021-Rate,Antonakopoulos-2020-Online,Vlatakis-2020-Noregret}, in a range of other applications such as matrix recovery~\citep{Fornasier-2011-Low,Candes-2008-Enhancing}, low-rank matrix factorization~\citep{Han-2021-Riemannian} and generative adversarial nets~\citep{Goodfellow-2014-Generative}, the constraints are fundamentally nonconvex and are often treated via special heuristics.

Thus, a general goal is to design algorithms that systematically take account of special geometric structure of the feasible set~\citep{Mei-2021-Leveraging, Lojasiewicz-1963-Propriete,Polyak-1963-Gradient}. A long line of work in the machine learning (ML) community has focused on understanding the geometric properties of commonly used constraints and how they affect optimization;~\citep[see, e.g.,][]{Ge-2015-Escaping, Anandkumar-2016-Efficient, Sra-2016-Geometric,Jin-2017-Escape, Ge-2017-No, Du-2017-Gradient, Reddi-2018-Generic,Criscitiello-2019-Efficiently, Jin-2021-Nonconvex}. A prominent aspect of this agenda has been the re-expression of these constraints through the lens of Riemannian manifolds. This has given rise to new algorithms~\citep{Sra-2015-Conic,Hosseini-2015-Matrix} with a wide range of ML applications, inclduing online principal component analysis (PCA), the computation of Mahalanobis distance from noisy measurements~\citep{Bonnabel-2013-Stochastic}, consensus distributed algorithms for aggregation in ad-hoc wireless networks~\citep{Tron-2012-Riemannian} and maximum likelihood estimation for certain non-Gaussian (heavy- or light-tailed) distributions~\citep{Wiesel-2012-Geodesic}.

Going beyond simple minimization problems, the robustification of many ML tasks can be formulated as min-max optimization problems. Well-known examples in this domain include adversarial machine learning \citep{Kumar-2017-Semi,Chen-2018-Metrics}, optimal transport~\citep{Lin-2020-Projection}, and online learning~\citep{Mertikopoulos-2018-Riemannian, Bomze-2019-Hessian,Antonakopoulos-2020-Online}. Similar to their minimization counterparts, non-convex constraints have been widely applicable to the min-max optimization as well~\citep{Heusel-2017-Gans, Daskalakis-2018-Limit, Balduzzi-2018-Mechanics, Mertikopoulos-2019-Optimistic, Jin-2020-Local}. Recently there has been significant effort in proving tighter results either under more structured assumptions~\citep{Thekumprampil-2019-Efficient, Nouiehed-2019-Solving, Lu-2020-Hybrid, Azizian-2020-Tight,Diakonikolas-2020-Halpern, Golowich-2020-Last, Lin-2020-Near, Lin-2020-Gradient, Liu-2021-First, Ostrovskii-2021-Efficient, Kong-2021-Accelerated}, and/or obtaining last-iterate convergence guarantees~\citep{Daskalakis-2018-Limit, Daskalakis-2019-Last, Mertikopoulos-2019-Optimistic,Adolphs-2019-Local, Liang-2019-Interaction, Gidel-2019-Negative, Mazumdar-2020-Gradient, Liu-2020-Towards, Mokhtari-2020-Unified, Lin-2020-Near, Hamedani-2021-Primal, Abernethy-2021-Last, Cai-2022-Tight} for computing min-max solutions in convex-concave settings.  Nonetheless, the analysis of the iteration complexity in the general \emph{non-convex} \emph{non-concave} setting is still in its infancy~\citep{Vlatakis-2019-Poincare,Vlatakis-2021-Solving}. In response, the optimization community has recently studied how to extend standard min-max optimization algorithms such as gradient descent ascent (GDA) and extragradient (EG) to the Riemannian setting. In mathematical terms, given two Riemannian manifolds $\MCal, \NCal$ and a function $f: \MCal\times \NCal \to \br$, the Riemannian min-max optimization (\RMMO) problem becomes 
\begin{equation*}
\min_{x \in \MCal} \max_{y \in \NCal} f(x, y).
\end{equation*}
The change of geometry from Euclidean to Riemannian poses several difficulties. Indeed, a fundamental stumbling block has been that this problem may not even have theoretically meaningful solutions. In contrast with minimization where an optimal solution in a bounded domain is always guaranteed \citep{Fearnley-2021-Complexity}, existence of such saddle points necessitates typically the application of topological fixed point theorems~\citep{Brouwer-1911-Abbildung, Kakutani-1941-Generalization}, KKM Theory~\citep{Knaster-1929-Beweis}). For the case of convex-concave $f$ with compact sets $\XCal$ and $\YCal$, ~\citet{Sion-1958-General} generalized the celebrated theorem~\citep{Neumann-1928-Theorie} and guaranteed that a solution $(x^\star,y^\star)$ with the following property exists
 \[\min_{x\in\XCal} f(x,y^\star) = f(x^\star,y^\star) = \max_{y\in\YCal} f(x^\star,y).\]
However, at the core of the proof of this result is an ingenuous application of Helly's lemma \citep{Helly-1923-Mengen} for the sublevel sets of $f$, and, until the work of~\citet{Ivanov-2014-Helly}, it has been unclear how to formulate an analogous lemma for the Riemannian geometry. As a result, until recently have extensions of the min-max theorem been established, and only for restricted manifold families~\citep{Komiya-1988-Elementary,Kristaly-2014-Nash, Park-2019-Riemannian}.

\citet{Zhang-2022-Minimax} was the first to establish a min-max theorem for a flurry of Riemannian manifolds equipped with unique geodesics. Notice that this family is not a mathematical artifact since it encompasses many practical applications of \RMMO, including Hadamard and Stiefel ones used in PCA~\citep{Lee-2022-Fast}. Intuitively, the unique geodesic between two points of a manifold is the analogue of the a linear segment between two points in convex set: For any two points $x_1,x_2 \in \XCal$, their connecting geodesic is the unique shortest path contained in $\XCal$ that connects them.

Even when the \RMMO is well defined, transferring the guarantees of traditional min-max optimization algorithms like Gradient Ascent Descent (GDA) and Extra-Gradient (EG) to the Riemannian case is non-trivial.  Intuitively speaking, in the Euclidean realm the main leitmotif of the last-iterate analyses the aforementioned algorithms is a proof that  $\delta_t=\|x_t-x^*\|^2$ is decreasing over time. To achieve this, typically the proof correlates $\delta_t$ and $\delta_{t-1}$ via a ``square expansion,'' namely:
\begin{equation}\label{eq:square-expansion}
\underbrace{\|x_{t-1}-x^*\|^2}_{\alpha^2} = \underbrace{\|x_{t}-x^*\|^2}_{\beta^2} + \underbrace{\|x_{t-1}-x_{t}\|^2}_{\gamma^2}-\underbrace{2\langle x_{t}-x^*,x_{t-1}-x_{t}\rangle}_{2\beta\gamma\cos(\hat{A})}.
\end{equation}
Notice, however that the above expression relies strongly on properties of Euclidean geometry (and the flatness of the corresponding line), namely that the the lines connecting the three points $x_t$, $x_{t-1}$ and $x^*$ form a triangle; indeed, it is the generalization of the Pythagorean theorem, known also as the law of cosines, for the induced triangle ${(ABC)}:=\{(x_{t},x_{t-1},x^*)\}$.
In a uniquely geodesic manifold such triangle may not belong to the manifold as discussed above. As a result, the difference of distances to the equilibrium using the geodesic paths $d_\MCal^2(x_{t}, x^* ) - d_\MCal^2(x_{t-1}, x^* )$ generally cannot be given in a closed form. The manifold's curvature controls how close these paths are to forming a Euclidean triangle. In fact, the phenomenon of \emph{distance distortion}, as it is typically called, was hypothesised by~\citet[Section~4.2]{Zhang-2022-Minimax} to be the cause of exponential slowdowns when applying EG to \RMMO problems when compared to their Euclidean counterparts. 

Multiple attempts have been made to bypass this hurdle.~\citet{Huang-2020-Gradient} analyzed the Riemannian GDA (RGDA) for the non-convex non-concave setting. However, they do not present any last-iterate convergence results and, even in the average/best iterate setting, they only derive sub-optimal rates for the geodesic convex-concave setting due to the lack of the machinery that convex analysis and optimization offers they derive sub-optimal rates for the geodesic convex-concave case, which is the problem of our interest. The analysis of~\citet{Han-2022-Riemannian} for Riemannian Hamiltonian Method (RHM), matches the rate of second-order methods in the Euclidean case. Although theoretically faster in terms of iterations, second-order methods are not preferred in practice since evaluating second order derivatives for optimization problems of thousands to millions of parameters quickly becomes prohibitive. Finally, \citet{Zhang-2022-Minimax} leveraged the standard averaging output trick in EG to derive a sublinear convergence rate of $O(1/\epsilon)$ for the general geodesically convex-concave Riemannian framework. In addition, they conjectured that the use of a different method could close the exponential gap for the geodesically strongly-convex-strongly-convex scenario and its Euclidean counterpart.

Given this background, a crucial question underlying the potential for successful application of first-order algorithms to Riemannian settings is the following:
\begin{center}
\emph{Is a performance gap necessary between Riemannian and Euclidean optimal convex-concave algorithms in terms of accuracy and the condition number?}
\end{center}
 
\subsection{Our Contributions} 
Our aim in this paper is to provide an extensive analysis of the Riemannian counterparts of Euclidean optimal first-order methods adapted to the manifold-constrained  setting. For the case of the smooth objectives, we consider the \emph{Riemannian corrected extragradient} (RCEG) method while for non-smooth cases, we analyze the textbook~\emph{Riemannian gradient descent ascent} (RGDA) method. Our main results are summarized in the following table. 
\begin{center}\small
\setlength{\arrayrulewidth}{0.1mm}
\setlength{\tabcolsep}{8pt}
\renewcommand{\arraystretch}{1}
Alg: \emph{RCEG}. Smooth setting with \emph{$\ell$-Lipschitz Gradient} (cf. Assumption~\ref{Assumption:basics},~\ref{Assumption:gscsc-smooth} and~\ref{Assumption:gcc-smooth}) \\
\begin{tabular}{>{\centering\arraybackslash}m{0.15\textwidth}
>{\centering\arraybackslash}m{0.15\textwidth}>{\centering\arraybackslash}m{0.2\textwidth}>{\centering\arraybackslash}m{0.38\textwidth}>{\centering\arraybackslash}m{0.05\textwidth}}
\toprule
\textbf{Perf. Measure} &  \textbf{Setting} &\textbf{Complexity} &\textbf{Theorem} \\
\midrule
Last-Iterate & Det. GSCSC & $O\left(\kappa(\sqrt{\tau_0} + \frac{1}{\underline{\xi}_0})\log (\frac{1}{\epsilon})\right)$ &\textbf{Thm.}~\ref{Thm:RCEG-SCSC} \\
Last-Iterate & Stoc. GSCSC & 
$O\left(\kappa(\sqrt{\tau_0} + \frac{1}{\underline{\xi}_0})\log(\frac{1}{\epsilon})+\frac{\sigma^2\overline{\xi}_0}{\mu^2\epsilon}\log(\frac{1}{\epsilon})\right)$ & \textbf{Thm.}~\ref{Thm:SRCEG-SCSC} \\
Avg-Iterate & Det.  GCC & $O\left(\frac{\ell\sqrt{\tau_0}}{\epsilon}\right)$ & \cite[\textbf{Thm.1}]{Zhang-2022-Minimax} \\
Avg-Iterate & Stoc. GCC & $O\left(\frac{\ell\sqrt{\tau_0}}{\epsilon} + \frac{\sigma^2\overline{\xi}_0}{\epsilon^2}\right)$ &\textbf{Thm.}~\ref{Thm:SRCEG-CC}\\
\end{tabular}
Alg: \emph{RGDA}. Nonsmooth setting with \emph{$L$-Lipschitz Function} (cf. Assumption~\ref{Assumption:gscsc-nonsmooth} and~\ref{Assumption:gcc-nonsmooth})
\begin{tabular}{>{\centering\arraybackslash}m{0.15\textwidth}
>{\centering\arraybackslash}m{0.15\textwidth}>{\centering\arraybackslash}m{0.2\textwidth}>{\centering\arraybackslash}m{0.38\textwidth}>{\centering\arraybackslash}m{0.05\textwidth}}
\toprule
Last-Iterate & Det. GSCSC & $O\left(\frac{L^2\overline{\xi}_0}{\mu^2\epsilon}\right)$ &\textbf{Thm.}~\ref{Thm:RGDA-SCSC}\\
Last-Iterate & Stoc. GSCSC & $O\left(\frac{(L^2+\sigma^2)\overline{\xi}_0}{\mu^2\epsilon}\right)$ &\textbf{Thm.}~\ref{Thm:SRGDA-SCSC}\\
Avg-Iterate & Det. GCC & $O\left(\frac{L^2\overline{\xi}_0}{\epsilon^2}\right)$ &\textbf{Thm.}~\ref{Thm:RGDA-CC}\\
Avg-Iterate & Stoc. GCC & $O\left(\frac{(L^2+\sigma^2)\overline{\xi}_0}{\epsilon^2}\right)$ &\textbf{Thm.}~\ref{Thm:SRGDA-CC}\\
\bottomrule
\end{tabular}
\end{center}
For the definition of the acronyms, Det and Stoc stand for deterministic and stochastic, respectively. GSCSC and GCC stand for geodesically strongly-convex-strongly-concave (cf. Assumption~\ref{Assumption:gscsc-smooth} or Assumption~\ref{Assumption:gscsc-nonsmooth}) and geodesically convex-concave (cf. Assumption~\ref{Assumption:gcc-smooth} or Assumption~\ref{Assumption:gcc-nonsmooth}). Here $\epsilon \in (0, 1)$ is the accuracy, $L,\ell$ the Lipschitzness of the objective and its gradient, $\kappa=\ell/\mu$ is the condition number of the function, where $\mu$ is the strong convexity parameter, $(\tau_0,\underline{\xi}_0, \overline{\xi}_0)$ are curvature parameters (cf. Assumption~\ref{Assumption:basics}), and $\sigma^2$ is the variance of a Riemannian gradient estimator. 

Our first main contribution is the derivation of a linear convergence rate for RCEG, answering the open conjecture of \cite{Zhang-2022-Minimax} about the performance gap of single-loop extragradient methods. Indeed, while a direct comparison between $d_\MCal^2(x_t, x^*)$ and $d_\MCal^2(x_{t-1}, x^*)$ is infeasible, we are able to establish a relationship between the iterates via appeal to the duality gap function and obtain a contraction in terms of $d_\MCal^2(x_t, x^*)$. In other words, the effect of Riemannian distance distortion is quantitative (the contraction ratio will depend on it) rather than qualitative (the geometric contraction still remains under a proper choice of constant stepsize). More specifically, we use $d_\MCal^2(x_t, x^\star) + d_\NCal^2(y_t, y^\star)$ and $d_\MCal^2(x_{t+1}, x^\star) + d_\NCal^2(y_{t+1}, y^\star)$ to bound a gap function defined by $f(\hat{x}_t, y^\star) - f(x^\star, \hat{y}_t)$. Since the objective function is geodesically strongly-convex-strongly-concave, we have $f(\hat{x}_t, y^\star) - f(x^\star, \hat{y}_t)$ is lower bounded by $\frac{\mu}{2}(d_\mathcal{M}(\hat{x}_t, x^\star)^2 + d_\mathcal{N}(\hat{y}_t, y^\star)^2)$. Then, using the relationship between $(x_t, y_t)$ and $(\hat{x}_t, \hat{y}_t)$, we conclude the desired results in Theorem~\ref{Thm:RCEG-SCSC}. Notably, our approach is not affected by the nonlinear geometry of the manifold.

Secondly, we endeavor to give a systematic analysis of aspects of the objective function, including its smoothness, its convexity and oracle access. As we shall see, similar to the Euclidean case, better finite-time convergence guarantees are connected with a geodesic smoothness condition. For the sake of completeness, in the paper's supplement we present the performance of Riemannian GDA for the full spectrum of stochasticity for the non-smooth case. More specifically, for the stochastic setting, the key ingredient to get the optimal convergence rate is to carefully select the step size such that the noise of the gradient estimator will not affect the final convergence rate significantly. As a highlight, such technique has been used for analyzing stochastic RCEG in the Euclidean setting~\citep{Kotsalis-2022-Simple} and our analysis can be seen as the extension to the Riemannian setting. For the nonsmooth setting, the analysis is relatively simpler compared to smooth settings but we still need to deal with the issue caused by the nonlinear geometry of manifolds and the interplay between the distortion of Riemannian metrics, the gap function and the bounds of Lipschitzness of our bi-objective. Interestingly, the rates we derive are near optimal in terms of accuracy and condition number of the objective, and analogous to their Euclidean counterparts.

\section{Preliminaries and Technical Background}\label{sec:prelim}
We present the basic setup and optimality conditions for Riemannian min-max optimization. Indeed, we focus on some of key concepts that we need from Riemannian geometry, deferring a fuller presentation, including motivating examples and further discussion of related work, to Appendix~\ref{sec:further-related}-\ref{app:MG}. 

\paragraph{Riemannian geometry.} An $n$-dimensional manifold $\MCal$ is a topological space where any point has a neighborhood that is homeomorphic to the $n$-dimensional Euclidean space. For each $x \in \MCal$, each tangent vector is tangent to all parametrized curves passing through $x$ and the tangent space $T_x \MCal$ of a manifold $\MCal$ at this point is defined as the set of all tangent vectors. A Riemannian manifold $\MCal$ is a smooth manifold that is endowed with a smooth (``Riemannian'') metric $\langle \cdot, \cdot\rangle_x$ on the tangent space $T_x \MCal$ for each point $x \in \MCal$. The inner metric induces a norm $\|\cdot\|_x$ on the tangent spaces.   

A geodesic can be seen as the generalization of an Euclidean linear segment and is modeled as a smooth curve (map),  $\gamma: [0, 1] \mapsto \MCal$, which is locally a distance minimizer.  Additionally, because of  the non-flatness of a manifold a different relation between the angles and the lengths of an arbitrary geodesic triangle is induced. This distortion can be quantified via the \emph{sectional curvature} parameter $\kappa_\MCal$ thanks to Toponogov's theorem~\citep{Cheeger-1975-Comparison, Burago-1992-AD}.
\begin{figure}[!h]
\centering
\includegraphics[width=0.5\textwidth]{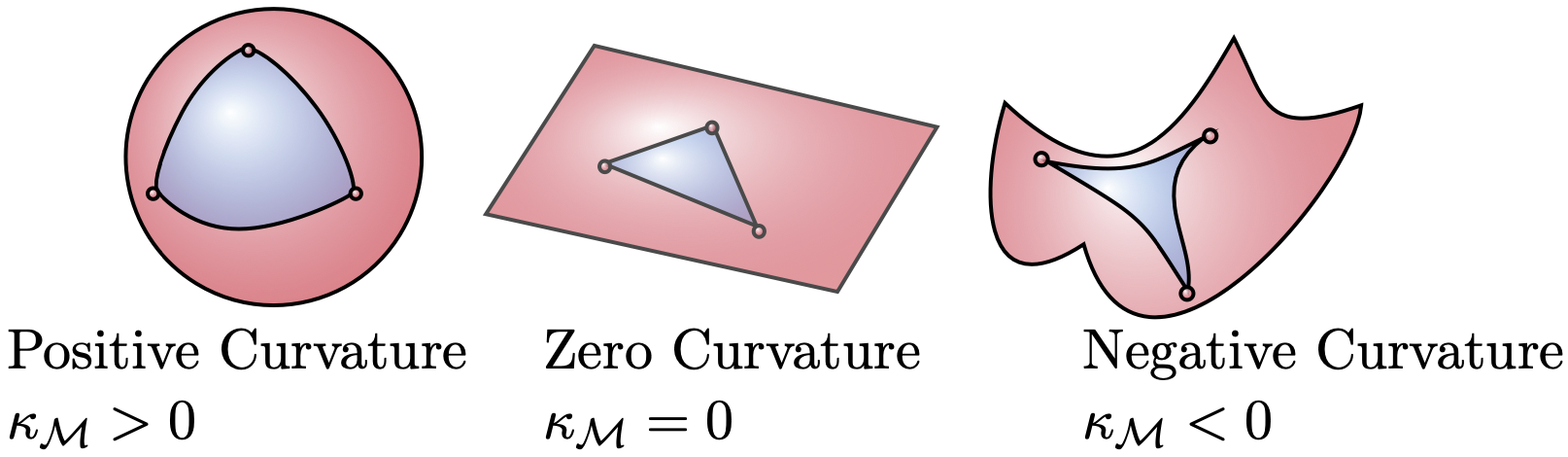}
\vspace{-0.5cm}
\end{figure}
A constructive consequence of this definition are the trigonometric comparison inequalities (TCIs) that will be essential in our proofs; see \citet[Corollary~2.1]{Alimisis-2020-Continuous} and~\citet[Lemma~5]{Zhang-2016-First} for detailed derivations. Assuming bounded sectional curvature, TCIs provide a tool for bounding Riemannian ``inner products'' that are more troublesome than classical Euclidean inner products.  

The following proposition summarizes the TCIs that we will need; note that if $\kappa_{\min} = \kappa_{\max} = 0$ (i.e., Euclidean spaces), then the proposition reduces to the law of cosines.
\begin{proposition}\label{Prop:key-inequality}
Suppose that $\MCal$ is a Riemannian manifold and let $\Delta$ be a geodesic triangle in $\MCal$ with the side length $a$, $b$, $c$ and let $A$ be the angle between $b$ and $c$. Then, we have
\begin{enumerate}[topsep=0pt,leftmargin=2ex]
\setlength{\itemsep}{0pt}
\setlength{\parskip}{.2ex}
\item  If  $\kappa_\MCal$ that is upper bounded by $\kappa_{\max} > 0$ and the diameter of $\MCal$ is bounded by $\frac{\pi}{\sqrt{\kappa_{\max}}}$, then 
\begin{equation*}
a^2 \geq \underline{\xi}(\kappa_{\max}, c) \cdot b^2 + c^2 - 2bc\cos(A), 
\end{equation*}
where $\underline{\xi}(\kappa, c) := 1$ for $\kappa \leq 0$ and $\underline{\xi}(\kappa, c) := c\sqrt{\kappa}\cot(c\sqrt{\kappa}) < 1$ for $\kappa > 0$.
\item  If  $\kappa_\MCal$ is lower bounded by $\kappa_{\min}$, then 
\begin{equation*}
a^2 \leq \overline{\xi}(\kappa_{\min}, c) \cdot b^2 + c^2 - 2bc\cos(A),
\end{equation*}
where $\overline{\xi}(\kappa, c) := c\sqrt{-\kappa}\coth(c\sqrt{-\kappa}) > 1$ if $\kappa < 0$ and $\overline{\xi}(\kappa, c) := 1$ if $\kappa \geq 0$. 
\end{enumerate}
\end{proposition}
\begin{wrapfigure}{r}{0.25\textwidth}
\vspace{-0.1em}
\includegraphics[width=0.25\textwidth]{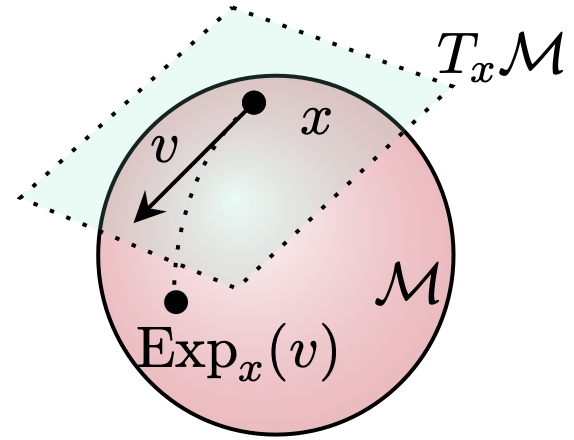}
\end{wrapfigure}
Also, in contrast to the Euclidean case, $x$ and $v=\grad_x f(x)$ do not lie in the same space, since $\MCal$ and $T_x \MCal$ respectively are distinct entities.
The interplay between these dual spaces typically is carried out via the \textit{exponential maps}.
An exponential map at a point $x \in \MCal$ is a mapping from the tangent space $T_x \MCal$ to $\MCal$. In particular, $y := \Exp_x(v) \in \MCal$ is defined such that there exists a geodesic $\gamma: [0, 1] \mapsto \MCal$ satisfying $\gamma(0) = x$, $\gamma(1) = y$ and $\gamma'(0) = v$. The inverse map exists since the manifold has a unique geodesic between any two points, which we denote as $\Exp_x^{-1}: \MCal \mapsto T_x \MCal$. Accordingly, we have $d_\MCal(x, y) = \|\Exp_x^{-1}(y)\|_x$ is the Riemannian distance induced by the exponential map. 

Finally, in contrast again to Euclidean spaces, we cannot compare the tangent vectors at different points $x, y \in \MCal$ since these vectors lie in different tangent spaces. To resolve this issue, it suffices to define a transport mapping that moves a tangent vector along the geodesics and also preserves the length and Riemannian metric $\langle \cdot, \cdot\rangle_x$; indeed, we can define a parallel transport $\Gamma_x^y: T_x \MCal \mapsto T_y \MCal$ such that the inner product between any $u, v \in T_x \MCal$ is preserved; i.e., $\langle u, v\rangle_x = \langle \Gamma_x^y(u), \Gamma_x^y(v)\rangle_y$. 

\paragraph{Riemannian min-max optimization and function classes.} We let $\MCal$ and $\NCal$ be Riemannian manifolds with unique geodesic and bounded sectional curvature and assume that the function $f: \MCal \times \NCal \mapsto \br$ is defined on the product of these manifolds. The regularity conditions that we impose on the function $f$ are as follows.   
\begin{definition}\label{def:Lip}
A function $f: \MCal \times \NCal \mapsto \br$ is \emph{geodesically $L$-Lipschitz} if for $\forall x, x' \in \MCal$ and $\forall y, y' \in \NCal$, the following statement holds true: $|f(x, y) - f (x', y')| \leq L(d_\MCal(x, x') + d_\NCal(y, y'))$. 
\label{def:smooth}
Additionally, if function $f$ is also differentiable, it is called \emph{geodesically $\ell$-smooth} if for $\forall x, x' \in \MCal$ and $\forall y, y' \in \NCal$, the following statement holds true, 
\begin{equation*}
\begin{array}{lll}
\|\grad_x f(x, y) - \Gamma_{x'}^x\grad_x f(x', y')\| & \leq & \ell(d_\MCal(x, x') + d_\NCal(y, y')), \\
\|\grad_y f(x, y) - \Gamma_{y'}^y\grad_y f(x', y')\| & \leq & \ell(d_\MCal(x, x') + d_\NCal(y, y')), 
\end{array}
\end{equation*}
where $(\grad_x f(x', y'), \grad_y f(x', y')) \in T_{x'} \MCal \times T_{y'} \NCal$ is the Riemannian gradient of $f$ at $(x', y')$, $\Gamma_{x'}^x$ is the parallel transport of $\MCal$ from $x'$ to $x$, and $\Gamma_{y'}^y$ is the parallel transport of $\NCal$ from $y'$ to $y$. 
\end{definition}
\begin{definition}\label{def:SCSC}
A function $f: \MCal \times \NCal \rightarrow \br$ is \emph{geodesically strongly-convex-strongly-concave} with the modulus $\mu > 0$ if the following statement holds true, 
\begin{equation*}
\begin{array}{llll}
f(x', y) & \geq & f(x, y) + \langle \subg_x f(x, y), \Exp_x^{-1}(x')\rangle_x + \tfrac{\mu}{2}(d_\MCal(x, x'))^2,  & \textnormal{ for each } y \in \NCal, \\
f(x, y') & \leq & f(x, y) + \langle \subg_y f(x, y), \Exp_y^{-1}(y')\rangle_y - \tfrac{\mu}{2}(d_\NCal(y, y'))^2,  & \textnormal{ for each } x \in \MCal.
\end{array}
\end{equation*}
where $(\subg_x f(x', y'), \subg_y f(x', y')) \in T_{x'} \MCal \times T_{y'} \NCal$ is a Riemannian subgradient of $f$ at a point $(x', y')$. A function $f$ is \emph{geodesically convex-concave} if the above holds true with $\mu = 0$. 
\end{definition}
Following standard conventions in Riemannian optimization~\citep{Zhang-2016-First, Alimisis-2020-Continuous, Zhang-2022-Minimax}, we make the following assumptions on the manifolds and objective functions:\footnote{In particular, our assumed upper and lower bounds $\kappa_{\min}, \kappa_{\max}$ guarantee that TCIs in Proposition~\ref{Prop:key-inequality} can be used in our analysis for proving finite-time convergence.}
\begin{assumption}\label{Assumption:basics} 
The objective function $f: \MCal \times \NCal \mapsto \br$ and manifolds $\MCal$ and $\NCal$ satisfy
\begin{enumerate} 
\item The diameter of the domain $\{(x, y) \in \MCal \times \NCal: -\infty < f(x, y) < +\infty\}$ is bounded by $D > 0$.
\item $\MCal,\NCal$ admit unique geodesic paths for any $(x, y), (x', y') \in \MCal \times \NCal$.
\item The sectional curvatures of $\MCal$ and $\NCal$ are both bounded in the range $[\kappa_{\min}, \kappa_{\max}]$ with $\kappa_{\min} \leq 0$. If $\kappa_{\max} > 0$, we assume that the diameter of manifolds is bounded by $\frac{\pi}{\sqrt{\kappa_{\max}}}$. 
\end{enumerate}
\end{assumption}
Under these conditions,~\citet{Zhang-2022-Minimax} proved an analog of Sion's minimax theorem~\citep{Sion-1958-General} in geodesic metric spaces. Formally, we have
\begin{equation*}
\max_{y \in \NCal} \min_{x \in \MCal} f(x, y) = \min_{x \in \MCal} \max_{y \in \NCal} f(x, y), 
\end{equation*}
which guarantees that there exists at least one global saddle point $(x^\star, y^\star) \in \MCal \times \NCal$ such that $\min_{x \in \MCal} f(x, y^\star) = f(x^\star, y^\star) = \max_{y \in \YCal} f(x^\star, y)$. Note that the unicity of geodesics assumption is algorithm-independent and is imposed for guaranteeing that a saddle-point solution always exist. Even though this rules out many manifolds of interest, there are still many manifolds that satisfy such conditions. More specifically, the Hadamard manifold (manifolds with non-positive curvature, $\kappa_{\max} = 0$) has a unique geodesic between any two points. This also becomes a common regularity condition in Riemannian optimization~\citep{Zhang-2016-First, Alimisis-2020-Continuous}. For any point $(\hat{\x}, \hat{\y}) \in \MCal \times \NCal$, the duality gap $f(\hat{x}, y^\star) - f(x^\star, \hat{y})$ thus gives an optimality criterion. 
\begin{definition}\label{def:opt-CC}
A point $(\hat{x}, \hat{y}) \in \MCal \times \NCal$ is an \emph{$\epsilon$-saddle point} of a geodesically convex-concave function $f(\cdot, \cdot)$ if $f(\hat{x}, y^\star) - f(x^\star, \hat{y}) \leq \epsilon$ where $(x^\star, y^\star) \in \MCal \times \NCal$ is a global saddle point.
\end{definition}
In the setting where $f$ is geodesically strongly-convex-strongly-concave with $\mu > 0$, it is not difficult to verify the uniqueness of a global saddle point $(x^\star, y^\star) \in \MCal \times \NCal$.  Then, we can consider the distance gap $(d(\hat{x}, x^\star))^2 + (d(\hat{y}, y^\star))^2$ as an optimality criterion for any point $(\hat{\x}, \hat{\y}) \in \MCal \times \NCal$. 
\begin{definition}\label{def:opt-SCSC}
A point $(\hat{x}, \hat{y}) \in \MCal \times \NCal$ is an $\epsilon$-saddle point of a geodesically strongly-convex-strongly-concave function $f(\cdot, \cdot)$ if $(d(\hat{x}, x^\star))^2 + (d(\hat{y}, y^\star))^2 \leq \epsilon$, where $(x^\star, y^\star) \in \MCal \times \NCal$ is a global saddle point. If $f$ is also geodesically $\ell$-smooth, we denote $\kappa = \frac{\ell}{\mu}$ as the condition number. 
\end{definition}
{ Given the above definitions, we can ask whether it is possible to find an $\epsilon$-saddle point efficiently or not. In this context,~\citet{Zhang-2022-Minimax} have answered this question in the affirmative for the setting where $f$ is geodesically $\ell$-smooth and geodesically convex-concave; indeed, they derive the convergence rate of Riemannian corrected extragradient (RCEG) method in terms of time-average iterates and also conjecture that \textit{RCEG does not guarantee convergence at a linear rate in terms of last iterates when $f$ is geodesically $\ell$-smooth and geodesically strongly-convex-strongly-concave, due to the existence of distance distortion}; see~\citet[Section~4.2]{Zhang-2022-Minimax}. Surprisingly, we show in Section~\ref{sec:RCEG} that RCEG with constant stepsize can achieve last-iterate convergence at a linear rate. Moreover, we establish the optimal convergence rates of stochastic RCEG for certain choices of stepsize for both geodesically convex-concave and geodesically strongly-convex-strongly-concave settings. }
\section{Riemannian Corrected Extragradient Method}\label{sec:RCEG}
In this section, we revisit the scheme of Riemannian corrected extragradient (RCEG) method proposed by~\citet{Zhang-2022-Minimax} and extend it to a stochastic algorithm that we refer to as \emph{stochastic RCEG}. We present our main results on an optimal last-iterate convergence guarantee for the geodesically strongly-convex-strongly-concave setting (both deterministic and stochastic) and a time-average convergence guarantee for the geodesically convex-concave setting (stochastic). This complements the time-average convergence guarantee for geodesically convex-concave setting (deterministic)~\citep[Theorem~4.1]{Zhang-2022-Minimax} and resolves an open problem posted in~\citet[Section~4.2]{Zhang-2022-Minimax}. 
\subsection{Algorithmic scheme}
\begin{table}[!t]
\begin{tabular}{cc}
\begin{minipage}{.53\textwidth}
\begin{algorithm}[H]\small
\caption{RCEG}\label{alg:RCEG}
\begin{algorithmic}
\State  \textbf{Input:} initial points $(x_0, y_0)$ and stepsizes $\eta > 0$. 
\For{$t = 0, 1, 2, \ldots, T-1$}
\State Query $(g_x^t, g_y^t)\gets(\grad_x f(x_t, y_t), \grad_y f(x_t, y_t))$, the Riemannian gradient of $f$ at a point $(x_t, y_t)$
\State $\hat{x}_t \leftarrow \Exp_{x_t}(-\eta \cdot g_x^t)$. 
\State $\hat{y}_t \leftarrow \Exp_{y_t}(\eta \cdot g_y^t)$. 
\State Query $(\hat{g}_x^t, \hat{g}_y^t)\gets(\grad_x f(\hat{x}_t, \hat{y}_t), \grad_y f(\hat{x}_t, \hat{y}_t))$, the Riemannian gradient of $f$ at a point $(\hat{x}_t, \hat{y}_t)$
\State $x_{t+1} \leftarrow \Exp_{\hat{x}_t}(-\eta \cdot \hat{g}_x^t + \Exp_{\hat{x}_t}^{-1}(x_t))$. 
\State $y_{t+1} \leftarrow \Exp_{\hat{y}_t}(\eta \cdot \hat{g}_y^t + \Exp_{\hat{y}_t}^{-1}(y_t))$.
\EndFor
\end{algorithmic}
\end{algorithm}
\end{minipage} &
\begin{minipage}{.47\textwidth}
\begin{algorithm}[H]\small
\caption{SRCEG}\label{alg:SRCEG}
\begin{algorithmic}
\State  \textbf{Input:} initial points $(x_0, y_0)$ and stepsizes $\eta > 0$. 
\For{$t = 0, 1, 2, \ldots, T-1$}
\State Query $(g_x^t, g_y^t)$ as a \textbf{noisy} estimator of Riemannian gradient of $f$ at a point $(x_t, y_t)$. 
\State $\hat{x}_t \leftarrow \Exp_{x_t}(-\eta \cdot g_x^t)$. 
\State $\hat{y}_t \leftarrow \Exp_{y_t}(\eta \cdot g_y^t)$. 
\State Query $(\hat{g}_x^t, \hat{g}_y^t)$ as a \textbf{noisy} estimator of Riemannian gradient of $f$ at a point $(\hat{x}_t,\hat{y}_t)$.  
\State $x_{t+1} \leftarrow \Exp_{\hat{x}_t}(-\eta \cdot \hat{g}_x^t + \Exp_{\hat{x}_t}^{-1}(x_t))$. 
\State $y_{t+1} \leftarrow \Exp_{\hat{y}_t}(\eta \cdot \hat{g}_y^t + \Exp_{\hat{y}_t}^{-1}(y_t))$.
\EndFor
\end{algorithmic}
\end{algorithm}
\end{minipage}
\end{tabular}
\end{table}
The recently proposed \emph{Riemannian corrected extragradient} (RCEG) method~\citep{Zhang-2022-Minimax} is a natural extension of the celebrated extragradient (EG) method to the Riemannian setting. Its scheme resembles that of EG in Euclidean spaces but employs a simple modification in the extrapolation step to accommodate the nonlinear geometry of Riemannian manifolds. Let us provide some intuition how such modifications work.

We start with a basic version of EG as follows, where $\MCal$ and $\NCal$ are classically restricted to be convex constraint sets in Euclidean spaces: 
\begin{equation}\label{def:EG}
\begin{array}{rclcrcl}
\hat{x}_t & \leftarrow & \proj_\MCal(x_t - \eta \cdot \nabla_x f(x_t, y_t)), & & \hat{y}_t & \leftarrow & \proj_\NCal(y_t + \eta \cdot \nabla_y f(x_t, y_t)), \\
x_{t+1} & \leftarrow & \proj_\MCal(x_t - \eta \cdot \nabla_x f(\hat{x}_t, \hat{y}_t)), & & y_{t+1} & \leftarrow & \proj_\NCal(y_t + \eta \cdot \nabla_y f(\hat{x}_t, \hat{y}_t)).  
\end{array}
\end{equation}
\begin{wrapfigure}{r}{0.55\textwidth}
\tikzset{every picture/.style={line width=0.75pt}} 

\begin{tikzpicture}[x=0.75pt,y=0.75pt,yscale=-0.6,xscale=0.6]

\draw  [fill={rgb, 255:red, 74; green, 144; blue, 226 }  ,fill opacity=0.31 ] (356.23,24.89) .. controls (417.67,13.14) and (583.36,102.34) .. (600,199.75) .. controls (571.84,144.33) and (439.28,159.83) .. (397.04,274.04) .. controls (341.99,154.79) and (206.46,140.78) .. (169.33,196.2) .. controls (134.77,61.84) and (311.42,26.57) .. (356.23,24.89) -- cycle ;
\draw [color={rgb, 255:red, 208; green, 2; blue, 27 }  ,draw opacity=1 ][fill={rgb, 255:red, 248; green, 231; blue, 28 }  ,fill opacity=1 ][line width=1.5]    (302,94) -- (389.08,112.18) ;
\draw [shift={(393,113)}, rotate = 191.79] [fill={rgb, 255:red, 208; green, 2; blue, 27 }  ,fill opacity=1 ][line width=0.08]  [draw opacity=0] (11.61,-5.58) -- (0,0) -- (11.61,5.58) -- cycle    ;
\draw [color={rgb, 255:red, 208; green, 2; blue, 27 }  ,draw opacity=1 ][line width=3]    (302,93.02) .. controls (341,103) and (363,102) .. (386,139) ;
\draw [shift={(386,139)}, rotate = 58.13] [color={rgb, 255:red, 208; green, 2; blue, 27 }  ,draw opacity=1 ][fill={rgb, 255:red, 208; green, 2; blue, 27 }  ,fill opacity=1 ][line width=3]      (0, 0) circle [x radius= 2.55, y radius= 2.55]   ;
\draw [shift={(302,93.02)}, rotate = 14.35] [color={rgb, 255:red, 208; green, 2; blue, 27 }  ,draw opacity=1 ][fill={rgb, 255:red, 208; green, 2; blue, 27 }  ,fill opacity=1 ][line width=3]      (0, 0) circle [x radius= 2.55, y radius= 2.55]   ;
\draw [color={rgb, 255:red, 144; green, 19; blue, 254 }  ,draw opacity=1 ][fill={rgb, 255:red, 248; green, 231; blue, 28 }  ,fill opacity=1 ][line width=1.5]    (386,140) -- (435.01,143.7) ;
\draw [shift={(439,144)}, rotate = 184.32] [fill={rgb, 255:red, 144; green, 19; blue, 254 }  ,fill opacity=1 ][line width=0.08]  [draw opacity=0] (9.29,-4.46) -- (0,0) -- (9.29,4.46) -- cycle    ;
\draw [color={rgb, 255:red, 144; green, 19; blue, 254 }  ,draw opacity=1 ][line width=1.5]  [dash pattern={on 1.69pt off 2.76pt}]  (302.08,89.8) -- (326.92,62.2) ;
\draw [shift={(328,61)}, rotate = 311.99] [color={rgb, 255:red, 144; green, 19; blue, 254 }  ,draw opacity=1 ][line width=1.5]      (0, 0) circle [x radius= 2.61, y radius= 2.61]   ;
\draw [shift={(301,91)}, rotate = 311.99] [color={rgb, 255:red, 144; green, 19; blue, 254 }  ,draw opacity=1 ][line width=1.5]      (0, 0) circle [x radius= 2.61, y radius= 2.61]   ;
\draw [color={rgb, 255:red, 144; green, 19; blue, 254 }  ,draw opacity=1 ][line width=1.5]    (387,137) -- (330.43,63.18) ;
\draw [shift={(328,60)}, rotate = 52.54] [fill={rgb, 255:red, 144; green, 19; blue, 254 }  ,fill opacity=1 ][line width=0.08]  [draw opacity=0] (9.29,-4.46) -- (0,0) -- (9.29,4.46) -- cycle    ;
\draw [color={rgb, 255:red, 144; green, 19; blue, 254 }  ,draw opacity=1 ][line width=1.5]    (386,140) -- (366.75,72.14) ;
\draw [shift={(365.65,68.29)}, rotate = 74.16] [fill={rgb, 255:red, 144; green, 19; blue, 254 }  ,fill opacity=1 ][line width=0.08]  [draw opacity=0] (9.29,-4.46) -- (0,0) -- (9.29,4.46) -- cycle    ;
\draw    (133,4) ;
\draw    (301,92) -- (301,93) ;
\draw [shift={(301,93)}, rotate = 90] [color={rgb, 255:red, 0; green, 0; blue, 0 }  ][fill={rgb, 255:red, 0; green, 0; blue, 0 }  ][line width=0.75]      (0, 0) circle [x radius= 3.69, y radius= 3.69]   ;
\draw [color={rgb, 255:red, 144; green, 19; blue, 254 }  ,draw opacity=1 ][line width=1.5]    (347,71) .. controls (374,55) and (374,103) .. (386,140) ;
\draw [shift={(347,71)}, rotate = 329.35] [color={rgb, 255:red, 144; green, 19; blue, 254 }  ,draw opacity=1 ][fill={rgb, 255:red, 144; green, 19; blue, 254 }  ,fill opacity=1 ][line width=1.5]      (0, 0) circle [x radius= 3.05, y radius= 3.05]   ;
\draw    (347,71) ;
\draw [shift={(347,71)}, rotate = 0] [color={rgb, 255:red, 0; green, 0; blue, 0 }  ][fill={rgb, 255:red, 0; green, 0; blue, 0 }  ][line width=0.75]      (0, 0) circle [x radius= 3.69, y radius= 3.69]   ;
\draw    (385,140) -- (386,140) ;
\draw [shift={(386,140)}, rotate = 0] [color={rgb, 255:red, 0; green, 0; blue, 0 }  ][fill={rgb, 255:red, 0; green, 0; blue, 0 }  ][line width=0.75]      (0, 0) circle [x radius= 3.69, y radius= 3.69]   ;
\draw  [draw opacity=0] (326.19,61.77) -- (376,70) -- (435.33,143.39) -- (385.52,135.16) -- cycle ; \draw   (326.19,61.77) -- (385.52,135.16)(345.92,65.03) -- (405.25,138.42)(365.65,68.29) -- (424.98,141.68) ; \draw   (326.19,61.77) -- (376,70)(345.1,85.17) -- (394.91,93.39)(364.01,108.56) -- (413.82,116.79)(382.92,131.96) -- (432.73,140.18) ; \draw    ;
\draw  [fill={rgb, 255:red, 144; green, 19; blue, 254 }  ,fill opacity=0.16 ] (411.33,20) -- (605.74,129.13) -- (386.28,214) -- (188.68,103.65) -- cycle ;
\draw  [fill={rgb, 255:red, 208; green, 2; blue, 27 }  ,fill opacity=0.28 ] (335.96,17) -- (611,100.61) -- (449.93,157.1) -- (398.87,175) -- (137,92.15) -- cycle ;

\draw (171.96,151.45) node [anchor=north west][inner sep=0.75pt]  [rotate=-359.68]  {$\mathcal{M}$};
\draw (271.96,81.45) node [anchor=north west][inner sep=0.75pt]  [font=\footnotesize,rotate=-359.68]  {$x_{t}$};
\draw (236.4,96.71) node [anchor=north west][inner sep=0.75pt]  [font=\footnotesize,rotate=-15.21,xslant=-0.05]  {$-\eta \ \mathrm{grad} f( x_{t} ,y_{t})$};
\draw (238.96,46.45) node [anchor=north west][inner sep=0.75pt]  [font=\scriptsize,rotate=-359.68]  {$\mathrm{Exp}_{\hat{x_{t}}}^{-1}( x_{t})$};
\draw (338.96,28.45) node [anchor=north west][inner sep=0.75pt]  [font=\footnotesize,rotate=-359.68]  {$x_{t+1}$};
\draw (424.65,171.82) node [anchor=north west][inner sep=0.75pt]  [font=\footnotesize,rotate=-339.21,xslant=-0.04]  {$-\eta \ \mathrm{grad} f\left(\widehat{x_{t}} ,\widehat{y_{t}}\right)$};
\draw (382.89,152.39) node [anchor=north west][inner sep=0.75pt]  [font=\footnotesize,rotate=-359.68]  {$\hat{x}_{t}$};

\end{tikzpicture}
\vspace{-0.5cm}
\end{wrapfigure}
Turning to the setting where $\MCal$ and $\NCal$ are Riemannian manifolds, the rather straightforward way to do the generalization is to replace the projection operator by the corresponding exponential map and the gradient by the corresponding Riemannian gradient. For the first line of Eq.~\eqref{def:EG}, this approach works and leads to the following updates: 
\begin{equation*}
\hat{x}_t \leftarrow \Exp_{x_t}(-\eta \cdot \grad_x f(x_t, y_t)), \quad \hat{y}_t \leftarrow \Exp_{y_t}(\eta \cdot \grad_y f(x_t, y_t)). 
\end{equation*}
However, we encounter some issues for the second line of Eq.~\eqref{def:EG}: The aforementioned approach leads to some problematic updates, $x_{t+1} \leftarrow \Exp_{x_t}(-\eta \cdot \grad_x f(\hat{x}_t, \hat{y}_t))$ and $y_{t+1} \leftarrow \Exp_{y_t}(\eta \cdot \grad_y f(\hat{x}_t, \hat{y}_t))$; indeed, the exponential maps $\Exp_{x_t}(\cdot)$ and $\Exp_{y_t}(\cdot)$ are defined from $T_{x_t} \MCal$ to $\MCal$ and from $T_{y_t} \NCal$ to $\NCal$ respectively. However, we have $-\grad_x f(\hat{x}_t, \hat{y}_t) \in T_{\hat{x}_t}\MCal$ and $\grad_y f(\hat{x}_t, \hat{y}_t) \in T_{\hat{y}_t}\NCal$. This motivates us to reformulate the second line of Eq.~\eqref{def:EG} as follows:
\begin{equation*}
x_{t+1} \leftarrow \proj_\MCal(\hat{x}_t - \eta \cdot \nabla_x f(\hat{x}_t, \hat{y}_t) + (x_t - \hat{x}_t)), \quad y_{t+1} \leftarrow \proj_\NCal(\hat{y}_t + \eta \cdot \nabla_y f(\hat{x}_t, \hat{y}_t) + (y_t - \hat{y}_t)).  
\end{equation*}
In the general setting of Riemannian manifolds, the terms $x_t - \hat{x}_t$ and $y_t - \hat{y}_t$ become $\Exp_{\hat{x}_t}^{-1}(x_t) \in T_{\hat{x}_t} \MCal$ and $\Exp_{\hat{y}_t}^{-1}(y_t) \in T_{\hat{y}_t} \NCal$. This observation yields the following updates: 
\begin{equation*}
x_{t+1} \leftarrow \Exp_{\hat{x}_t}(-\eta \cdot \grad_x f(\hat{x}_t, \hat{y}_t) + \Exp_{\hat{x}_t}^{-1}(x_t)), \quad \hat{y}_t \leftarrow \Exp_{\hat{y}_t}(\eta \cdot \grad_y f(\hat{x}_t, \hat{y}_t) + \Exp_{\hat{y}_t}^{-1}(y_t)). 
\end{equation*} 
We summarize the resulting RCEG method in Algorithm~\ref{alg:RCEG} and present the stochastic extension with noisy estimators of Riemannian gradients of $f$ in Algorithm~\ref{alg:SRCEG}. 

\subsection{Main results}
We present our main results on global convergence for Algorithms~\ref{alg:RCEG} and~\ref{alg:SRCEG}. 
To simplify the presentation, we treat separately the following two cases:
\begin{assumption}\label{Assumption:gscsc-smooth} 
The objective function $f$ is geodesically $\ell$-smooth and geodesically strongly-convex-strongly-concave with $\mu > 0$. 
\end{assumption}
\vspace{-0.3cm}
\begin{assumption}\label{Assumption:gcc-smooth} 
The objective function $f$ is geodesically $\ell$-smooth and geodesically convex-concave. 
\end{assumption}
Letting $(x^\star, y^\star) \in \MCal \times \NCal$ be a global saddle point of $f$ (which exists under either Assumption~\ref{Assumption:gscsc-smooth} or~\ref{Assumption:gcc-smooth}), we let $D_0 = (d_\MCal(x_0, x^\star))^2 + (d_\NCal(y_0, y^\star))^2 > 0$ and $\kappa = \ell/\mu$ for geodesically strongly-convex-strongly-concave setting. For simplicity of presentation, we also define a ratio $\tau(\cdot, \cdot)$ that measures how non-flatness changes in the spaces: $
\tau([\kappa_{\min}, \kappa_{\max}], c) = \tfrac{\overline{\xi}(\kappa_{\min}, c)}{\underline{\xi}(\kappa_{\max}, c)} \geq 1. 
$
 We summarize our results for Algorithm~\ref{alg:RCEG} in the following theorem. 
\begin{theorem}\label{Thm:RCEG-SCSC}
Given Assumptions~\ref{Assumption:basics} and \ref{Assumption:gscsc-smooth}, and letting $\eta = \min\{1/(2\ell\sqrt{\tau_0}), \underline{\xi}_0/(2\mu)\}$, there exists some $T > 0$ such that the output of Algorithm~\ref{alg:RCEG} satisfies that $(d(x_T, x^\star))^2 + (d(y_T, y^\star))^2 \leq \epsilon$ (i.e., an $\epsilon$-saddle point of $f$ in Definition~\ref{def:opt-SCSC}) and the total number of Riemannian gradient evaluations is bounded by
\begin{equation*}
O\left(\left(\kappa\sqrt{\tau_0} + \frac{1}{\underline{\xi}_0}\right)\log\left(\frac{D_0}{\epsilon}\right)\right), 
\end{equation*}
where $\tau_0 = \tau([\kappa_{\min}, \kappa_{\max}], D) \geq 1$ measures how non-flatness changes in $\MCal$ and $\NCal$ and $\underline{\xi}_0 = \underline{\xi}(\kappa_{\max}, D) \leq 1$ is properly defined in Proposition~\ref{Prop:key-inequality}.  
\end{theorem}
\begin{remark}
Theorem~\ref{Thm:RCEG-SCSC} illustrates the last-iterate convergence of Algorithm~\ref{alg:RCEG} for solving geodesically strongly-convex-strongly-concave problems, thereby resolving an open problem delineated by~\citet{Zhang-2022-Minimax}. Further, the dependence on $\kappa$ and $1/\epsilon$ cannot be improved since it matches the lower bound established for min-max optimization problems in Euclidean spaces~\citep{Zhang-2021-Lower}. However, we believe that the dependence on $\tau_0$ and $\underline{\xi}_0$ is not tight, and it is of interest to either improve the rate or establish a lower bound for general Riemannian min-max optimization.
\end{remark}
\begin{remark}
The current theoretical analysis covers local geodesic strong-convex-strong-concave settings. The key ingredient is how to define the local region; indeed, if we say the set of $\{(x, y): d_\MCal(x, x^\star) \leq \delta, d_\NCal(y_t, y^\star) \leq \delta\}$ is a local region where the function is geodesic strong-convex-strong-concave. Then, the set of $\{(x, y): (d_\MCal(x, x^\star)^2 + d_\NCal(y_t, y^\star)^2) \leq \delta^2\}$ must be contained in the above local region and the objective function is also geodesic strong-convex-strong-concave. If $(x_0, y_0) \in \{(x, y): (d_\MCal(x, x^\star)^2 + d_\NCal(y_t, y^\star)^2) \leq \delta^2\}$, our theoretical analysis guarantees the last-iterate linear convergence rate. Such argument and definition of local region were standard for min-max optimization in the Euclidean setting; see~\citet[Assumption 2.1]{Liang-2019-Interaction}. For an important optimization problem that is globally geodesically strongly-convex-strongly-concave, we refer to Appendix~\ref{app:examples} where \textit{Robust matrix Karcher mean problem} is indeed the desired one.
\end{remark}
In the scheme of SRECG, we highlight that $(g_x^t, g_y^t)$ and $(\hat{g}_x^t, \hat{g}_y^t)$ are noisy estimators of Riemannian gradients of $f$ at $(x_t, y_t)$ and $(\hat{x}_t, \hat{y}_t)$. It is necessary to impose the conditions such that these estimators are unbiased and has bounded variance. By abuse of notation, we assume that 
\begin{equation}\label{def:noisy-model}
\begin{array}{lcl}
g_x^t = \grad_x f(x_t, y_t) + \xi_x^t, & & g_y^t = \grad_y f(x_t, y_t) + \xi_y^t,  \\
\hat{g}_x^t = \grad_x f(\hat{x}_t, \hat{y}_t) + \hat{\xi}_x^t, & & \hat{g}_y^t = \grad_y f(\hat{x}_t, \hat{y}_t) + \hat{\xi}_y^t. 
\end{array}
\end{equation}
where the noises $(\xi_x^t, \xi_y^t)$ and $(\hat{\xi}_x^t, \hat{\xi}_y^t)$ are independent and satisfy that 
\begin{equation}\label{def:noisy-bound}
\begin{array}{lclcl}
\EE[\xi_x^t] = 0, & & \EE[\xi_y^t] = 0, & & \EE[\|\xi_x^t\|^2 + \|\xi_y^t\|^2] \leq \sigma^2, \\
\EE[\hat{\xi}_x^t] = 0, & & \EE[\hat{\xi}_y^t] = 0, & & \EE[\|\hat{\xi}_x^t\|^2 + \|\hat{\xi}_y^t\|^2] \leq \sigma^2. 
\end{array}
\end{equation}
We are ready to summarize our results for Algorithm~\ref{alg:SRCEG} in the following theorems. 
\begin{theorem}\label{Thm:SRCEG-SCSC}
Given Assumptions~\ref{Assumption:basics} and \ref{Assumption:gscsc-smooth}, letting Eq.~\eqref{def:noisy-model} and Eq.~\eqref{def:noisy-bound} hold with $\sigma > 0$ and letting $\eta > 0$ satisfy $\eta = \min\{\frac{1}{24\ell\sqrt{\tau_0}}, \frac{\underline{\xi}_0}{2\mu}, \tfrac{2(\log(T) + \log(\mu^2 D_0 \sigma^{-2}))}{\mu T}\}$, there exists some $T > 0$ such that the output of Algorithm~\ref{alg:SRCEG} satisfies that $\EE[(d(x_T, x^\star))^2 + (d(y_T, y^\star))^2] \leq \epsilon$ and the total number of noisy Riemannian gradient evaluations is bounded by
\begin{equation*}
O\left(\left(\kappa\sqrt{\tau_0} + \frac{1}{\underline{\xi}_0}\right)\log\left(\frac{D_0}{\epsilon}\right) + \frac{\sigma^2\overline{\xi}_0}{\mu^2\epsilon}\log\left(\frac{1}{\epsilon}\right)\right), 
\end{equation*}
where $\tau_0 = \tau([\kappa_{\min}, \kappa_{\max}], D) \geq 1$ measures how non-flatness changes in $\MCal$ and $\NCal$ and $\underline{\xi}_0 = \underline{\xi}(\kappa_{\max}, D) \leq 1$ is properly defined in Proposition~\ref{Prop:key-inequality}.  
\end{theorem}
\begin{theorem}\label{Thm:SRCEG-CC}
Given Assumptions~\ref{Assumption:basics} and \ref{Assumption:gcc-smooth} and assume that Eq.~\eqref{def:noisy-model} and Eq.~\eqref{def:noisy-bound} hold with $\sigma > 0$ and let $\eta > 0$ satisfies that $\eta = \min\{\frac{1}{4\ell\sqrt{\tau_0}}, \tfrac{1}{\sigma}\sqrt{\tfrac{D_0}{\overline{\xi}_0 T}}\}$, there exists some $T > 0$ such that the output of Algorithm~\ref{alg:SRCEG} satisfies that $\EE[f(\bar{x}_T, y^\star) - f(x^\star, \bar{y}_T)] \leq \epsilon$ and the total number of noisy Riemannian gradient evaluations is bounded by
\begin{equation*}
O\left(\frac{\ell D_0\sqrt{\tau_0}}{\epsilon} + \frac{\sigma^2\overline{\xi}_0}{\epsilon^2}\right), 
\end{equation*}
where $\tau_0 = \tau([\kappa_{\min}, \kappa_{\max}], D)$ measures how non-flatness changes in $\MCal$ and $\NCal$ and $\overline{\xi}_0 = \overline{\xi}(\kappa_{\min}, D) \geq 1$ is properly defined in Proposition~\ref{Prop:key-inequality}. The time-average iterates $(\bar{x}_T, \bar{y}_T) \in \MCal \times \NCal$ can be computed by $(\bar{x}_0, \bar{y}_0) = (0, 0)$ and the inductive formula: $\bar{x}_{t+1} = \Exp_{\bar{x}_t}(\tfrac{1}{t+1} \cdot \Exp_{\bar{x}_t}^{-1}(\hat{x}_t))$ and $\bar{y}_{t+1} = \Exp_{\bar{y}_t}(\tfrac{1}{t+1} \cdot \Exp_{\bar{y}_t}^{-1}(\hat{y}_t))$ for all $t = 0, 1, \ldots, T-1$. 
\end{theorem}
\begin{remark}
Theorem~\ref{Thm:SRCEG-SCSC} presents the last-iterate convergence rate of Algorithm~\ref{alg:SRCEG} for solving geodesically strongly-convex-strongly-concave problems while Theorem~\ref{Thm:SRCEG-CC} gives the time-average convergence rate when the function $f$ is only assumed to be geodesically convex-concave. Note that we carefully choose the stepsizes such that our upper bounds match the lower bounds established for stochastic min-max optimization problems in Euclidean spaces~\citep{Juditsky-2011-Solving, Fallah-2020-Optimal, Kotsalis-2022-Simple}, in terms of the dependence on $\kappa$, $1/\epsilon$ and $\sigma^2$, up to log factors.  
\end{remark}
\paragraph{Discussions:} The last-iterate linear convergence rate in terms of Riemannian metrics is only limited to geodesically strongly convex-concave cases but other results, e.g., the average-iterate sublinear convergence rate, are derived under more mild conditions. This is consistent with classical results in the Euclidean setting where geodesic convexity reduces to convexity; indeed, the last-iterate linear convergence rate in terms of squared Euclidean norm is only obtained for strongly convex-concave cases. As such, our setting is not restrictive. Moreover, ~\citet{Zhang-2022-Minimax} showed that the existence of a global saddle point is only guaranteed under the geodesically convex-concave assumption. For geodesically nonconvex-concave or geodesically nonconvex-nonconcave cases, a global saddle point might not exist and new optimality notions are required before algorithmic design. This question remains open in the Euclidean setting and is beyond the scope of this paper. However, we remark that an interesting class of robustification problems are nonconvex-nonconcave min-max problems in the Euclidean setting can be geodesically convex-concave in the Riemannian setting; see Appendix~\ref{app:examples}.
\section{Experiments}\label{sec:exp}
We present numerical experiments on the task of robust principal component analysis (RPCA) for symmetric positive definite (SPD) matrices. In particular, we compare the performance of Algorithm~\ref{alg:RCEG} and~\ref{alg:SRCEG} with different outputs, i.e., the last iterate $(x_T, y_T)$ versus the time-average iterate $(\bar{x}_T, \bar{y}_T)$ (see the precise definition in Theorem~\ref{Thm:SRCEG-CC}). Note that our implementations of both algorithms are based on the \textsc{manopt} package~\citep{Boumal-2014-Manopt}. All the experiments were implemented in MATLAB R2021b on a workstation with a 2.6 GHz Intel Core i7 and 16GB of memory. Due to space limitations, some additional experimental results are deferred to Appendix~\ref{sec:appendix-exp}. 

\paragraph{Experimental setup.} The problem of RPCA~\citep{Candes-2011-Robust, Harandi-2017-Dimensionality} can be formulated as the Riemannian min-max optimization problem with an SPD manifold and a sphere manifold. Formally, we have
\begin{equation}\label{prob:RPCA}
\max_{M \in \MCal_{\textnormal{PSD}}^d} \min_{x \in \SCal^d}\left\{ -x^{\top}Mx - \frac{\alpha}{n}\sum_{i=1}^{n}d(M, M_i)\right\}. 
\end{equation}
In this formulation, $\alpha > 0$ denotes the penalty parameter, $\{M_i\}_{i \in [n]}$ is a sequence of given data SPD matrices, $\MCal_{\textnormal{PSD}}^d = \{M \in \br^{d \times d}: M \succ 0, M = M^\top\}$ denotes the SPD manifold, $\SCal^d = \{x \in \br^d: \|x\| = 1\}$ denotes the sphere manifold and $d(\cdot, \cdot): \MCal_{\textnormal{PSD}}^d \times \MCal_{\textnormal{PSD}}^d \mapsto \br$ is the Riemannian distance induced by the exponential map on the SPD manifold $\MCal_{\textnormal{PSD}}^d$. As demonstrated by~\citet{Zhang-2022-Minimax}, the problem of RPCA is nonconvex-nonconcave from a Euclidean perspective but is \textit{locally geodesically strongly-convex-strongly-concave} and satisfies most of the assumptions that we make in this paper. In particular, the SPD manifold is complete with sectional curvature in $[-\frac{1}{2}, 1]$~\citep{Criscitiello-2022-Accelerated} and the sphere manifold is complete with sectional curvature of $1$. Other reasons why we use such example are: (i) it is a classical one in ML; (ii)~\citet{Zhang-2022-Minimax} also uses this example and observes the linear convergence behavior; (iii) the numerical results show that the unicity of geodesics assumption may not be necessary in practice; and (iv) this is an application where both min and max sides are done on Riemannian manifolds.

Following the previous works of~\citet{Zhang-2022-Minimax} and~\citet{Han-2022-Riemannian}, we generate a sequence of data matrices $M_i$ satisfying that their eigenvalues are in the range of $[0.2,  4.5]$. In our experiment, we fix $\alpha=1.0$ and also vary the problem dimension $d \in \{25, 50, 100\}$. The evaluation metric is set as gradient norm. We set $n=40$ and $n=200$ in Figure~\ref{fig:exp-deterministic} and~\ref{fig:exp-stoc}. For RCEG, we set $\eta = \frac{1}{2\ell}$ where $\ell > 0$ is selected via grid search. For SRCEG, we set $\eta_t = \min\{\frac{1}{2\ell}, \frac{a}{t}\}$ where $\ell, a > 0$ are selected via grid search. Additional results on the effect of stepsize are summarized in Appendix~\ref{sec:appendix-exp}. 

\paragraph{Experimental results.} Figure~\ref{fig:exp-deterministic} summarizes the effects of different outputs for RCEG; indeed, RCEG-last and RCEG-avg refer to Algorithm~\ref{alg:RCEG} with last iterate and time-average iterate respectively. It is clear that the last iterate of RCEG consistently exhibits linear convergence to an optimal solution in all the settings, verifying our theoretical results in Theorem~\ref{Thm:RCEG-SCSC}. In contrast, the average iterate of RCEG converges much slower than the last iterate of RCEG. The possible reason is that the problem of RPCA is \textit{only} locally geodesically strongly-convex-strongly-concave and averaging with the iterates generated during early stage will significantly slow down the convergence of RCEG. 

Figure~\ref{fig:exp-stoc} presents the comparison between SRCEG (with either last iterate or time-average iterate) and RCEG with last-iterate; here, SRCEG-last and SRCEG-avg refer to Algorithm~\ref{alg:SRCEG} with last iterate and time-average iterate respectively. We observe that SRCEG with either last iterate or average iterate converge faster than RCEG at the early stage and all of them finally converge to an optimal solution. This demonstrates the effectiveness and efficiency of SRCEG in practice. It is also worth mentioning that the difference between last-iterate convergence and time-average-iterate convergence is not as significant as in the deterministic setting. This is possibly because the technique of averaging help cancels the negative effect of imperfect information~\citep{Kingma-2014-Adam, Yazici-2018-Unusual}. 
\begin{figure*}[!t]
\centering
\includegraphics[width=0.32\textwidth]{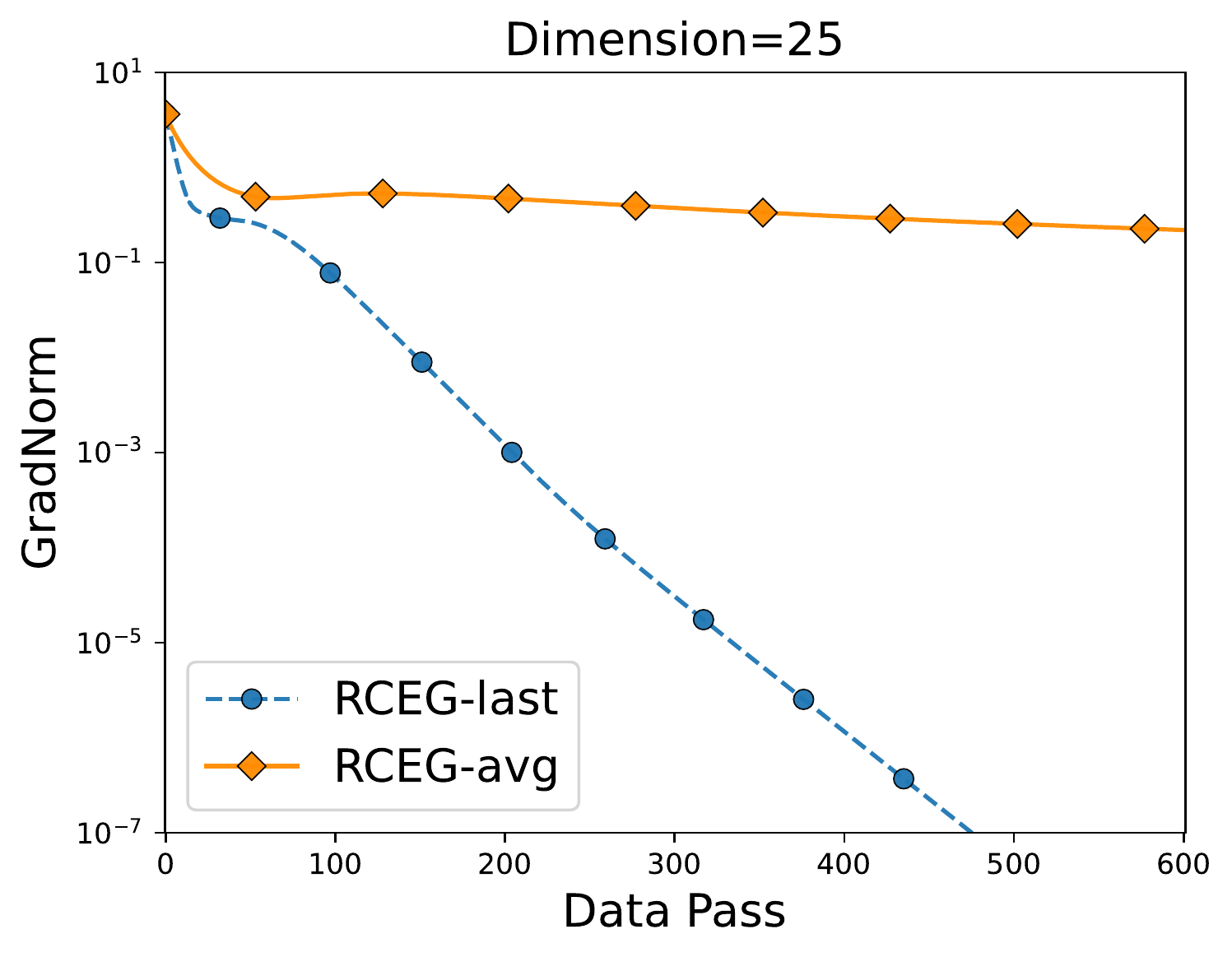}
\includegraphics[width=0.32\textwidth]{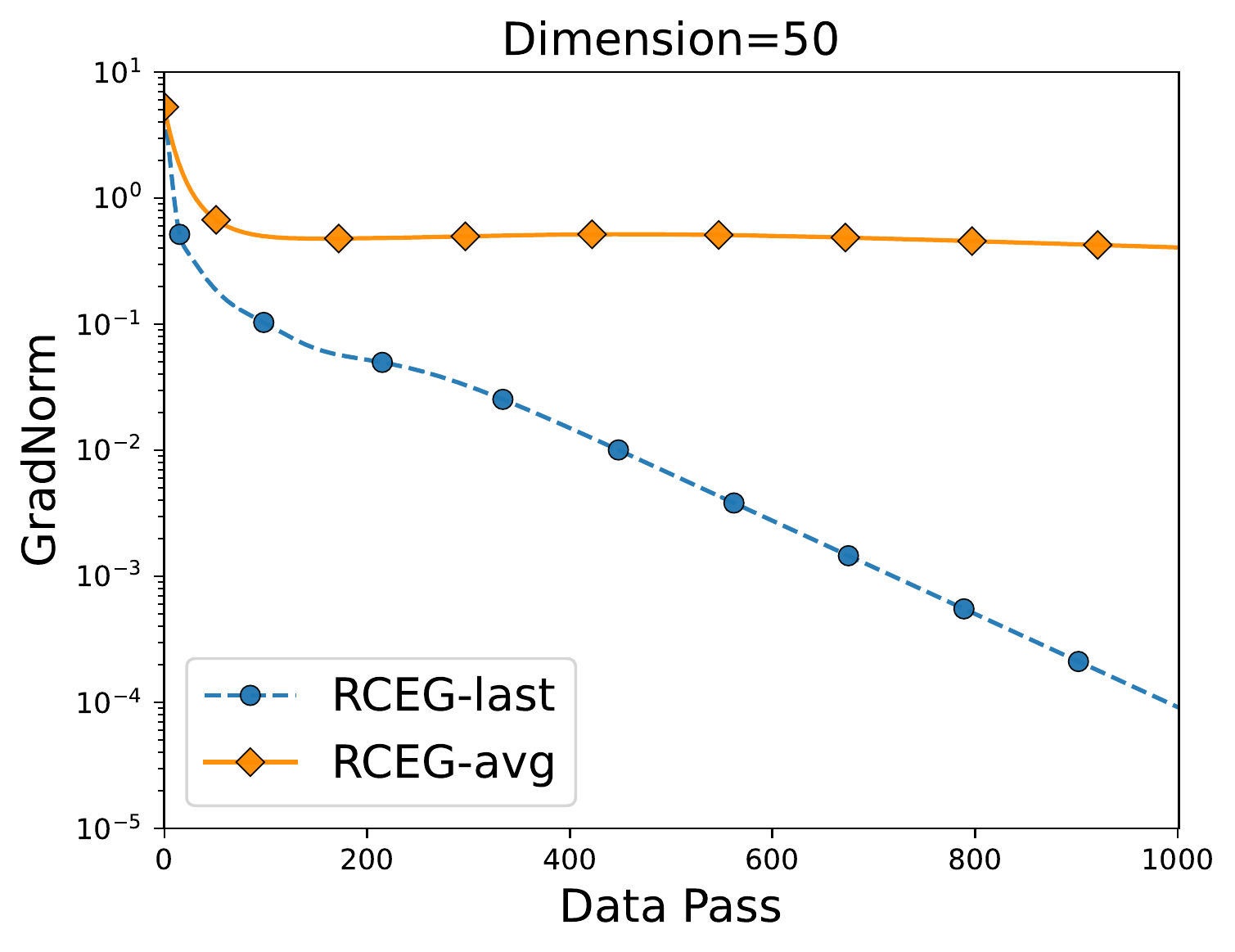}
\includegraphics[width=0.32\textwidth]{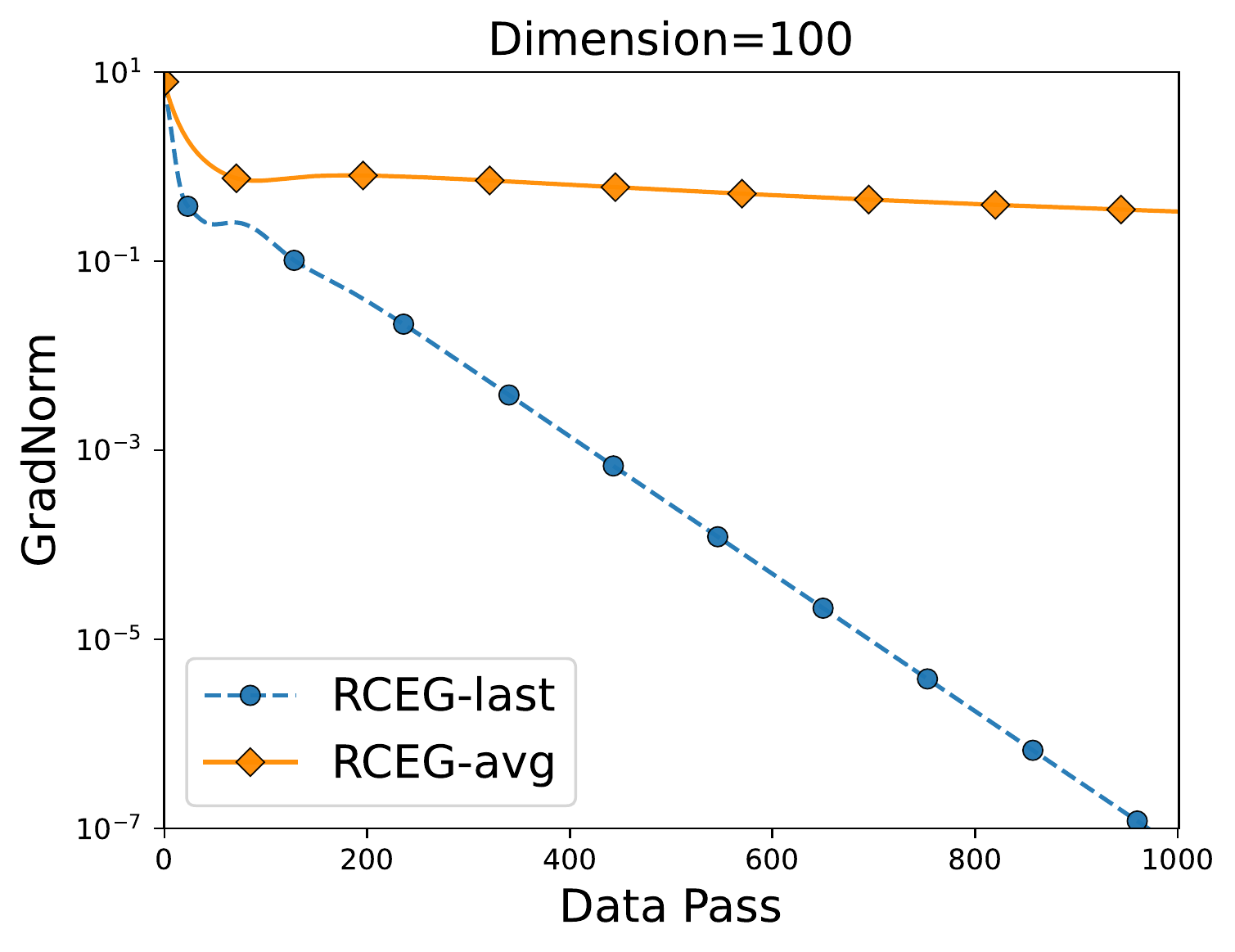}
\caption{\footnotesize{Comparison of last iterate (RCEG-last) and time-average iterate (RCEG-avg) for solving the RPCA problem in Eq.~\eqref{prob:RPCA} with different problem dimensions $d \in \{25, 50, 100\}$. The horizontal axis represents the number of data passes and the vertical axis represents gradient norm. }}\vspace*{-.5em}
\label{fig:exp-deterministic}
\end{figure*}
\begin{figure*}[!t]
\centering
\includegraphics[width=0.24\textwidth]{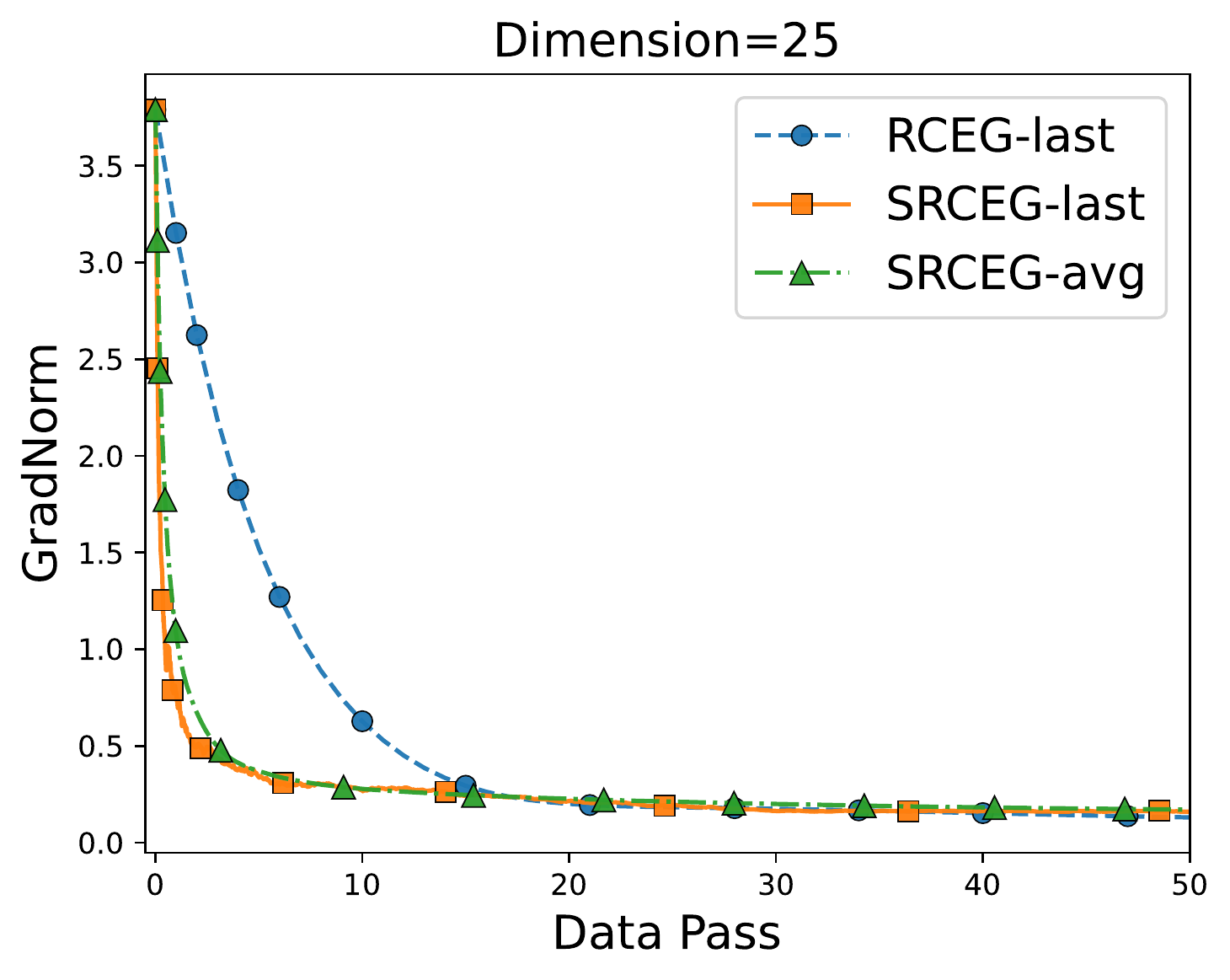}
\includegraphics[width=0.24\textwidth]{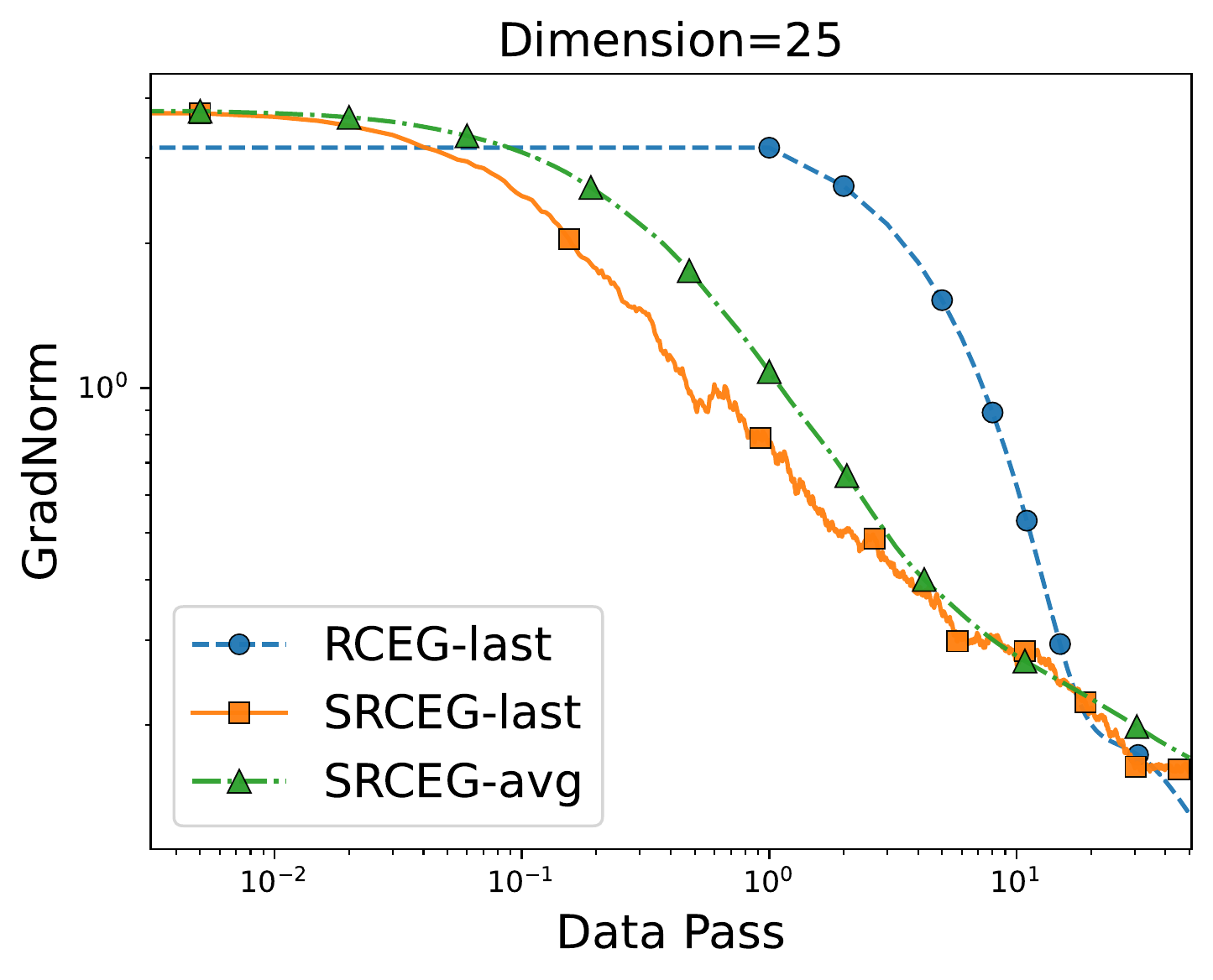}
\includegraphics[width=0.24\textwidth]{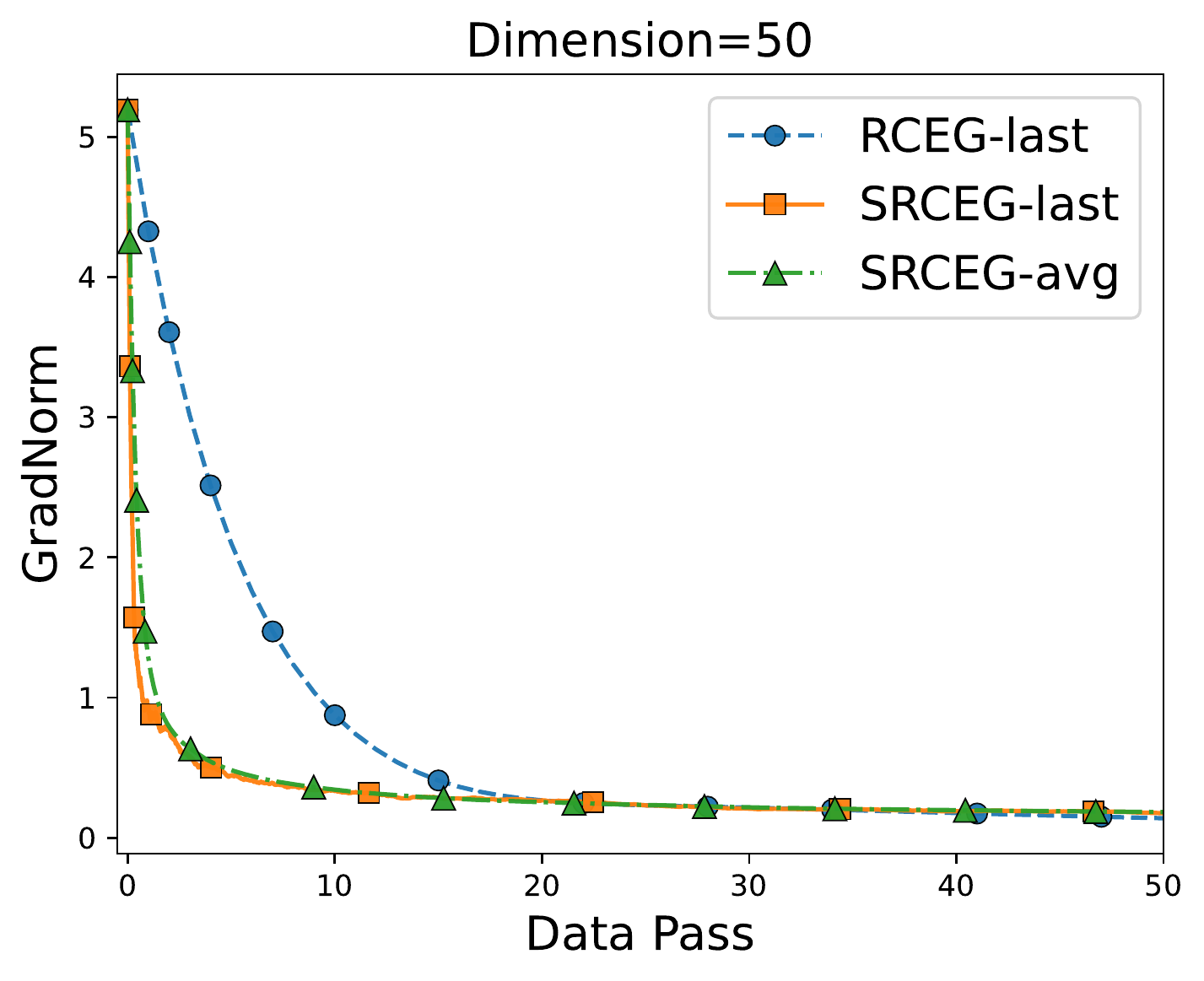}
\includegraphics[width=0.24\textwidth]{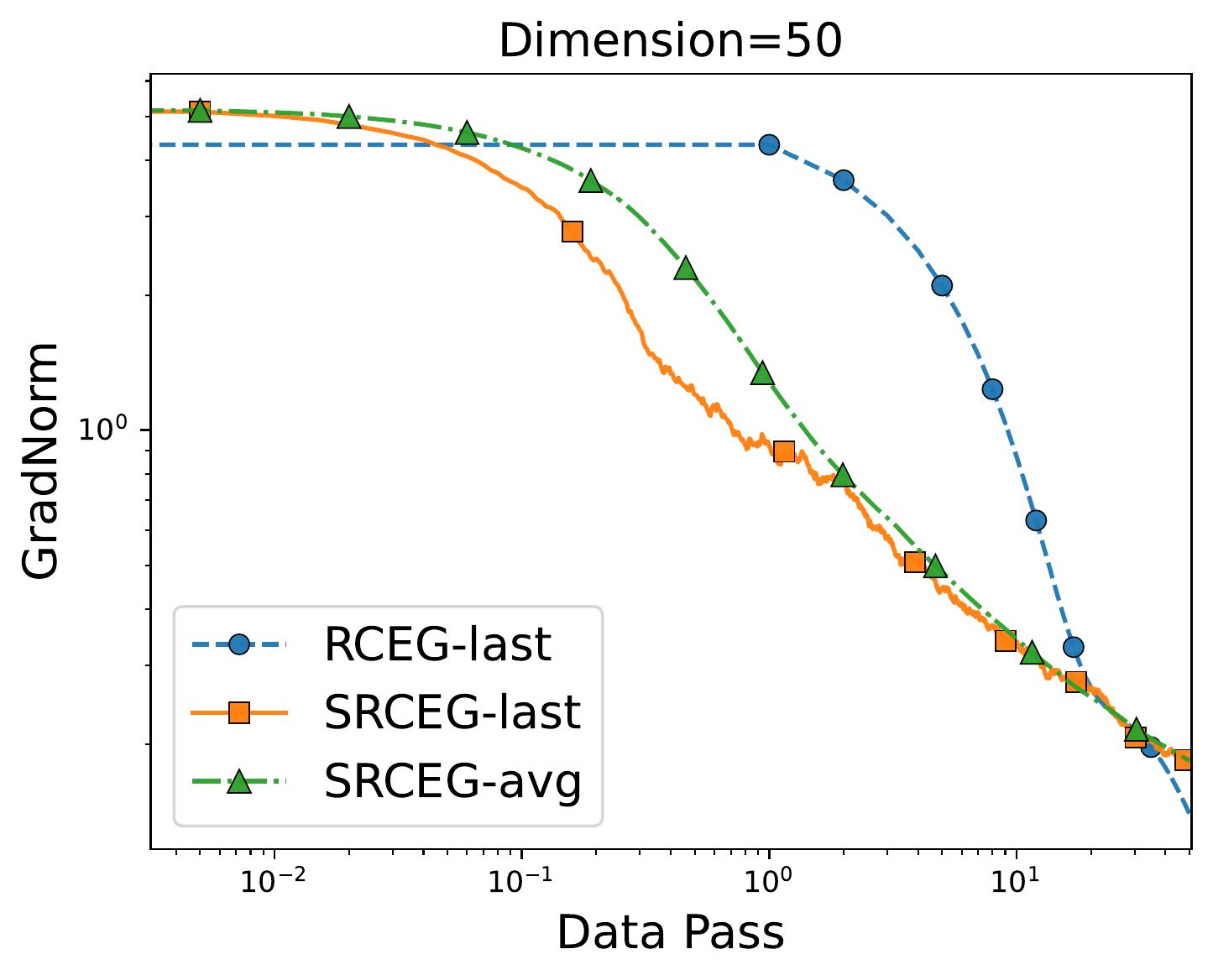}
\caption{\footnotesize{Comparison of RCEG and SRCEG for solving the RPCA problem in Eq.~\eqref{prob:RPCA} with different problem dimensions $d \in \{25, 50\}$. The horizontal axis is the number of data passes and the vertical axis is gradient norm. } }
\label{fig:exp-stoc}\vspace*{-2em}
\end{figure*}
\section{Conclusions}\label{sec:conclu}
Inspired broadly by the structure of the complex competition that arises in many applications of robust optimization in ML, we focus on the problem of min-max optimization in the pure Riemannian setting (where both min and max player are constrained in a smooth manifold). Answering the open question of \citet{Zhang-2022-Minimax} for the geodesically (strongly) convex-concave case, we showed that the Riemannian correction technique for EG matches  the linear last-iterate complexity of their Euclidean counterparts in terms of accuracy and conditional number of objective for both deterministic and stochastic case. Additionally, we provide near-optimal guarantees for both smooth and non-smooth min-max optimization via Riemannian EG and GDA for the simple convex-concave case.

As a consequence of this work numerous open problems emerge; one immediate open question for future work is to explore whether the dependence on the curvature constant is also tight. Additionally, another generalization of interest would be to consider the performance of RCEG in the case of Riemannian Monotone Variational inequalities (RMVI) and examine the generalization of~\citet{Zhang-2022-Minimax} existence proof. Finally, there has been recent work in proving last-iterate convergence in the convex-concave setting via Sum-Of-Squares techniques~\citep{Cai-2022-Tight}. It would be interesting to examine how one could leverage this machinery in a non-Euclidean but geodesic-metric-friendly framework.

\section*{Acknowledgments}
This work was supported in part by the Mathematical Data Science program of the Office of Naval Research under grant number N00014-18-1-2764 and by the Vannevar Bush Faculty Fellowship program under grant number N00014-21-1-2941. The work of Michael I. Jordan is also partially supported by NSF Grant IIS-1901252. Emmanouil V. Vlatakis-Gkaragkounis is grateful for financial support by the Google-Simons Fellowship, Pancretan Association of America and Simons Collaboration on Algorithms and Geometry. This project was completed while he was a visiting research fellow at the Simons Institute for the Theory of Computing. Additionally, he would like to acknowledge the following series of NSF-CCF grants under the numbers 1763970/2107187/1563155/1814873. 

\bibliographystyle{plainnat}
\bibliography{ref}

\appendix
\section{Related Work}\label{sec:further-related}
The literature for the geometric properties of Riemannian Manifolds is immense and hence we cannot hope to survey them here; for an appetizer, we refer the reader to~\citet{Burago-2001-Course} and~\citet{Lee-2012-Introduction} and references therein. On the other hand, as stated, it is not until recently that the long-run non-asymptotic behavior of optimization algorithms in Riemannian manifolds (even the smooth ones) has encountered a lot of interest. For concision, we have deferred here a detailed exposition of the rest of recent results to Appendix~\ref{sec:further-related} of the paper’s supplement. Additionally, in Appendix~\ref{app:examples} we also give a bunch of motivating examples which can be solved by Riemannian min-max optimization. 

\paragraph{Minimization on Riemannian manifolds.} Many application problems can be formulated as the minimization or maximization of a smooth function over Riemannian manifold and has triggered a line of research on the extension of the classical first-order and second-order methods to Riemannian setting with asymptotic convergence to first-order stationary points in general~\citep{Absil-2009-Optimization}. Recent years have witnessed the renewed interests on nonasymptotic convergence analysis of solution methods. In particular,~\citet{Boumal-2019-Global} proved the global sublinear convergence results for Riemannian gradient descent method and Riemannian trust region method, and further demonstrated that the Riemannian trust region method converges to a second-order stationary point in polynomial time; see also similar results in some other works~\citep{Kasai-2018-Inexact, Hu-2018-Adaptive, Hu-2019-Structured}. We are also aware of recent works on problem-specific methods~\citep{Wen-2013-Feasible, Gao-2018-New, Liu-2019-Quadratic} and primal-dual methods~\citep{Zhang-2020-Primal}.

Compared to the smooth counterpart, Riemannian nonsmooth optimization is harder and relatively less explored~\citep{Absil-2019-Collection}. A few existing works focus on optimizing geodesically convex functions over Riemannian manifold with subgradient methods~\citep{Ferreira-1998-Subgradient, Zhang-2016-First, Bento-2017-Iteration}. In particular,~\citet{Ferreira-1998-Subgradient} provided the first asymptotic convergence result while~\citet{Zhang-2016-First} and~\citep{Bento-2017-Iteration} proved an nonasymptotic global convergence rate of $\bigO(\epsilon^{-2})$ for Riemannian subgradient methods. Further,~\citet{Ferreira-2002-Proximal} assumed that the proximal mapping over Riemannian manifold is computationally tractable and proved the global sublinear convergence of Riemannian proximal point method. Focusing on optimization over Stiefel manifold,~\citet{Chen-2020-Proximal} studied the composite objective function and proposed Riemannian proximal gradient method which only needs to compute the proximal mapping of nonsmooth component function over the tangent space of Stiefel manifold.~\citet{Li-2021-Weakly} consider optimizing a weakly convex function over Stiefel manifold and proposed Riemannian subgradient methods that drive a near-optimal stationarity measure below $\epsilon$ within the number of iterations bounded by $O(\epsilon^{-4})$. 

There are some results on stochastic optimization over Riemannian manifold. In particular,~\citet{Bonnabel-2013-Stochastic} proved the first asymptotic convergence result for Riemannian stochastic gradient descent, which is extended by a line of subsequent works~\citep{Zhang-2016-Riemannian, Tripuraneni-2018-Averaging, Becigneul-2019-Riemannian, Kasai-2019-Riemannian}. If the Riemannian Hessian is not positive definite, some recent works have suggested frameworks to escape saddle points~\citep{Sun-2019-Escaping, Criscitiello-2019-Efficiently}. 

\paragraph{Min-Max optimization in Euclidean spaces.} Focusing on solving specifically min-max problems, the algorithms under euclidean geometry  have a very rich history in optimization that goes back at least to the original proximal point algorithms~\citep{Martinet-1970-Regularisation, Rockafellar-1976-Monotone} for variational inequality (VI) problems;  At a high level, if the objective function is Lipschitz and strictly convex-concave, the simple forward-backward schemes are known to converge – and if combined with a Polyak–Ruppert averaging scheme~\citep{Ruppert-1988-Efficient, Polyak-1992-Acceleration, Nemirovski-2009-Robust}, they achieve an $O(1/\epsilon^2)$ complexity\footnote{For the rest of the presentation, we adopt the convention of presenting the \emph{fine-grained complexity} performance measure for computing an $O(\epsilon)$-close solution instead of the \emph{convergence rate} of a method. Thus a rate of the form $\|\x_t-\x^*\|\le O(1/t^{1/p})$ typically corresponds to $O(1/\epsilon^p)$ gradient computations and the geometric rate $\|\x_t-\x^*\|\le O(\exp(-\mu t))$ matches usually up with the $O(\ln (1/\epsilon))$ computational complexity.} without the caveat of strictness~\citep{Bauschke-2011-Convex}. If, in addition, the objective admits Lipschitz continuous gradients, then the extragradient (EG) algorithm~\citep{Korpelevich-1976-Extragradient} achieves trajectory convergence without strict monotonicity requirements, while the time-average iterate converges at $O(1/\epsilon)$ steps~\citep{Nemirovski-2004-Prox}. Finally, if the problem is strongly convex-concave, forward-backward methods computes an $\epsilon$-saddle point at $O(1/\epsilon)$ steps; and if the operator is also Lipschitz continuous, classical results in operator theory show that simple forward-backward methods suffice to achieve a linear convergence rate~\citep{Facchinei-2007-Finite, Bauschke-2011-Convex}.

\paragraph{Min-Max optimization on Riemannian manifolds.} In the case of nonlinear geometry, the literature has been devoted on two different orthogonal axes: \emph{a)} the existence of saddle point  for min-max objective bi-functions and \emph{b)} the design of algorithms for the computation of such points. For the existence of saddle point, a long line of recent work tried to generalize the seminal minima theorem for quasi-convex-quasi-concave problems of \citet{Sion-1958-General}. The crucial bottleneck of this generalization to Riemannian smooth manifolds had been the application of  both Knaster–Kuratowski–Mazurkiewicz (KKM) theorem and Helly’s theorem in non-flat spaces. Before \citet{Zhang-2022-Minimax}, the existence of saddle points had been identified for the special case of Hadamard manifolds~\citep{Komiya-1988-Elementary, Kristaly-2014-Nash,Bento-2017-Iteration,Park-2019-Riemannian}.

Similar with the existence results, initially the developed methods referred to the computation of singularities in monotone variational operators typically in hyperbolic Hadamard manifolds with negative curvature~\citep{Li-2009-Monotone}. More recently, ~\citet{Huang-2020-Gradient} proposed a Riemannian gradient descent ascent method (RGDA), yet the analysis is restricted to $\NCal$ being a convex subset of the Euclidean space and $f(x, y)$ being strongly concave in $y$. It is worth mentioning that for the case Hadamard and generally hyperbolic manifolds, extra-gradient style algorithms have been proposed~\citep{Wang-2010-Monotone, Ferreira-2005-Singularities} in the literature, establishing mainly their asymptotic convergence. However it was not until recent \citet{Zhang-2022-Minimax} that the riemannian correction trick has been analyzed for the case of the extra-gradient algorithm. Bearing in our mind the higher-order methods,~\citet{Han-2022-Riemannian} has recently proposed the Riemannian Hamiltonian Descent and versions of Newton's method for for geodesic convex geodesic concave functions. Since in this work, we focus only on first-order methods, we don't compare with the aforementioned Hamiltonian alternative since it incorporates always the extra computational burden of second-derivatives and hessian over a manifold.

\section{Motivating Examples}\label{app:examples}
We provide some examples of Riemannian min-max optimization to give a sense of their expressivity. Two of the examples are the generic models from the optimization literature~\citep{Ben-2009-Robust, Absil-2009-Optimization, Hu-2020-Brief} and the two others are the formulations of application problems arising from machine learning and data analytics~\citep{Pennec-2006-Riemannian, Fletcher-2007-Riemannian, Lin-2020-Projection}. 
\begin{example}[Riemannian optimization with nonlinear constraints] 
We can consider a rather straightforward generalization of constrained optimization problem from Euclidean spaces to Riemannian manifolds~\citep{Bergmann-2019-Intrinsic}. This formulation finds a wide range of real-world applications, e.g., non-negative principle component analysis, weighted max-cut and so on. Letting $\MCal$ be a finite-dimensional Riemannian manifold with unique geodesic, we focus on the following problem: 
\begin{equation*}
\min_{x \in \MCal} \ f(x), \quad \st \ g(x) \leq 0, \ h(x) = 0, 
\end{equation*}
where $g := (g_1, g_2, \ldots, g_m): \MCal \mapsto \br^m$ and $h := (h_1, h_2, \ldots, h_n): \MCal \mapsto \br^n$ are two mappings. Then, we can introduce the dual variables $\lambda$ and $\mu$ and reformulate the aforementioned constrained optimization problem as follows, 
\begin{equation*}
\min_{x \in \MCal} \max_{(\lambda, \mu) \in \br_+^m \times \br^n} \ f(x) + \langle \lambda, g(x)\rangle + \langle \mu, h(x)\rangle. 
\end{equation*}
Suppose that $f$ and all of $g_i$ and $h_i$ are geodesically convex and smooth, the above problem is a geodesic-convex-Euclidean-concave min-max optimization problem. 
\end{example}
\begin{example}[Distributionally robust Riemannian optimization]
Distributionally robust optimization (DRO) is an effective method to deal with the noisy data, adversarial data, and imbalanced data. We consider the problem of DRO over Riemannian manifold; indeed, given a
set of data samples $\{\xi_i\}_{i=1}^N$, the problem of DRO over Riemannian manifold $\MCal$ can be written in the form of 
\begin{equation*}
\min_{x \in \MCal} \max_{\textbf{p} \in \SCal} \ \sum_{i=1}^N p_i \ell(x; \xi_i) - \|\textbf{p} - \tfrac{1}{N}\textbf{1}\|^2, 
\end{equation*}
where $\textbf{p} = (p_1, p_2, \ldots, p_N)$ and $\SCal = \{\textbf{p} \in \br^N: \sum_{i=1}^N p_i = 1, p_i \geq 0\}$. In general, $\ell(x; \xi_i)$ denotes the loss function over Riemannian manifold $\MCal$. If $\ell$ is geodesically convex and smooth, the above problem is a geodesic-convex-Euclidean-concave min-max optimization problem. 
\end{example}
\begin{example}[Robust matrix Karcher mean problem] We consider a robust version of classical matrix Karcher mean problem. More specifically, the Karcher mean of $N$ symmetric positive definite matrices $\{A_i\}_{i=1}^N$ is defined as the matrix $X \in \MCal = \{X \in \br^{n \times n}: X \succ 0, \ X = X^\top\}$ that minimizes the sum of squared distance induced by the Riemannian metric: 
\begin{equation*}
d(X, Y) = \|\log(X^{-1/2}YX^{-1/2})\|_F. 
\end{equation*}
The loss function is thus defined by 
\begin{equation*}
f(X; \{A_i\}_{i=1}^N) = \sum_{i=1}^N (d(X, A_i))^2.  
\end{equation*}
which is known to be nonconvex in Euclidean spaces but geodesically strongly convex. Then, the robust version of classical matrix Karcher mean problem is aiming at solving the following problem: 
\begin{equation*}
\min_{X \in \MCal} \max_{Y_i \in \MCal} \ f(X; \{Y_i\}_{i=1}^N) - \gamma\left(\sum_{i=1}^N (d(Y_i, A_i))^2\right), 
\end{equation*}
where $\gamma > 0$ stands for the trade-off between the computation of Karcher mean over a set of $\{Y_i\}_{i=1}^N$ and the difference between the observed samples $\{A_i\}_{i=1}^N$ and $\{Y_i\}_{i=1}^N$. It is clear that the above problem is a geodesically strongly-convex-strongly-concave min-max optimization problem. 
\end{example}
\begin{example}[Projection robust optimal transport problem] We consider the projection robust optimal transport (OT) problem -- a robust variant of the OT problem -- that achieves superior sample complexity bound~\citep{Lin-2021-Projection}. Let $\{x_1, x_2, \ldots, x_n\} \subseteq \br^d$ and $\{y_1, y_2, \ldots, y_n\} \subseteq \br^d$ denote sets of $n$ atoms, and let $(r_1, r_2, \ldots, r_n)$ and $(c_1, c_2, \ldots, c_n)$ denote weight vectors. We define discrete probability
measures $\mu = \sum_{i=1}^n r_i \delta_{x_i}$ and $\nu = \sum_{j=1}^n c_j \delta_{y_j}$. In this setting,  the computation of the $k$-dimensional projection robust OT distance between $\mu$ and $\nu$ resorts to solving the following problem: 
\begin{equation*}
\max\limits_{U \in \St(d, k)} \min\limits_{\pi \in \Pi(\mu, \nu)} \sum_{i=1}^n \sum_{j=1}^n \pi_{i, j}\|U^\top x_i - U^\top y_j\|^2, 
\end{equation*}
where $\St(d, k) = \{U \in \br^{d \times k} \mid U^\top U = I_k\}$ is a Stiefel manifold and $\Pi(r, c) = \{\pi \in \br_+^{n \times n} \mid \sum_{j=1}^n \pi_{ij} = r_i, \sum_{i=1}^n \pi_{ij} = c_j\}$ is a transportation polytope. It is worth mentioning that the above problem is a geodesically-nonconvex-Euclidean-concave min-max optimization problem with special structures, making the computation of stationary points tractable. While the global convergence guarantee for our algorithm does not apply, the above problem might be locally geodesically-convex-Euclidean-concave such that our algorithm with sufficiently good initialization works here. 
\end{example}
In addition to these examples, it is worth mentioning that Riemannian min-max optimization problems contain all general min-max optimization problems in Euclidean spaces and all Riemannian minimization or maximization optimization problems. It is also an abstraction of many machine learning problems, e.g,. principle component analysis~\citep{Boumal-2011-RTRMC}, dictionary learning~\citep{Sun-2016-Complete-I, Sun-2016-Complete-II}, deep neural networks (DNNs)~\citep{Huang-2018-Orthogonal} and low-rank matrix learning~\citep{Vandereycken-2013-Low, Jawanpuria-2018-Unified}; indeed, the problem of principle component analysis resorts to optimization problems on Grassmann manifolds for example. 
 
\section{Metric Geometry}\label{app:MG}
To generalize the first-order methods in Euclidean setting, we introduce several basic concepts in metric geometry~\citep{Burago-2001-Course}, which are known to include both Euclidean spaces and Riemannian manifolds as special cases. Formally, we have
\begin{definition}[Metric Space]
A metric space $(X, d)$ is a pair of a set $X$ and a distance function $d(\cdot, \cdot)$ satisfying: (i) $d(x, x') \geq 0$ for any $x, x' \in X$; (ii) $d(x, x') = d(x', x)$ for any $x, x' \in X$; and (iii) $d(x, x'') \leq d(x, x') + d(x', x'')$ for any $x, x', x'' \in X$. In other words, the distance function $d(\cdot, \cdot)$ is non-negative, symmetrical and satisfies the triangle inequality. 
\end{definition}
A \textit{path} $\gamma: [0, 1] \mapsto X$ is a continuous mapping from the interval $[0, 1]$ to $X$ and the \textit{length} of $\gamma$ is defined as $\textnormal{length}(\gamma) \mydefn \lim_{n \rightarrow +\infty} \sup_{0 = t_0 < \ldots < t_n = 1} \sum_{i=1}^n d(\gamma(t_{i-1}), \gamma(t_i))$. Note that the triangle inequality implies that $\sup_{0 = t_0 < \ldots < t_n = 1} \sum_{i=1}^n d(\gamma(t_{i-1}), \gamma(t_i))$ is nondecreasing. Then, the length of a path $\gamma$ is well defined since the limit is either $+\infty$ or a finite scalar. Moreover, for $\forall \epsilon > 0$, there exists $n \in \bn$ and the partition $0 = t_0 < \ldots < t_n = 1$ of the interval $[0, 1]$ such that $\textnormal{length}(\gamma) \leq \sum_{i=1}^n d(\gamma(t_{i-1}), \gamma(t_i)) + \epsilon$. 
\begin{definition}[Length Space]\label{def:length}
A metric space $(X, d)$ is a length space if, for any $x, x' \in X$ and $\epsilon > 0$, there exists a path $\gamma: [0, 1] \mapsto X$ connecting $x$ and $x'$ such that $\textnormal{length}(\gamma) \leq d(x, x') + \epsilon$. 
\end{definition}
We can see from Definition~\ref{def:length} that a set of length spaces is strict subclass of metric spaces; indeed, for some $x, x' \in X$, there does not exist a path $\gamma$ such that its length can be approximated by $d(x, x')$ for some tolerance $\epsilon > 0$. In metric geometry, a \textit{geodesic} is a path which is locally a distance minimizer everywhere. More precisely, a path $\gamma$ is a geodesic if there is a constant $\nu > 0$ such that for any $t \in [0, 1]$ there is a neighborhood $I$ of $[0,1]$ such that, 
\begin{equation*}
d(\gamma(t_1),\gamma(t_2)) = \nu|t_1 - t_2|, \quad \textnormal{for any } t_1, t_2 \in I. 
\end{equation*}
Note that the above generalizes the notion of geodesic for Riemannian manifolds. Then, we are ready to introduce the geodesic space and uniquely geodesic space~\citep{Bacak-2014-Convex}. 
\begin{definition}
A metric space $(X, d)$ is a geodesic space if, for any $x, x' \in X$, there exists a geodesic $\gamma: [0, 1] \mapsto X$ connecting $x$ and $x'$. Furthermore, it is called uniquely geodesic if the geodesic connecting $x$ and $x'$ is unique for any $x, x' \in X$. 
\end{definition}
Trigonometric geometry in nonlinear spaces is intrinsically different from Euclidean space. In particular, we remark that the law of cosines in Euclidean space (with $\|\cdot\|$ as $\ell_2$-norm) is crucial for analyzing the convergence property of optimization algorithms, e.g., 
\begin{equation*}
\|a\|^2 = \|b\|^2 + \|c\|^2 - 2bc\cos(A), 
\end{equation*}
where $a$, $b$, $c$ are sides of \textit{a geodesic triangle} in Euclidean space and $A$ is the angle between $b$ and $c$. However, such nice property does not hold for nonlinear spaces due to the lack of flat geometry, further motivating us to extend the law of cosines under nonlinear trigonometric geometry. That is to say, given a geodesic triangle in $X$ with sides $a$, $b$, $c$ where $A$ is the angle between $b$ and $c$, we hope to establish the relationship between $a^2$, $b^2$, $c^2$ and $2bc\cos(A)$ in nonlinear spaces; see the main context for the comparing inequalities. 

Finally, we specify the definition of \textit{section curvature} of Riemannian manifolds and clarify how such quantity affects the trigonometric comparison inequalities. More specifically, the sectional curvature is defined as the Gauss curvature of a 2-dimensional sub-manifold that are obtained from the image of a two-dimensional subspace of a tangent space after exponential mapping. It is worth mentioning that the above 2-dimensional sub-manifold is locally isometric to a 2-dimensional sphere, a Euclidean plane, and a hyperbolic plane with the same Gauss curvature if its sectional curvature is positive, zero and negative respectively.  Then we are ready to summarize the existing trigonometric comparison inequalities for Riemannian manifold with bounded sectional curvatures. Note that the following two propositions are the full version of Proposition~\ref{Prop:key-inequality} and will be used in our subsequent proofs. 
\begin{proposition}\label{Prop:key-inequality-UB}
Suppose that $\MCal$ is a Riemannian manifold with sectional curvature that is upper bounded by $\kappa_{\max}$ and let $\Delta$ be a geodesic triangle in $\MCal$ with the side length $a$, $b$, $c$ and $A$ which is the angle between $b$ and $c$. If $\kappa_{\max} > 0$, we assume the diameter of $\MCal$ is bounded by $\frac{\pi}{\sqrt{\kappa_{\max}}}$. Then, we have
\begin{equation*}
a^2 \geq \underline{\xi}(\kappa_{\max}, c) \cdot b^2 + c^2 - 2bc\cos(A), 
\end{equation*}
where $\underline{\xi}(\kappa, c) := 1$ for $\kappa \leq 0$ and $\underline{\xi}(\kappa, c) := c\sqrt{\kappa}\cot(c\sqrt{\kappa}) < 1$ for $\kappa > 0$.  
\end{proposition}
\begin{proposition}\label{Prop:key-inequality-LB}
Suppose that $\MCal$ is a Riemannian manifold with sectional curvature that is lower bounded by $\kappa_{\min}$ and let $\Delta$ be a geodesic triangle in $\MCal$ with the side length $a$, $b$, $c$ and $A$ which is the angle between $b$ and $c$. Then, we have
\begin{equation*}
a^2 \leq \overline{\xi}(\kappa_{\min}, c) \cdot b^2 + c^2 - 2bc\cos(A), 
\end{equation*}
where $\overline{\xi}(\kappa, c) := c\sqrt{-\kappa}\coth(c\sqrt{-\kappa}) > 1$ if $\kappa < 0$ and $\overline{\xi}(\kappa, c) := 1$ if $\kappa \geq 0$. 
\end{proposition}
\begin{remark}
Proposition~\ref{Prop:key-inequality-UB} and~\ref{Prop:key-inequality-LB} are simply the restatement of~\citet[Corollary~2.1]{Alimisis-2020-Continuous} and~\citet[Lemma~5]{Zhang-2016-First}. The former inequality is obtained when the sectional curvature is bounded from above while the latter inequality characterizes the relationship between the trigonometric lengths when the sectional curvature is bounded from below. If $\kappa_{\min} = \kappa_{\max} = 0$ (i.e., Euclidean spaces), we have $\overline{\xi}(\kappa_{\min}, c) = \underline{\xi}(\kappa_{\max}, c) = 1$. The proof is based on Toponogov's theorem and Riccati comparison estimate~\citep[Proposition~25]{Petersen-2006-Riemannian} and we refer the interested readers to~\citet{Zhang-2016-First} and~\citet{Alimisis-2020-Continuous} for the details. 
\end{remark}

\section{Riemannian Gradient Descent Ascent for Nonsmooth Setting}
In this section, we propose and analyze Riemannian gradient descent ascent (RGDA) method for nonsmooth Riemannian min-max optimization and extend it to stochastic RGDA. We present our results on the optimal last-iterate convergence guarantee for geodesically strongly-convex-strongly-concave setting (both deterministic and stochastic) and time-average convergence guarantee for geodesically convex-concave setting (both deterministic and stochastic). 
\subsection{Algorithmic scheme}
Compared to Riemannian corrected extragradient (RCEG) method, our Riemannian gradient descent ascent (RGDA) method is a relatively straightforward generalization of GDA in Euclidean spaces. More specifically, we start with the scheme of GDA as follows (just consider $\MCal$ and $\NCal$ as convex constraint sets in Euclidean spaces), 
\begin{equation}\label{def:GDA}
\begin{array}{rclcrcl}
x_{t+1} & \leftarrow & \proj_\MCal(x_t - \eta_t \cdot g_x^t), & & y_{t+1} & \leftarrow & \proj_\NCal(y_t + \eta_t \cdot g_y^t). 
\end{array}
\end{equation}
where $(g_x^t, g_y^t) \in (\partial_x f(x_t, y_t), \partial_y f(x_t, y_t))$ is one subgradient of $f$. By replacing the projection operator by the corresponding exponential map and the gradient by the corresponding Riemannian gradient, we have
\begin{equation*}
x_{t+1} \leftarrow \Exp_{x_t}(- \eta_t \cdot g_x^t), \quad y_{t+1} \leftarrow \Exp_{y_t}(\eta_t \cdot g_y^t). 
\end{equation*}
where $(g_x^t, g_y^t) \leftarrow (\subg_x f(x_t, y_t), \subg_y f(x_t, y_t))$ is one Riemannian subgradient of $f$. Then, we summarize the resulting scheme of RGDA method in Algorithm~\ref{alg:RGDA} and its stochastic extension with noisy estimators of Riemannian gradients of $f$ in Algorithm~\ref{alg:SRGDA}. 

\begin{table}[!t]
\begin{tabular}{cc}
\begin{minipage}{.54\textwidth}
\begin{algorithm}[H]\small
\caption{RGDA}\label{alg:RGDA}
\begin{algorithmic}
\State \textbf{Input:} initial points $(x_0, y_0)$ and stepsizes $\eta_t > 0$. 
\For{$t = 0, 1, 2, \ldots, T-1$}
\State Query $(g_x^t, g_y^t) \leftarrow (\subg_x f(x_t, y_t), \subg_y f(x_t, y_t))$ as Riemannian subgradient of $f$ at a point $(x_t, y_t)$. 
\State $x_{t+1} \leftarrow \Exp_{x_t}(-\eta_t \cdot g_x^t)$. 
\State $y_{t+1} \leftarrow \Exp_{y_t}(\eta_t \cdot g_y^t)$. 
\EndFor
\end{algorithmic}
\end{algorithm}
\end{minipage} &
\begin{minipage}{.46\textwidth}
\begin{algorithm}[H]\small
\caption{SRGDA}\label{alg:SRGDA}
\begin{algorithmic}
\State  \textbf{Input:} initial points $(x_0, y_0)$ and stepsizes $\eta_t > 0$. 
\For{$t = 0, 1, 2, \ldots, T-1$}
\State Query $(g_x^t, g_y^t)$ as a \textbf{noisy} estimator of Riemannian subgradient of $f$ at a point $(x_t, y_t)$. 
\State $x_{t+1} \leftarrow \Exp_{x_t}(-\eta_t \cdot g_x^t)$. 
\State $y_{t+1} \leftarrow \Exp_{y_t}(\eta_t \cdot g_y^t)$. 
\EndFor
\end{algorithmic}
\end{algorithm}
\end{minipage}
\end{tabular}
\end{table}

\subsection{Main results}
We present our main results on the global convergence rate estimation for Algorithm~\ref{alg:RGDA} and~\ref{alg:SRGDA} in terms of Riemannian gradient and noisy Riemannian gradient evaluations. The following assumptions are made throughout for geodesically strongly-convex-strongly-concave and geodesically convex-concave settings.
\begin{assumption}\label{Assumption:gscsc-nonsmooth} 
The objective function $f: \MCal \times \NCal \mapsto \br$ and manifolds $\MCal$ and $\NCal$ satisfy
\begin{enumerate}
\item $f$ is geodesically $L$-Lipschitz and geodesically strongly-convex-strongly-concave with $\mu > 0$. 
\item The diameter of the domain $\{(x, y) \in \MCal \times \NCal: -\infty < f(x, y) < +\infty\}$ is bounded by $D > 0$. 
\item The sectional curvatures of $\MCal$ and $\NCal$ are both bounded in the range $[\kappa_{\min}, +\infty)$ with $\kappa_{\min} \leq 0$. 
\end{enumerate}
\end{assumption}
\begin{assumption}\label{Assumption:gcc-nonsmooth}
The objective function $f: \MCal \times \NCal \mapsto \br$ and manifolds $\MCal$ and $\NCal$ satisfy
\begin{enumerate}
\item $f$ is geodesically $L$-Lipschitz and geodesically convex-concave. 
\item The diameter of the domain $\{(x, y) \in \MCal \times \NCal: -\infty < f(x, y) < +\infty\}$ is bounded by $D > 0$. 
\item The sectional curvatures of $\MCal$ and $\NCal$ are both bounded in the range $[\kappa_{\min}, +\infty)$ with $\kappa_{\min} \leq 0$. 
\end{enumerate}
\end{assumption}  
Imposing the geodesically Lipschitzness condition is crucial to achieve finite-time convergence guarantee if we do not assume the geodesically smoothness condition. Note that we only require the lower bound for the sectional curvatures of manifolds and this is weaker than that presented in the main context. 

Letting $(x^\star, y^\star) \in \MCal \times \NCal$ be a global saddle point of $f$ (it exists under either Assumption~\ref{Assumption:gscsc-nonsmooth} or~\ref{Assumption:gcc-nonsmooth}), we let $D_0 = (d_\MCal(x_0, x^\star))^2 + (d_\NCal(y_0, y^\star))^2 > 0$ and summarize our results for Algorithm~\ref{alg:RGDA} in the following theorems. 
\begin{theorem}\label{Thm:RGDA-SCSC}
Under Assumption~\ref{Assumption:gscsc-nonsmooth} and let $\eta_t > 0$ satisfies that $\eta_t = \frac{1}{\mu}\min\{1, \frac{2}{t}\}$. There exists some $T > 0$ such that the output of Algorithm~\ref{alg:RGDA} satisfies that $(d(x_T, x^\star))^2 + (d(y_T, y^\star))^2 \leq \epsilon$ and the total number of Riemannian subgradient evaluations is bounded by
\begin{equation*}
O\left(\frac{\overline{\xi}_0L^2}{\mu^2\epsilon}\right), 
\end{equation*}
where $\overline{\xi}_0 = \overline{\xi}(\kappa_{\min}, D)$ measures the lower bound for the change of non-flatness in $\MCal$ and $\NCal$. 
\end{theorem}
\begin{theorem}\label{Thm:RGDA-CC}
Under Assumption~\ref{Assumption:gcc-nonsmooth} and let $\eta_t > 0$ satisfies that $\eta_t = \tfrac{1}{L}\sqrt{\tfrac{D_0}{2\overline{\xi}_0 T}}$. There exists some $T > 0$ such that the output of Algorithm~\ref{alg:RGDA} satisfies that $f(\bar{x}_T, y^\star) - f(x^\star, \bar{y}_T) \leq \epsilon$ and the total number of Riemannian subgradient evaluations is bounded by
\begin{equation*}
O\left(\frac{\overline{\xi}_0L^2D_0}{\epsilon^2}\right), 
\end{equation*}
where $\overline{\xi}_0 = \overline{\xi}(\kappa_{\min}, D)$ measures the lower bound for the change of non-flatness in $\MCal$ and $\NCal$, and the time-average iterates $(\bar{x}_T, \bar{y}_T) \in \MCal \times \NCal$ can be computed by $(\bar{x}_0, \bar{y}_0) = (0, 0)$ and the inductive formula: $\bar{x}_{t+1} = \Exp_{\bar{x}_t}(\tfrac{1}{t+1} \cdot \Exp_{\bar{x}_t}^{-1}(x_t))$ and $\bar{y}_{t+1} = \Exp_{\bar{y}_t}(\tfrac{1}{t+1} \cdot \Exp_{\bar{y}_t}^{-1}(y_t))$ for all $t = 0, 1, \ldots, T-1$. 
\end{theorem}
\begin{remark}
Theorem~\ref{Thm:RGDA-SCSC} and~\ref{Thm:RGDA-CC} establish the last-iterate and time-average rates of convergence of Algorithm~\ref{alg:RGDA} for solving Riemannian min-max optimization problems under Assumption~\ref{Assumption:gscsc-nonsmooth} and~\ref{Assumption:gcc-nonsmooth} respectively. Further, the dependence on $L$ and $1/\epsilon$ can not be improved since it has matched the lower bound established for the nonsmooth min-max optimization problems in Euclidean spaces. 
\end{remark}
In the scheme of SRGDA, we highlight that $(g_x^t, g_y^t)$ is a noisy estimators of Riemannian subgradient of $f$ at $(x_t, y_t)$. It is necessary to impose the conditions such that these estimators are unbiased and has bounded variance. By abuse of notation, we assume that 
\begin{equation}\label{def:noisy-model-RGDA}
\begin{array}{lcl}
g_x^t = \subg_x f(x_t, y_t) + \xi_x^t, & & g_y^t = \subg_y f(x_t, y_t) + \xi_y^t, 
\end{array}
\end{equation}
where the noises $(\xi_x^t, \xi_y^t)$ satisfy that 
\begin{equation}\label{def:noisy-bound-RGDA}
\begin{array}{lclcl}
\EE[\xi_x^t] = 0, & & \EE[\xi_y^t] = 0, & & \EE[\|\xi_x^t\|^2 + \|\xi_y^t\|^2] \leq \sigma^2. 
\end{array}
\end{equation}
We are ready to summarize our results for Algorithm~\ref{alg:SRGDA} in the following theorems. 
\begin{theorem}\label{Thm:SRGDA-SCSC}
Under Assumption~\ref{Assumption:gscsc-nonsmooth} and let Eq.~\eqref{def:noisy-model-RGDA} and Eq.~\eqref{def:noisy-bound-RGDA} hold with $\sigma > 0$ and let $\eta_t > 0$ satisfies that $\eta_t = \frac{1}{\mu}\min\{1, \frac{2}{t}\}$. There exists some $T > 0$ such that the output of Algorithm~\ref{alg:SRGDA} satisfies that $\EE[(d(x_T, x^\star))^2 + (d(y_T, y^\star))^2] \leq \epsilon$ and the total number of noisy Riemannian gradient evaluations is bounded by
\begin{equation*}
O\left(\frac{\overline{\xi}_0(L^2 + \sigma^2)}{\mu^2\epsilon}\right), 
\end{equation*}
where $\overline{\xi}_0 = \overline{\xi}(\kappa_{\min}, D)$ measures the lower bound for the change of non-flatness in $\MCal$ and $\NCal$. 
\end{theorem}
\begin{theorem}\label{Thm:SRGDA-CC}
Under Assumption~\ref{Assumption:gcc-nonsmooth} and let Eq.~\eqref{def:noisy-model-RGDA} and Eq.~\eqref{def:noisy-bound-RGDA} hold with $\sigma > 0$ and let $\eta_t > 0$ satisfies that $\eta_t = \tfrac{1}{2}\sqrt{\tfrac{D_0}{\overline{\xi}_0(L^2 + \sigma^2)T}}$. There exists some $T > 0$ such that the output of Algorithm~\ref{alg:SRGDA} satisfies that $\EE[f(\bar{x}_T, y^\star) - f(x^\star, \bar{y}_T)] \leq \epsilon$ and the total number of noisy Riemannian gradient evaluations is bounded by
\begin{equation*}
O\left(\frac{\overline{\xi}_0(L^2 + \sigma^2)D_0}{\epsilon^2}\right), 
\end{equation*}
where $\overline{\xi}_0 = \overline{\xi}(\kappa_{\min}, D)$ measures the lower bound for the change of non-flatness in $\MCal$ and $\NCal$, and the time-average iterates $(\bar{x}_T, \bar{y}_T) \in \MCal \times \NCal$ can be computed by $(\bar{x}_0, \bar{y}_0) = (0, 0)$ and the inductive formula: $\bar{x}_{t+1} = \Exp_{\bar{x}_t}(\tfrac{1}{t+1} \cdot \Exp_{\bar{x}_t}^{-1}(x_t))$ and $\bar{y}_{t+1} = \Exp_{\bar{y}_t}(\tfrac{1}{t+1} \cdot \Exp_{\bar{y}_t}^{-1}(y_t))$ for all $t = 0, 1, \ldots, T-1$. 
\end{theorem}
\begin{remark}
Theorem~\ref{Thm:SRGDA-SCSC} and~\ref{Thm:SRGDA-CC} establish the last-iterate and time-average rates of convergence of Algorithm~\ref{alg:SRGDA} for solving Riemannian min-max optimization problems under Assumption~\ref{Assumption:gscsc-nonsmooth} and~\ref{Assumption:gcc-nonsmooth}. Moreover, the dependence on $L$ and $1/\epsilon$ can not be improved since it has matched the lower bound established for nonsmooth stochastic min-max optimization problems in Euclidean spaces. 
\end{remark}

\section{Missing Proofs for Riemannian Corrected Extragradient Method}
In this section, we present some technical lemmas for analyzing the convergence property of Algorithm~\ref{alg:RCEG} and~\ref{alg:SRCEG}. We also give the proofs of Theorem~\ref{Thm:RCEG-SCSC},~\ref{Thm:SRCEG-SCSC} and~\ref{Thm:SRCEG-CC}. 

\subsection{Technical lemmas}
We provide two technical lemmas for analyzing Algorithm~\ref{alg:RCEG} and~\ref{alg:SRCEG} respectively.  Parts of the first lemma were presented in~\citet[Lemma~C.1]{Zhang-2022-Minimax}. For the completeness, we provide the proof details. 
\begin{lemma}\label{Lemma:RCEG-key-inequality}
Under Assumption~\ref{Assumption:gscsc-smooth} and let $\{(x_t, y_t), (\hat{x}_t, \hat{y}_t)\}_{t=0}^{T-1}$ be generated by Algorithm~\ref{alg:RCEG} with the stepsize $\eta > 0$. Then, we have
\begin{eqnarray*}
\lefteqn{0 \leq \tfrac{1}{2}\left((d_\MCal(x_t, x^\star))^2 - (d_\MCal(x_{t+1}, x^\star))^2 + (d_\NCal(y_t, y^\star))^2 - (d_\NCal(y_{t+1}, y^\star))^2\right)} \\
& & + 2\overline{\xi}_0\eta^2\ell^2((d_\MCal(\hat{x}_t, x_t))^2 + (d_\NCal(\hat{y}_t, y_t))^2 - \tfrac{1}{2}\underline{\xi}_0\left((d_\MCal(\hat{x}_t, x_t))^2 + (d_\NCal(\hat{y}_t, y_t))^2\right) \\ 
& & - \tfrac{\mu \eta}{2}\left((d_\MCal(\hat{x}_t, x^\star))^2 + (d_\NCal(\hat{y}_t, y^\star))^2\right). 
\end{eqnarray*}
where $(x^\star, y^\star) \in \MCal \times \NCal$ is a global saddle point of $f$. 
\end{lemma}
\begin{proof}
Since $f$ is geodesically $\ell$-smooth, we have the Riemannian gradients of $f$, i.e., $(\grad_x f, \grad_y f)$, are well defined. Since $f$ is geodesically strongly-concave-strongly-concave with the modulus $\mu \geq 0$ (here $\mu = 0$ means that $f$ is geodesically concave-concave), we have
\begin{eqnarray*}
\lefteqn{f(\hat{x}_t, y^\star) - f(x^\star, \hat{y}_t) =  f(\hat{x}_t, \hat{y}_t) - f(x^\star, \hat{y}_t) - (f(\hat{x}_t, \hat{y}_t) - f(\hat{x}_t, y^\star))}  \\
& \overset{\textnormal{Definition~\ref{def:SCSC}}}{\leq} & - \langle \grad_x f(\hat{x}_t,\hat{y}_t), \Exp_{\hat{x}_t}^{-1}(x^\star) \rangle + \langle \grad_y f(\hat{x}_t,\hat{y}_t), \Exp_{\hat{y}_t}^{-1}(y^\star)\rangle - \tfrac{\mu}{2}(d_\MCal(\hat{x}_t, x^\star))^2 - \tfrac{\mu}{2}(d_\NCal(\hat{y}_t, y^\star))^2.
\end{eqnarray*}
Since $(x^\star, y^\star) \in \MCal \times \NCal$ is a global saddle point of $f$, we have $f(\hat{x}_t, y^\star) - f(x^\star, \hat{y}_t) \geq 0$. Recalling also from the scheme of Algorithm~\ref{alg:RCEG} that we have
\begin{eqnarray*}
x_{t+1} & \leftarrow & \Exp_{\hat{x}_t}(-\eta \cdot \grad_x f(\hat{x}_t,\hat{y}_t) + \Exp_{\hat{x}_t}^{-1}(x_t)), \\
y_{t+1} & \leftarrow & \Exp_{\hat{y}_t}(\eta \cdot \grad_y f(\hat{x}_t,\hat{y}_t) + \Exp_{\hat{y}_t}^{-1}(y_t)). 
\end{eqnarray*}
By the definition of an exponential map, we have
\begin{equation}\label{inequality:RCEG-opt}
\begin{array}{lcl}
\Exp_{\hat{x}_t}^{-1}(x_{t+1}) & = & -\eta \cdot \grad_x f(\hat{x}_t,\hat{y}_t) + \Exp_{\hat{x}_t}^{-1}(x_t),  \\
\Exp_{\hat{y}_t}^{-1}(y_{t+1}) & = & \eta \cdot \grad_y f(\hat{x}_t,\hat{y}_t) + \Exp_{\hat{y}_t}^{-1}(y_t). 
\end{array}
\end{equation}
This implies that 
\begin{eqnarray*}
- \langle \grad_x f(\hat{x}_t,\hat{y}_t), \Exp_{\hat{x}_t}^{-1}(x^\star) \rangle & = & \tfrac{1}{\eta} (\langle \Exp_{\hat{x}_t}^{-1}(x_{t+1}), \Exp_{\hat{x}_t}^{-1}(x^\star)\rangle - \langle \Exp_{\hat{x}_t}^{-1}(x_t), \Exp_{\hat{x}_t}^{-1}(x^\star) \rangle), \\
\langle \grad_y f(\hat{x}_t,\hat{y}_t), \Exp_{\hat{y}_t}^{-1}(y^\star) \rangle & = & \tfrac{1}{\eta} (\langle \Exp_{\hat{y}_t}^{-1}(y_{t+1}), \Exp_{\hat{y}_t}^{-1}(y^\star)\rangle - \langle \Exp_{\hat{y}_t}^{-1}(y_t), \Exp_{\hat{y}_t}^{-1}(y^\star)\rangle). 
\end{eqnarray*}
Putting these pieces together yields that 
\begin{eqnarray*}
\lefteqn{0 \leq \tfrac{1}{\eta} (\langle \Exp_{\hat{x}_t}^{-1}(x_{t+1}), \Exp_{\hat{x}_t}^{-1}(x^\star)\rangle - \langle \Exp_{\hat{x}_t}^{-1}(x_t), \Exp_{\hat{x}_t}^{-1}(x^\star) \rangle) - \tfrac{\mu}{2}(d_\MCal(\hat{x}_t, x^\star))^2} \\
& & + \tfrac{1}{\eta} (\langle \Exp_{\hat{y}_t}^{-1}(y_{t+1}), \Exp_{\hat{y}_t}^{-1}(y^\star)\rangle - \langle \Exp_{\hat{y}_t}^{-1}(y_t), \Exp_{\hat{y}_t}^{-1}(y^\star)\rangle) - \tfrac{\mu}{2}(d_\NCal(\hat{y}_t, y^\star))^2. 
\end{eqnarray*}
Equivalently, we have
\begin{eqnarray}\label{inequality:RCEG-key-first}
\lefteqn{0 \leq \langle \Exp_{\hat{x}_t}^{-1}(x_{t+1}), \Exp_{\hat{x}_t}^{-1}(x^\star)\rangle - \langle \Exp_{\hat{x}_t}^{-1}(x_t), \Exp_{\hat{x}_t}^{-1}(x^\star) \rangle - \tfrac{\mu \eta}{2}(d_\MCal(\hat{x}_t, x^\star))^2} \\
& & + \langle \Exp_{\hat{y}_t}^{-1}(y_{t+1}), \Exp_{\hat{y}_t}^{-1}(y^\star)\rangle - \langle \Exp_{\hat{y}_t}^{-1}(y_t), \Exp_{\hat{y}_t}^{-1}(y^\star)\rangle - \tfrac{\mu\eta}{2}(d_\NCal(\hat{y}_t, y^\star))^2. \nonumber
\end{eqnarray}
It suffices to bound the terms in the right-hand side of Eq.~\eqref{inequality:RCEG-key-first} by leveraging the celebrated comparison inequalities on Riemannian manifold with bounded sectional curvature (see Proposition~\ref{Prop:key-inequality-UB} and~\ref{Prop:key-inequality-LB}). More specifically, we define the constants using $\overline{\xi}(\cdot, \cdot)$ and $\underline{\xi}(\cdot, \cdot)$ from Proposition~\ref{Prop:key-inequality-UB} and~\ref{Prop:key-inequality-LB} as follows, 
\begin{equation*}
\overline{\xi}_0 = \overline{\xi}(\kappa_{\min}, D), \qquad \underline{\xi}_0 = \underline{\xi}(\kappa_{\max}, D). 
\end{equation*}
By Proposition~\ref{Prop:key-inequality-UB} and using that $\max\{d_\MCal(\hat{x}_t, x^\star), d_\NCal(\hat{y}_t, y^\star)\} \leq D$, we have
\begin{equation}\label{inequality:RCEG-key-second}
\begin{array}{lll}
- \langle \Exp_{\hat{x}_t}^{-1}(x_t), \Exp_{\hat{x}_t}^{-1}(x^\star) \rangle & \leq & - \tfrac{1}{2}\left(\underline{\xi}_0(d_\MCal(\hat{x}_t, x_t))^2 + (d_\MCal(\hat{x}_t, x^\star))^2 - (d_\MCal(x_t, x^\star))^2\right), \\
- \langle \Exp_{\hat{y}_t}^{-1}(y_t), \Exp_{\hat{y}_t}^{-1}(y^\star) \rangle & \leq & - \tfrac{1}{2}\left(\underline{\xi}_0(d_\NCal(\hat{y}_t, y_t))^2 + (d_\NCal(\hat{y}_t, y^\star))^2 - (d_\NCal(y_t, y^\star))^2\right). 
\end{array}
\end{equation}
By Proposition~\ref{Prop:key-inequality-LB} and using that $\max\{d_\MCal(\hat{x}_t, x^\star), d_\NCal(\hat{y}_t, y^\star)\} \leq D$, we have
\begin{equation*}
\langle \Exp_{\hat{x}_t}^{-1}(x_{t+1}), \Exp_{\hat{x}_t}^{-1}(x^\star) \rangle \leq \tfrac{1}{2}\left(\overline{\xi}_0(d_\MCal(\hat{x}_t, x_{t+1}))^2 + (d_\MCal(\hat{x}_t, x^\star))^2 - (d_\MCal(x_{t+1}, x^\star))^2\right). 
\end{equation*}
and
\begin{equation*}
\langle \Exp_{\hat{y}_t}^{-1}(y_{t+1}), \Exp_{\hat{y}_t}^{-1}(y^\star) \rangle \leq \tfrac{1}{2}\left(\overline{\xi}_0(d_\NCal(\hat{y}_t, y_{t+1}))^2 + (d_\NCal(\hat{y}_t, y^\star))^2 - (d_\NCal(y_{t+1}, y^\star))^2\right). 
\end{equation*}
By the definition of an exponential map and Riemannian metric, we have
\begin{equation}\label{inequality:RCEG-key-third}
\begin{array}{lll}
d_\MCal(\hat{x}_t, x_{t+1}) & = & \|\Exp_{\hat{x}_t}^{-1}(x_{t+1})\| \overset{\textnormal{Eq.~\eqref{inequality:RCEG-opt}}}{=} \|\eta \cdot \grad_x f(\hat{x}_t,\hat{y}_t) - \Exp_{\hat{x}_t}^{-1}(x_t)\|, \\
d_\NCal(\hat{y}_t, y_{t+1}) & = & \|\Exp_{\hat{y}_t}^{-1}(y_{t+1})\| \overset{\textnormal{Eq.~\eqref{inequality:RCEG-opt}}}{=} \|\eta \cdot \grad_y f(\hat{x}_t,\hat{y}_t) + \Exp_{\hat{y}_t}^{-1}(y_t)\|. 
\end{array}
\end{equation}
Further, we see from the scheme of Algorithm~\ref{alg:RCEG} that we have
\begin{eqnarray*}
\hat{x}_t & \leftarrow & \Exp_{x_t}(-\eta \cdot \grad_x f(x_t, y_t)), \\
\hat{y}_t & \leftarrow & \Exp_{y_t}(\eta \cdot \grad_y f(x_t, y_t)). 
\end{eqnarray*}
By the definition of an exponential map, we have
\begin{equation*}
\Exp_{x_t}^{-1}(\hat{x}_t) = -\eta \cdot \grad_x f(x_t, y_t),  \qquad \Exp_{y_t}^{-1}(\hat{y}_t) = \eta \cdot \grad_y f(x_t, y_t). 
\end{equation*}
Using the definition of a parallel transport map and the above equations, we have
\begin{equation*}
\Exp_{\hat{x}_t}^{-1}(x_t) = \eta \cdot \Gamma_{x_t}^{\hat{x}_t}\grad_x f(x_t, y_t), \qquad \Exp_{\hat{y}_t}^{-1}(y_t) = - \eta \cdot \Gamma_{y_t}^{\hat{y}_t}\grad_y f(x_t, y_t)
\end{equation*}
Since $f$ is geodesically $\ell$-smooth, we have
\begin{equation*}
\begin{array}{lll}
\|\grad_x f(\hat{x}_t,\hat{y}_t) - \Gamma_{x_t}^{\hat{x}_t}\grad_x f(x_t, y_t)\| & \leq & \ell(d_\MCal(\hat{x}_t, x_t) + d_\NCal(\hat{y}_t, y_t)), \\
\|\grad_y f(\hat{x}_t,\hat{y}_t) - \Gamma_{y_t}^{\hat{y}_t}\grad_y f(x_t, y_t)\| & \leq & \ell(d_\MCal(\hat{x}_t, x_t) + d_\NCal(\hat{y}_t, y_t)). 
\end{array}
\end{equation*}
Plugging the above inequalities into Eq.~\eqref{inequality:RCEG-key-third} yields that 
\begin{equation*}
\max\left\{d_\MCal(\hat{x}_t, x_{t+1}), d_\NCal(\hat{y}_t, y_{t+1})\right\} \leq \eta\ell(d_\MCal(\hat{x}_t, x_t) + d_\NCal(\hat{y}_t, y_t)). 
\end{equation*}
Therefore, we have 
\begin{equation*}
\begin{array}{lll}
\langle \Exp_{\hat{x}_t}^{-1}(x_{t+1}), \Exp_{\hat{x}_t}^{-1}(x^\star) \rangle & \leq & \tfrac{1}{2}\left(2\overline{\xi}_0\eta^2\ell^2((d_\MCal(\hat{x}_t, x_t))^2 + (d_\NCal(\hat{y}_t, y_t))^2) + (d_\MCal(\hat{x}_t, x^\star))^2 - (d_\MCal(x_{t+1}, x^\star))^2\right), \\
\langle \Exp_{\hat{y}_t}^{-1}(y_{t+1}), \Exp_{\hat{y}_t}^{-1}(y^\star) \rangle & \leq & \tfrac{1}{2}\left(2\overline{\xi}_0\eta^2\ell^2((d_\MCal(\hat{x}_t, x_t))^2 + (d_\NCal(\hat{y}_t, y_t))^2) + (d_\NCal(\hat{y}_t, y^\star))^2 - (d_\NCal(y_{t+1}, y^\star))^2\right). 
\end{array}
\end{equation*}
Plugging the above inequalities and Eq.~\eqref{inequality:RCEG-key-second} into Eq.~\eqref{inequality:RCEG-key-first} yields the desired inequality. 
\end{proof}
The second lemma gives another key inequality that is satisfied by the iterates generated by Algorithm~\ref{alg:SRCEG}. 
\begin{lemma}\label{Lemma:SRCEG-key-inequality}
Under Assumption~\ref{Assumption:gscsc-smooth} (or Assumption~\ref{Assumption:gcc-smooth}) and the noisy model (cf. Eq.~\eqref{def:noisy-model} and~\eqref{def:noisy-bound}) and let $\{(x_t, y_t), (\hat{x}_t, \hat{y}_t)\}_{t=0}^{T-1}$ be generated by Algorithm~\ref{alg:SRCEG} with the stepsize $\eta > 0$. Then, we have
\begin{eqnarray*}
\lefteqn{\EE[f(\hat{x}_t, y^\star) - f(x^\star, \hat{y}_t)] \leq \tfrac{1}{2\eta}\EE\left[(d_\MCal(x_t, x^\star))^2 - (d_\MCal(x_{t+1}, x^\star))^2 + (d_\NCal(y_t, y^\star))^2 - (d_\NCal(y_{t+1}, y^\star))^2\right]} \\
& & + 6\overline{\xi}_0\eta\ell^2\EE\left[(d_\MCal(\hat{x}_t, x_t))^2 + (d_\NCal(\hat{y}_t, y_t))^2\right] - \tfrac{1}{2\eta}\underline{\xi}_0\EE\left[(d_\MCal(\hat{x}_t, x_t))^2 + (d_\NCal(\hat{y}_t, y_t))^2\right] \\
& & - \tfrac{\mu}{2}\EE\left[(d_\MCal(\hat{x}_t, x^\star))^2 + (d_\NCal(\hat{y}_t, y^\star))^2\right] + 3\overline{\xi}_0\eta\sigma^2, 
\end{eqnarray*}
where $(x^\star, y^\star) \in \MCal \times \NCal$ is a global saddle point of $f$. 
\end{lemma}
\begin{proof}
Using the same argument, we have ($\mu=0$ refers to geodesically convex-concave case)
\begin{eqnarray*}
\lefteqn{f(\hat{x}_t, y^\star) - f(x^\star, \hat{y}_t) =  f(\hat{x}_t, \hat{y}_t) - f(x^\star, \hat{y}_t) - (f(\hat{x}_t, \hat{y}_t) - f(\hat{x}_t, y^\star))}  \\
& \leq & - \langle \grad_x f(\hat{x}_t,\hat{y}_t), \Exp_{\hat{x}_t}^{-1}(x^\star) \rangle + \langle \grad_y f(\hat{x}_t,\hat{y}_t), \Exp_{\hat{y}_t}^{-1}(y^\star)\rangle - \tfrac{\mu}{2}(d_\MCal(\hat{x}_t, x^\star))^2 - \tfrac{\mu}{2}(d_\NCal(\hat{y}_t, y^\star))^2.
\end{eqnarray*}
Combining the arguments used in Lemma~\ref{Lemma:RCEG-key-inequality} and the scheme of Algorithm~\ref{alg:SRCEG}, we have
\begin{eqnarray*}
- \langle \hat{g}_x^t, \Exp_{\hat{x}_t}^{-1}(x^\star) \rangle & = & \tfrac{1}{\eta} (\langle \Exp_{\hat{x}_t}^{-1}(x_{t+1}), \Exp_{\hat{x}_t}^{-1}(x^\star)\rangle - \langle \Exp_{\hat{x}_t}^{-1}(x_t), \Exp_{\hat{x}_t}^{-1}(x^\star) \rangle), \\
\langle \hat{g}_y^t, \Exp_{\hat{y}_t}^{-1}(y^\star) \rangle & = & \tfrac{1}{\eta} (\langle \Exp_{\hat{y}_t}^{-1}(y_{t+1}), \Exp_{\hat{y}_t}^{-1}(y^\star)\rangle - \langle \Exp_{\hat{y}_t}^{-1}(y_t), \Exp_{\hat{y}_t}^{-1}(y^\star)\rangle). 
\end{eqnarray*}
Putting these pieces together with Eq.~\eqref{def:noisy-model} yields that 
\begin{eqnarray}\label{inequality:SRCEG-key-first}
\lefteqn{f(\hat{x}_t, y^\star) - f(x^\star, \hat{y}_t) \leq \tfrac{1}{\eta} (\langle \Exp_{\hat{x}_t}^{-1}(x_{t+1}), \Exp_{\hat{x}_t}^{-1}(x^\star)\rangle - \langle \Exp_{\hat{x}_t}^{-1}(x_t), \Exp_{\hat{x}_t}^{-1}(x^\star) \rangle)} \\
& & + \tfrac{1}{\eta} (\langle \Exp_{\hat{y}_t}^{-1}(y_{t+1}), \Exp_{\hat{y}_t}^{-1}(y^\star)\rangle - \langle \Exp_{\hat{y}_t}^{-1}(y_t), \Exp_{\hat{y}_t}^{-1}(y^\star)\rangle) - \tfrac{\mu}{2}(d_\MCal(\hat{x}_t, x^\star))^2 - \tfrac{\mu}{2}(d_\NCal(\hat{y}_t, y^\star))^2 \nonumber \\
& & + \langle \hat{\xi}_x^t, \Exp_{\hat{x}_t}^{-1}(x^\star)\rangle - \langle \hat{\xi}_y^t, \Exp_{\hat{y}_t}^{-1}(y^\star)\rangle. \nonumber
\end{eqnarray}
By the same argument as used in Lemma~\ref{Lemma:RCEG-key-inequality}, we have
\begin{equation}\label{inequality:SRCEG-key-second}
\begin{array}{lll}
- \langle \Exp_{\hat{x}_t}^{-1}(x_t), \Exp_{\hat{x}_t}^{-1}(x^\star) \rangle & \leq & - \tfrac{1}{2}\left(\underline{\xi}_0(d_\MCal(\hat{x}_t, x_t))^2 + (d_\MCal(\hat{x}_t, x^\star))^2 - (d_\MCal(x_t, x^\star))^2\right), \\
- \langle \Exp_{\hat{y}_t}^{-1}(y_t), \Exp_{\hat{y}_t}^{-1}(y^\star) \rangle & \leq & - \tfrac{1}{2}\left(\underline{\xi}_0(d_\NCal(\hat{y}_t, y_t))^2 + (d_\NCal(\hat{y}_t, y^\star))^2 - (d_\NCal(y_t, y^\star))^2\right), 
\end{array}
\end{equation}
and 
\begin{equation*}
\begin{array}{lll}
\langle \Exp_{\hat{x}_t}^{-1}(x_{t+1}), \Exp_{\hat{x}_t}^{-1}(x^\star) \rangle & \leq & \tfrac{1}{2}\left(\overline{\xi}_0\eta^2\|\hat{g}_x^t - \Gamma_{x_t}^{\hat{x}_t} g_x^t\|^2 + (d_\MCal(\hat{x}_t, x^\star))^2 - (d_\MCal(x_{t+1}, x^\star))^2\right), \\
\langle \Exp_{\hat{y}_t}^{-1}(y_{t+1}), \Exp_{\hat{y}_t}^{-1}(y^\star) \rangle & \leq & \tfrac{1}{2}\left(\overline{\xi}_0\eta^2\|\hat{g}_y^t - \Gamma_{y_t}^{\hat{y}_t} g_y^t\|^2 + (d_\NCal(\hat{y}_t, y^\star))^2 - (d_\NCal(y_{t+1}, y^\star))^2\right). 
\end{array}
\end{equation*}
Since $f$ is geodesically $\ell$-smooth and Eq.~\eqref{def:noisy-model} holds, we have
\begin{eqnarray*}
\|\hat{g}_x^t - \Gamma_{x_t}^{\hat{x}_t} g_x^t\|^2 & \leq & 3\|\hat{\xi}_x^t\|^2 + 3\|\xi_x^t\|^2 + 6\ell^2(d_\MCal(\hat{x}_t, x_t))^2 + 6\ell^2(d_\NCal(\hat{y}_t, y_t))^2, \\
\|\hat{g}_y^t - \Gamma_{y_t}^{\hat{y}_t} g_y^t\|^2 & \leq & 3\|\hat{\xi}_y^t\|^2 + 3\|\xi_y^t\|^2 + 6\ell^2(d_\MCal(\hat{x}_t, x_t))^2 + 6\ell^2(d_\NCal(\hat{y}_t, y_t))^2. 
\end{eqnarray*}
Therefore, we have 
\begin{eqnarray*}
\lefteqn{\langle \Exp_{\hat{x}_t}^{-1}(x_{t+1}), \Exp_{\hat{x}_t}^{-1}(x^\star) \rangle + \langle \Exp_{\hat{y}_t}^{-1}(y_{t+1}), \Exp_{\hat{y}_t}^{-1}(y^\star) \rangle} \\
& \leq & 6\overline{\xi}_0\eta^2\ell^2((d_\MCal(\hat{x}_t, x_t))^2 + (d_\NCal(\hat{y}_t, y_t))^2) + \tfrac{3}{2}\overline{\xi}_0\eta^2(\|\hat{\xi}_x^t\|^2 + \|\xi_x^t\|^2 + \|\hat{\xi}_y^t\|^2 + \|\xi_y^t\|^2) \\
& & + \tfrac{1}{2}\left((d_\MCal(\hat{x}_t, x^\star))^2 - (d_\MCal(x_{t+1}, x^\star))^2 + (d_\NCal(\hat{y}_t, y^\star))^2 - (d_\NCal(y_{t+1}, y^\star))^2\right). 
\end{eqnarray*}
Plugging the above inequalities and Eq.~\eqref{inequality:SRCEG-key-second} into Eq.~\eqref{inequality:SRCEG-key-first} yields that 
\begin{eqnarray*}
\lefteqn{f(\hat{x}_t, y^\star) - f(x^\star, \hat{y}_t) \leq \tfrac{1}{2\eta}\left((d_\MCal(x_t, x^\star))^2 - (d_\MCal(x_{t+1}, x^\star))^2 + (d_\NCal(y_t, y^\star))^2 - (d_\NCal(y_{t+1}, y^\star))^2\right)} \\
& & + 6\overline{\xi}_0\eta\ell^2((d_\MCal(\hat{x}_t, x_t))^2 + (d_\NCal(\hat{y}_t, y_t))^2) + \tfrac{3}{2}\overline{\xi}_0\eta(\|\hat{\xi}_x^t\|^2 + \|\xi_x^t\|^2 + \|\hat{\xi}_y^t\|^2 + \|\xi_y^t\|^2) \\
& & - \tfrac{1}{2\eta}\underline{\xi}_0\left((d_\MCal(\hat{x}_t, x_t))^2 + (d_\NCal(\hat{y}_t, y_t))^2\right) - \tfrac{\mu}{2}(d_\MCal(\hat{x}_t, x^\star))^2 - \tfrac{\mu}{2}(d_\NCal(\hat{y}_t, y^\star))^2 \\
& & + \langle \hat{\xi}_x^t, \Exp_{\hat{x}_t}^{-1}(x^\star)\rangle - \langle \hat{\xi}_y^t, \Exp_{\hat{y}_t}^{-1}(y^\star)\rangle.
\end{eqnarray*}
Taking the expectation of both sides and using Eq.~\eqref{def:noisy-bound} yields the desired inequality. 
\end{proof}

\subsection{Proof of Theorem~\ref{Thm:RCEG-SCSC}}
Since Riemannian metrics satisfy the triangle inequality, we have
\begin{equation*}
(d_\MCal(\hat{x}_t, x^\star))^2 + (d_\NCal(\hat{y}_t, y^\star))^2 \geq \tfrac{1}{2}((d_\MCal(x_t, x^\star))^2 + (d_\NCal(y_t, y^\star))^2) - (d_\MCal(\hat{x}_t, x_t))^2 + (d_\NCal(\hat{y}_t, y_t))^2.  
\end{equation*}
Plugging the above inequality into the inequality from Lemma~\ref{Lemma:RCEG-key-inequality} yields that 
\begin{eqnarray*}
\lefteqn{(d_\MCal(x_{t+1}, x^\star))^2 + (d_\NCal(y_{t+1}, y^\star))^2} \\
& \leq & \left(1 - \tfrac{\mu\eta}{2}\right)\left((d_\MCal(x_t, x^\star))^2 + (d_\NCal(y_t, y^\star))^2\right) + (4\overline{\xi}_0\eta^2\ell^2 + \mu\eta - \underline{\xi}_0)((d_\MCal(\hat{x}_t, x_t))^2 + (d_\NCal(\hat{y}_t, y_t))^2. 
\end{eqnarray*}
Since $\eta = \min\{\frac{1}{4\ell\sqrt{\tau_0}}, \frac{\underline{\xi}_0}{2\mu}\}$, we have $4\overline{\xi}_0\eta^2\ell^2 + \mu\eta - \underline{\xi}_0 \leq 0$. By the definition, we have $\tau_0 \geq 1$, $\kappa \geq 1$ and $\underline{\xi}_0 \leq 1$. This implies that 
\begin{equation*}
1 - \tfrac{\mu\eta}{2} = 1 - \min\left\{\tfrac{1}{8\kappa\sqrt{\tau_0}}, \tfrac{\underline{\xi}_0}{4}\right\} > 0. 
\end{equation*}         
Putting these pieces together yields that    
\begin{eqnarray*}
(d_\MCal(x_T, x^\star))^2 + (d_\NCal(y_T, y^\star))^2 & \leq & \left(1 - \min\left\{\tfrac{1}{8\kappa\sqrt{\tau_0}}, \tfrac{\underline{\xi}_0}{4}\right\}\right)^T(d_\MCal(x_0, x^\star))^2 + (d_\NCal(y_0, y^\star))^2 \\
& \leq & \left(1 - \min\left\{\tfrac{1}{8\kappa\sqrt{\tau_0}}, \tfrac{\underline{\xi}_0}{4}\right\}\right)^T D_0. 
\end{eqnarray*}
This completes the proof. 

\subsection{Proof of Theorem~\ref{Thm:SRCEG-SCSC}}
Since Riemannian metrics satisfy the triangle inequality, we have
\begin{equation*}
(d_\MCal(\hat{x}_t, x^\star))^2 + (d_\NCal(\hat{y}_t, y^\star))^2 \geq \tfrac{1}{2}((d_\MCal(x_t, x^\star))^2 + (d_\NCal(y_t, y^\star))^2) - (d_\MCal(\hat{x}_t, x_t))^2 + (d_\NCal(\hat{y}_t, y_t))^2.  
\end{equation*}
Plugging the above inequality into the inequality from Lemma~\ref{Lemma:SRCEG-key-inequality} yields that 
\begin{eqnarray*}
\lefteqn{\EE[f(\hat{x}_t, y^\star) - f(x^\star, \hat{y}_t)] \leq \tfrac{1}{2\eta}\EE\left[(d_\MCal(x_t, x^\star))^2 - (d_\MCal(x_{t+1}, x^\star))^2 + (d_\NCal(y_t, y^\star))^2 - (d_\NCal(y_{t+1}, y^\star))^2\right]} \\
& & + (6\overline{\xi}_0\eta\ell^2 + \tfrac{\mu}{2} - \tfrac{1}{2\eta}\underline{\xi}_0)\EE\left[(d_\MCal(\hat{x}_t, x_t))^2 + (d_\NCal(\hat{y}_t, y_t))^2\right] - \tfrac{\mu}{4}\EE\left[(d_\MCal(\hat{x}_t, x^\star))^2 + (d_\NCal(\hat{y}_t, y^\star))^2\right] + 3\overline{\xi}_0\eta\sigma^2. 
\end{eqnarray*}
Since $(x^\star, y^\star) \in \MCal \times \NCal$ is a global saddle point of $f$, we have $\EE[f(\hat{x}_t, y^\star) - f(x^\star, \hat{y}_t)] \geq 0$.  Then, we have
\begin{eqnarray*}
\lefteqn{\EE\left[(d_\MCal(x_{t+1}, x^\star))^2 + (d_\NCal(y_{t+1}, y^\star))^2\right]} \\
& \leq & \left(1 - \tfrac{\mu\eta}{2}\right)\EE\left[(d_\MCal(x_t, x^\star))^2 + (d_\NCal(y_t, y^\star))^2\right] + (12\overline{\xi}_0\eta^2\ell^2 + \mu\eta - \underline{\xi}_0)\EE\left[(d_\MCal(\hat{x}_t, x_t))^2 + (d_\NCal(\hat{y}_t, y_t))^2\right] \\ 
& & + 6\overline{\xi}_0\eta^2\sigma^2. 
\end{eqnarray*}
Since $\eta \leq \min\{\frac{1}{24\ell\sqrt{\tau_0}}, \frac{\underline{\xi}_0}{2\mu}\}$, we have $12\overline{\xi}_0\eta^2\ell^2 + \mu\eta - \underline{\xi}_0 \leq 0$. This implies that 
\begin{equation*}
\EE\left[(d_\MCal(x_{t+1}, x^\star))^2 + (d_\NCal(y_{t+1}, y^\star))^2\right] \leq \left(1 - \tfrac{\mu\eta}{2}\right)\EE\left[(d_\MCal(x_t, x^\star))^2 + (d_\NCal(y_t, y^\star))^2\right] + 6\overline{\xi}_0\eta^2\sigma^2. 
\end{equation*}
By the definition, we have $\tau_0 \geq 1$, $\kappa \geq 1$ and $\underline{\xi}_0 \leq 1$. This implies that 
\begin{equation*}
1 - \tfrac{\mu\eta}{2} \geq 1 - \min\left\{\tfrac{1}{48\kappa\sqrt{\tau_0}}, \tfrac{\underline{\xi}_0}{4}\right\} > 0. 
\end{equation*}         
By the inductive arguments, we have
\begin{eqnarray*}
\EE\left[(d_\MCal(x_T, x^\star))^2 + (d_\NCal(y_T, y^\star))^2\right] & \leq & \left(1 - \tfrac{\mu\eta}{2}\right)^T\left((d_\MCal(x_0, x^\star))^2 + (d_\NCal(y_0, y^\star))^2\right) + 6\overline{\xi}_0\eta^2\sigma^2\left(\sum_{t=0}^{T-1} \left(1 - \tfrac{\mu\eta}{2}\right)^t\right) \\
& \leq & \left(1 - \tfrac{\mu\eta}{2}\right)^T D_0 + \tfrac{12\overline{\xi}_0\eta\sigma^2}{\mu}. 
\end{eqnarray*}
Since $\eta = \min\{\frac{1}{24\ell\sqrt{\tau_0}}, \frac{\underline{\xi}_0}{2\mu}, \tfrac{2(\log(T) + \log(\mu^2 D_0 \sigma^{-2}))}{\mu T}\}$, we have 
\begin{eqnarray*}
\left(1 - \tfrac{\mu\eta}{2}\right)^T D_0 & \leq & \left(1 - \min\left\{\tfrac{1}{48\kappa\sqrt{\tau_0}}, \tfrac{\underline{\xi}_0}{4}\right\}\right)^T D_0 + \left(1 - \tfrac{\log(\mu^2 D_0 \sigma^{-2}T)}{T}\right)^T D_0 \\ 
&  \overset{1 + x \leq e^x}{\leq} & \left(1 - \min\left\{\tfrac{1}{48\kappa\sqrt{\tau_0}}, \tfrac{\underline{\xi}_0}{4}\right\}\right)^T D_0 + \tfrac{\sigma^2}{\mu^2 T}, 
\end{eqnarray*}
and 
\begin{equation*}
\tfrac{12\overline{\xi}_0\eta\sigma^2}{\mu} \leq \tfrac{24\overline{\xi}_0\sigma^2}{\mu^2 T}\log\left(\tfrac{\mu^2 D_0 T}{\sigma^2}\right). 
\end{equation*}
Putting these pieces together yields that 
\begin{equation*}
\EE\left[(d_\MCal(x_T, x^\star))^2 + (d_\NCal(y_T, y^\star))^2\right] \leq \left(1 - \min\left\{\tfrac{1}{48\kappa\sqrt{\tau_0}}, \tfrac{\underline{\xi}_0}{4}\right\}\right)^T D_0 + \tfrac{\sigma^2}{\mu^2 T} + \tfrac{24\overline{\xi}_0\sigma^2}{\mu^2 T}\log\left(\tfrac{\mu^2 D_0 T}{\sigma^2}\right). 
\end{equation*}
This completes the proof. 

\subsection{Proof of Theorem~\ref{Thm:SRCEG-CC}}
By the inductive formulas of $\bar{x}_{t+1} = \Exp_{\bar{x}_t}(\tfrac{1}{t+1} \cdot \Exp_{\bar{x}_t}^{-1}(\hat{x}_t))$ and $\bar{y}_{t+1} = \Exp_{\bar{y}_t}(\tfrac{1}{t+1} \cdot \Exp_{\bar{y}_t}^{-1}(\hat{y}_t))$ and using~\citet[Lemma~C.2]{Zhang-2022-Minimax}, we have
\begin{equation*}
f(\bar{x}_T, y^\star) - f(x^\star, \bar{y}_T) \leq \tfrac{1}{T}\left(\sum_{t=0}^{T-1} f(\hat{x}_t, y^\star) - f(x^\star, \hat{y}_t)\right). 
\end{equation*}
Plugging the above inequality into the inequality from Lemma~\ref{Lemma:SRCEG-key-inequality} yields that (recall that $\mu = 0$ in geodesically convex-concave setting here)
\begin{eqnarray*}
\lefteqn{\EE[f(\bar{x}_T, y^\star) - f(x^\star, \bar{y}_T)] \leq \tfrac{1}{2\eta T}\left((d_\MCal(x_0, x^\star))^2 + (d_\NCal(y_0, y^\star))^2\right)} \\
& & + \tfrac{1}{T}\left(6\overline{\xi}_0\eta\ell^2 - \tfrac{1}{2\eta}\underline{\xi}_0\right)\left(\sum_{t=0}^{T-1} \EE\left[(d_\MCal(\hat{x}_t, x_t))^2 + (d_\NCal(\hat{y}_t, y_t))^2\right]\right) + 3\overline{\xi}_0\eta\sigma^2.  
\end{eqnarray*}
Since $\eta \leq \frac{1}{4\ell\sqrt{\tau_0}}$, we have $6\overline{\xi}_0\eta\ell^2 - \tfrac{1}{2\eta}\underline{\xi}_0 \leq 0$. Then, this together with $(d_\MCal(x_0, x^\star))^2 + (d_\NCal(y_0, y^\star))^2 \leq D_0$ implies that 
\begin{equation*}
\EE[f(\bar{x}_T, y^\star) - f(x^\star, \bar{y}_T)] \leq \tfrac{D_0}{2\eta T} + 3\overline{\xi}_0\eta\sigma^2. 
\end{equation*}
Since $\eta = \min\{\frac{1}{4\ell\sqrt{\tau_0}}, \tfrac{1}{\sigma}\sqrt{\tfrac{D_0}{\overline{\xi}_0 T}}\}$, we have 
\begin{equation*}
\tfrac{D_0}{2\eta T} \leq \tfrac{2\ell D_0\sqrt{\tau_0}}{T} + \tfrac{\sigma}{2}\sqrt{\tfrac{\overline{\xi}_0 D_0}{T}}, 
\end{equation*}
and 
\begin{equation*}
3\overline{\xi}_0\eta\sigma^2 \leq 3\sigma\sqrt{\tfrac{\overline{\xi}_0 D_0}{T}}. 
\end{equation*}
Putting these pieces together yields that 
\begin{equation*}
\EE[f(\bar{x}_T, y^\star) - f(x^\star, \bar{y}_T)] \leq \tfrac{2\ell D_0\sqrt{\tau_0}}{T} + \tfrac{7\sigma}{2}\sqrt{\tfrac{\overline{\xi}_0 D_0}{T}}. 
\end{equation*}
This completes the proof. 

\section{Missing Proofs for Riemannian Gradient Descent Ascent}
In this section, we present some technical lemmas for analyzing the convergence property of Algorithm~\ref{alg:RGDA} and~\ref{alg:SRGDA}. We also give the proofs of Theorem~\ref{Thm:RGDA-SCSC},~\ref{Thm:RGDA-CC},~\ref{Thm:SRGDA-SCSC} and~\ref{Thm:SRGDA-CC}. 

\subsection{Technical lemmas}
We provide two technical lemmas for analyzing Algorithm~\ref{alg:RGDA} and~\ref{alg:SRGDA} respectively. The first lemma gives a key inequality that is satisfied by the iterates generated by Algorithm~\ref{alg:RGDA}. 
\begin{lemma}\label{Lemma:RGDA-key-inequality}
Under Assumption~\ref{Assumption:gscsc-nonsmooth} (or Assumption~\ref{Assumption:gcc-nonsmooth}) and let $\{(x_t, y_t)\}_{t=0}^{T-1}$ be generated by Algorithm~\ref{alg:RGDA} with the stepsize $\eta_t > 0$. Then, we have
\begin{eqnarray*}
\lefteqn{f(x_t, y^\star) - f(x^\star, y_t) \leq \tfrac{1}{2\eta_t}\left((d_\MCal(x_t, x^\star))^2 - (d_\MCal(x_{t+1}, x^\star))^2\right)} \\
& & + \tfrac{1}{2\eta_t}\left((d_\NCal(y_t, y^\star))^2 - (d_\NCal(y_{t+1}, y^\star))^2\right) - \tfrac{\mu}{2}(d_\MCal(x_t, x^\star))^2 - \tfrac{\mu}{2}(d_\NCal(y_t, y^\star))^2 + \overline{\xi}_0 \eta_t L^2, 
\end{eqnarray*}
where $(x^\star, y^\star) \in \MCal \times \NCal$ is a global saddle point of $f$. 
\end{lemma}
\begin{proof}
Since $f$ is geodesically strongly-concave-strongly-concave with the modulus $\mu \geq 0$ (here $\mu = 0$ means that $f$ is geodesically concave-concave), we have
\begin{eqnarray*}
\lefteqn{f(x_t, y^\star) - f(x^\star, y_t) =  f(x_t, y_t) - f(x^\star, y_t) - (f(x_t, y_t) - f(x_t, y^\star))}  \\
& \leq & - \langle \subg_x f(x_t, y_t), \Exp_{x_t}^{-1}(x^\star) \rangle + \langle \subg_y f(x_t, y_t), \Exp_{y_t}^{-1}(y^\star)\rangle - \tfrac{\mu}{2}(d_\MCal(x_t, x^\star))^2 - \tfrac{\mu}{2}(d_\NCal(y_t, y^\star))^2.
\end{eqnarray*}
Recalling also from the scheme of Algorithm~\ref{alg:RGDA} that we have
\begin{eqnarray*}
x_{t+1} & \leftarrow & \Exp_{x_t}(-\eta_t \cdot \subg_x f(x_t, y_t)), \\
y_{t+1} & \leftarrow & \Exp_{y_t}(\eta_t \cdot \subg_y f(x_t, y_t)). 
\end{eqnarray*}
By the definition of an exponential map, we have
\begin{equation}\label{inequality:RGDA-opt}
\begin{array}{lcl}
\Exp_{x_t}^{-1}(x_{t+1}) & = & -\eta_t \cdot \subg_x f(x_t,y_t),  \\
\Exp_{y_t}^{-1}(y_{t+1}) & = & \eta_t \cdot \subg_y f(x_t, y_t). 
\end{array}
\end{equation}
This implies that 
\begin{eqnarray*}
- \langle \subg_x f(x_t, y_t), \Exp_{x_t}^{-1}(x^\star) \rangle & = & \tfrac{1}{\eta_t} \langle \Exp_{x_t}^{-1}(x_{t+1}), \Exp_{x_t}^{-1}(x^\star)\rangle, \\
\langle \subg_y f(x_t, y_t), \Exp_{y_t}^{-1}(y^\star) \rangle & = & \tfrac{1}{\eta_t} \langle \Exp_{y_t}^{-1}(y_{t+1}), \Exp_{y_t}^{-1}(y^\star)\rangle. 
\end{eqnarray*}
Putting these pieces together yields that 
\begin{eqnarray}\label{inequality:RGDA-key-first}
\lefteqn{f(x_t, y^\star) - f(x^\star, y_t) \leq \tfrac{1}{\eta_t}\langle \Exp_{x_t}^{-1}(x_{t+1}), \Exp_{x_t}^{-1}(x^\star)\rangle} \\
& & + \tfrac{1}{\eta_t}\langle \Exp_{y_t}^{-1}(y_{t+1}), \Exp_{y_t}^{-1}(y^\star)\rangle - \tfrac{\mu}{2}(d_\MCal(x_t, x^\star))^2 - \tfrac{\mu}{2}(d_\NCal(y_t, y^\star))^2.  \nonumber
\end{eqnarray}
It suffices to bound the terms in the right-hand side of Eq.~\eqref{inequality:RGDA-key-first} by leveraging the celebrated comparison inequalities on Riemannian manifold with lower bounded sectional curvature (see Proposition~\ref{Prop:key-inequality-LB}). More specifically, we define the constants using $\overline{\xi}(\cdot, \cdot)$ and $\underline{\xi}(\cdot, \cdot)$ from Proposition~\ref{Prop:key-inequality-LB} as follows, 
\begin{equation*}
\overline{\xi}_0 = \overline{\xi}(\kappa_{\min}, D). 
\end{equation*}
By Proposition~\ref{Prop:key-inequality-LB} and using that $\max\{d_\MCal(x_t, x^\star), d_\NCal(y_t, y^\star)\} \leq D$, we have
\begin{eqnarray*}
\langle \Exp_{x_t}^{-1}(x_{t+1}), \Exp_{x_t}^{-1}(x^\star) \rangle & \leq & \tfrac{1}{2}\left(\overline{\xi}_0(d_\MCal(x_t, x_{t+1}))^2 + (d_\MCal(x_t, x^\star))^2 - (d_\MCal(x_{t+1}, x^\star))^2\right), \\
\langle \Exp_{y_t}^{-1}(y_{t+1}), \Exp_{y_t}^{-1}(y^\star) \rangle & \leq & \tfrac{1}{2}\left(\overline{\xi}_0(d_\NCal(y_t, y_{t+1}))^2 + (d_\NCal(y_t, y^\star))^2 - (d_\NCal(y_{t+1}, y^\star))^2\right). 
\end{eqnarray*}
Since $f$ is geodesically $L$-Lipschitz, we have 
\begin{equation*}
\|\subg_x f(x_t,y_t)\| \leq L, \quad \|\subg_y f(x_t,y_t)\| \leq L. 
\end{equation*}
By the definition of an exponential map and Riemannian metric, we have
\begin{equation*}
\begin{array}{lll}
d_\MCal(x_t, x_{t+1}) & = & \|\Exp_{x_t}^{-1}(x_{t+1})\| \overset{\textnormal{Eq.~\eqref{inequality:RGDA-opt}}}{=} \|\eta_t \cdot \subg_x f(x_t, y_t)\| \leq \eta_t L,  \\
d_\NCal(y_t, y_{t+1}) & = & \|\Exp_{y_t}^{-1}(y_{t+1})\| \overset{\textnormal{Eq.~\eqref{inequality:RGDA-opt}}}{=} \|\eta_t \cdot \subg_y f(x_t, y_t)\| \leq \eta_t L. 
\end{array}
\end{equation*}
Putting these pieces together yields that 
\begin{eqnarray*}
\langle \Exp_{x_t}^{-1}(x_{t+1}), \Exp_{x_t}^{-1}(x^\star) \rangle & \leq & \tfrac{1}{2}\left(\overline{\xi}_0 \eta_t^2 L^2 + (d_\MCal(x_t, x^\star))^2 - (d_\MCal(x_{t+1}, x^\star))^2\right), \\
\langle \Exp_{y_t}^{-1}(y_{t+1}), \Exp_{y_t}^{-1}(y^\star) \rangle & \leq & \tfrac{1}{2}\left(\overline{\xi}_0 \eta_t^2 L^2 + (d_\NCal(y_t, y^\star))^2 - (d_\NCal(y_{t+1}, y^\star))^2\right). 
\end{eqnarray*}
Plugging the above inequalities into Eq.~\eqref{inequality:RGDA-key-first} yields the desired inequality. 
\end{proof}
The second lemma gives another key inequality that is satisfied by the iterates generated by Algorithm~\ref{alg:SRGDA}. 
\begin{lemma}\label{Lemma:SRGDA-key-inequality}
Under Assumption~\ref{Assumption:gscsc-nonsmooth} (or Assumption~\ref{Assumption:gcc-nonsmooth}) and the noisy model (cf. Eq.~\eqref{def:noisy-model-RGDA} and~\eqref{def:noisy-bound-RGDA}) and let $\{(x_t, y_t)\}_{t=0}^{T-1}$ be generated by Algorithm~\ref{alg:SRGDA} with the stepsize $\eta_t > 0$. Then, we have
\begin{eqnarray*}
\lefteqn{\EE[f(x_t, y^\star) - f(x^\star, y_t)] \leq \tfrac{1}{2\eta_t}\EE\left[(d_\MCal(x_t, x^\star))^2 - (d_\MCal(x_{t+1}, x^\star))^2\right]} \\
& & + \tfrac{1}{2\eta_t}\EE\left[(d_\NCal(y_t, y^\star))^2 - (d_\NCal(y_{t+1}, y^\star))^2\right] - \tfrac{\mu}{2}\EE\left[(d_\MCal(x_t, x^\star))^2 + (d_\NCal(y_t, y^\star))^2\right] + 2\overline{\xi}_0\eta_t(L^2 + \sigma^2), 
\end{eqnarray*}
where $(x^\star, y^\star) \in \MCal \times \NCal$ is a global saddle point of $f$. 
\end{lemma}
\begin{proof}
Using the same argument, we have ($\mu=0$ refers to geodesically convex-concave case)
\begin{eqnarray*}
\lefteqn{f(x_t, y^\star) - f(x^\star, y_t) =  f(x_t, y_t) - f(x^\star, y_t) - (f(x_t, y_t) - f(x_t, y^\star))}  \\
& \leq & - \langle \subg_x f(x_t, y_t), \Exp_{x_t}^{-1}(x^\star) \rangle + \langle \subg_y f(x_t, y_t), \Exp_{y_t}^{-1}(y^\star)\rangle - \tfrac{\mu}{2}(d_\MCal(x_t, x^\star))^2 - \tfrac{\mu}{2}(d_\NCal(y_t, y^\star))^2.
\end{eqnarray*}
Combining the arguments used in Lemma~\ref{Lemma:RGDA-key-inequality} and the scheme of Algorithm~\ref{alg:SRCEG}, we have
\begin{eqnarray*}
- \langle g_x^t, \Exp_{x_t}^{-1}(x^\star) \rangle & = & \tfrac{1}{\eta_t}\langle \Exp_{x_t}^{-1}(x_{t+1}), \Exp_{x_t}^{-1}(x^\star)\rangle, \\
\langle g_y^t, \Exp_{y_t}^{-1}(y^\star) \rangle & = & \tfrac{1}{\eta_t}\langle \Exp_{y_t}^{-1}(y_{t+1}), \Exp_{y_t}^{-1}(y^\star)\rangle. 
\end{eqnarray*}
Putting these pieces together with Eq.~\eqref{def:noisy-model-RGDA} yields that 
\begin{eqnarray}\label{inequality:SRGDA-key-first}
\lefteqn{f(x_t, y^\star) - f(x^\star, y_t) \leq \tfrac{1}{\eta_t}\langle \Exp_{x_t}^{-1}(x_{t+1}), \Exp_{x_t}^{-1}(x^\star)} \\
& & + \tfrac{1}{\eta_t}\langle \Exp_{y_t}^{-1}(y_{t+1}), \Exp_{y_t}^{-1}(y^\star)\rangle - \tfrac{\mu}{2}(d_\MCal(x_t, x^\star))^2 - \tfrac{\mu}{2}(d_\NCal(y_t, y^\star))^2 + \langle \xi_x^t, \Exp_{x_t}^{-1}(x^\star)\rangle - \langle \xi_y^t, \Exp_{y_t}^{-1}(y^\star)\rangle. \nonumber
\end{eqnarray}
By the same argument as used in Lemma~\ref{Lemma:RGDA-key-inequality} and Eq.~\eqref{def:noisy-model-RGDA}, we have
\begin{eqnarray*}
\langle \Exp_{x_t}^{-1}(x_{t+1}), \Exp_{x_t}^{-1}(x^\star) \rangle & \leq & \tfrac{1}{2}\left(\overline{\xi}_0(d_\MCal(x_t, x_{t+1}))^2 + (d_\MCal(x_t, x^\star))^2 - (d_\MCal(x_{t+1}, x^\star))^2\right), \\
\langle \Exp_{y_t}^{-1}(y_{t+1}), \Exp_{y_t}^{-1}(y^\star) \rangle & \leq & \tfrac{1}{2}\left(\overline{\xi}_0(d_\NCal(y_t, y_{t+1}))^2 + (d_\NCal(y_t, y^\star))^2 - (d_\NCal(y_{t+1}, y^\star))^2\right), 
\end{eqnarray*}
and 
\begin{equation*}
\begin{array}{lll}
d_\MCal(x_t, x_{t+1}) & = & \|\Exp_{x_t}^{-1}(x_{t+1})\| \ = \ \|\eta_t \cdot g_x^t\| \ \leq \ \eta_t (L + \|\xi_x^t\|),  \\
d_\NCal(y_t, y_{t+1}) & = & \|\Exp_{y_t}^{-1}(y_{t+1})\| \ = \ \|\eta_t \cdot g_y^t\| \ \leq \ \eta_t (L + \|\xi_y^t\|). 
\end{array}
\end{equation*}
Therefore, we have 
\begin{eqnarray*}
\lefteqn{\langle \Exp_{x_t}^{-1}(x_{t+1}), \Exp_{x_t}^{-1}(x^\star) \rangle + \langle \Exp_{y_t}^{-1}(y_{t+1}), \Exp_{y_t}^{-1}(y^\star) \rangle} \\
& \leq & \tfrac{1}{2}\overline{\xi}_0\eta_t^2(4L^2 + 2\|\xi_x^t\|^2 + 2\|\xi_y^t\|^2) + \tfrac{1}{2}\left((d_\MCal(x_t, x^\star))^2 - (d_\MCal(x_{t+1}, x^\star))^2 + (d_\NCal(y_t, y^\star))^2 - (d_\NCal(y_{t+1}, y^\star))^2\right). 
\end{eqnarray*}
Plugging the above inequalities into Eq.~\eqref{inequality:SRGDA-key-first} yields that 
\begin{eqnarray*}
\lefteqn{f(x_t, y^\star) - f(x^\star, y_t) \leq \tfrac{1}{2\eta_t}\left((d_\MCal(x_t, x^\star))^2 - (d_\MCal(x_{t+1}, x^\star))^2 + (d_\NCal(y_t, y^\star))^2 - (d_\NCal(y_{t+1}, y^\star))^2\right)} \\
& & + \overline{\xi}_0\eta_t(2L^2 + \|\xi_x^t\|^2 + \|\xi_y^t\|^2) - \tfrac{\mu}{2}(d_\MCal(x_t, x^\star))^2 - \tfrac{\mu}{2}(d_\NCal(y_t, y^\star))^2 + \langle \xi_x^t, \Exp_{x_t}^{-1}(x^\star)\rangle - \langle \xi_y^t, \Exp_{y_t}^{-1}(y^\star)\rangle.
\end{eqnarray*}
Taking the expectation of both sides and using Eq.~\eqref{def:noisy-bound-RGDA} yields the desired inequality. 
\end{proof}

\subsection{Proof of Theorem~\ref{Thm:RGDA-SCSC}}
Since $(x^\star, y^\star) \in \MCal \times \NCal$ is a global saddle point of $f$, we have $f(x_t, y^\star) - f(x^\star, y_t) \geq 0$. Plugging this inequality into the inequality from Lemma~\ref{Lemma:RGDA-key-inequality} yields that 
\begin{equation*}
(d_\MCal(x_{t+1}, x^\star))^2 + (d_\NCal(y_{t+1}, y^\star))^2 \leq (1 - \mu\eta_t)\left((d_\MCal(x_t, x^\star))^2 + (d_\NCal(y_t, y^\star))^2\right) + 2\overline{\xi}_0\eta_t^2 L^2. 
\end{equation*}
Since $\eta_t = \frac{1}{\mu}\min\{1, \frac{2}{t}\}$, we have 
\begin{equation*}
(d_\MCal(x_{t+1}, x^\star))^2 + (d_\NCal(y_{t+1}, y^\star))^2 \leq (1 - \tfrac{2}{t})\left((d_\MCal(x_t, x^\star))^2 + (d_\NCal(y_t, y^\star))^2\right) + \tfrac{8\overline{\xi}_0 L^2}{\mu^2 t^2}, \quad \textnormal{for all } t \geq 2. 
\end{equation*}         
Letting $\{b_t\}_{t \geq 1}$ be a nonnegative sequence such that $a_{t+1} \leq (1 - \frac{P}{t})a_t + \frac{Q}{t^2}$ where $P > 1$ and $Q > 0$. Then,~\citet{Chung-1954-Stochastic} proved that $a_t \leq \frac{Q}{P-1}\frac{1}{t}$. Therefore, we have
\begin{equation*}
(d_\MCal(x_t, x^\star))^2 + (d_\NCal(y_t, y^\star))^2 \leq \tfrac{8\overline{\xi}_0 L^2}{\mu^2 t}, \quad \textnormal{for all } t \geq 2. 
\end{equation*}
This completes the proof. 

\subsection{Proof of Theorem~\ref{Thm:RGDA-CC}}
By the inductive formulas of $\bar{x}_{t+1} = \Exp_{\bar{x}_t}(\tfrac{1}{t+1} \cdot \Exp_{\bar{x}_t}^{-1}(x_t))$ and $\bar{y}_{t+1} = \Exp_{\bar{y}_t}(\tfrac{1}{t+1} \cdot \Exp_{\bar{y}_t}^{-1}(y_t))$ and using~\citet[Lemma~C.2]{Zhang-2022-Minimax}, we have
\begin{equation*}
f(\bar{x}_T, y^\star) - f(x^\star, \bar{y}_T) \leq \tfrac{1}{T}\left(\sum_{t=0}^{T-1} f(x_t, y^\star) - f(x^\star, y_t)\right). 
\end{equation*}
Plugging the above inequality into the inequality from Lemma~\ref{Lemma:RGDA-key-inequality} yields that (recall that $\mu = 0$ in geodesically convex-concave setting and $\eta_t = \eta = \tfrac{1}{L}\sqrt{\tfrac{D_0}{2\overline{\xi}_0 T}}$)
\begin{equation*}
f(\bar{x}_T, y^\star) - f(x^\star, \bar{y}_T) \leq \tfrac{1}{2\eta T}\left((d_\MCal(x_0, x^\star))^2 + (d_\NCal(y_0, y^\star))^2\right) + \overline{\xi}_0 \eta L^2.  
\end{equation*}
This together with $(d_\MCal(x_0, x^\star))^2 + (d_\NCal(y_0, y^\star))^2 \leq D_0$ implies that 
\begin{equation*}
f(\bar{x}_T, y^\star) - f(x^\star, \bar{y}_T) \leq \tfrac{D_0}{2\eta T} + \overline{\xi}_0\eta L^2. 
\end{equation*}
Since $\eta = \tfrac{1}{L}\sqrt{\tfrac{D_0}{2\overline{\xi}_0 T}}$, we have 
\begin{equation*}
f(\bar{x}_T, y^\star) - f(x^\star, \bar{y}_T) \leq L\sqrt{\tfrac{2\overline{\xi}_0 D_0}{T}}. 
\end{equation*}
This completes the proof. 

\subsection{Proof of Theorem~\ref{Thm:SRGDA-SCSC}}
Since $(x^\star, y^\star) \in \MCal \times \NCal$ is a global saddle point of $f$, we have $\EE[f(x_t, y^\star) - f(x^\star, y_t)] \geq 0$. Plugging this inequality into the inequality from Lemma~\ref{Lemma:SRGDA-key-inequality} yields that 
\begin{equation*}
\EE\left[(d_\MCal(x_{t+1}, x^\star))^2 + (d_\NCal(y_{t+1}, y^\star))^2\right] \leq (1 - \mu\eta_t)\EE\left[(d_\MCal(x_t, x^\star))^2 + (d_\NCal(y_t, y^\star))^2\right] + 4\overline{\xi}_0\eta_t^2(L^2 + \sigma^2). 
\end{equation*}
Since $\eta_t = \frac{1}{\mu}\min\{1, \frac{2}{t}\}$, we have 
\begin{equation*}
\EE\left[(d_\MCal(x_{t+1}, x^\star))^2 + (d_\NCal(y_{t+1}, y^\star))^2\right] \leq (1 - \tfrac{2}{t})\EE\left[(d_\MCal(x_t, x^\star))^2 + (d_\NCal(y_t, y^\star))^2\right] + \tfrac{16\overline{\xi}_0(L^2 + \sigma^2)}{\mu^2 t^2}, \quad \textnormal{for all } t \geq 2. 
\end{equation*}         
Applying the same argument as used in Theorem~\ref{Thm:RGDA-SCSC}, we have
\begin{equation*}
(d_\MCal(x_t, x^\star))^2 + (d_\NCal(y_t, y^\star))^2 \leq \tfrac{16\overline{\xi}_0(L^2 + \sigma^2)}{\mu^2 t}, \quad \textnormal{for all } t \geq 2. 
\end{equation*}
This completes the proof. 

\subsection{Proof of Theorem~\ref{Thm:SRGDA-CC}}
Using the same argument, we have
\begin{equation*}
f(\bar{x}_T, y^\star) - f(x^\star, \bar{y}_T) \leq \tfrac{1}{T}\left(\sum_{t=0}^{T-1} f(x_t, y^\star) - f(x^\star, y_t)\right). 
\end{equation*}
Plugging the above inequality into the inequality from Lemma~\ref{Lemma:SRGDA-key-inequality} yields that (recall that $\mu = 0$ in geodesically convex-concave setting and $\eta_t = \eta = \tfrac{1}{2}\sqrt{\tfrac{D_0}{\overline{\xi}_0(L^2 + \sigma^2)T}}$)
\begin{equation*}
\EE[f(\bar{x}_T, y^\star) - f(x^\star, \bar{y}_T)] \leq \tfrac{1}{2\eta T}\left((d_\MCal(x_0, x^\star))^2 + (d_\NCal(y_0, y^\star))^2\right) + 2\overline{\xi}_0 \eta(L^2 + \sigma^2).  
\end{equation*}
This together with $(d_\MCal(x_0, x^\star))^2 + (d_\NCal(y_0, y^\star))^2 \leq D_0$ implies that 
\begin{equation*}
\EE[f(\bar{x}_T, y^\star) - f(x^\star, \bar{y}_T)] \leq \tfrac{D_0}{2\eta T} + 2\overline{\xi}_0\eta(L^2 + \sigma^2). 
\end{equation*}
Since $\eta = \tfrac{1}{2}\sqrt{\tfrac{D_0}{\overline{\xi}_0(L^2 + \sigma^2)T}}$, we have 
\begin{equation*}
f(\bar{x}_T, y^\star) - f(x^\star, \bar{y}_T) \leq 2\sqrt{\tfrac{\overline{\xi}_0(L^2 + \sigma^2)D_0}{T}}. 
\end{equation*}
This completes the proof. 

\section{Additional Experimental Results}\label{sec:appendix-exp}
We present some additional experimental results for the effect of different choices of $\alpha$ as well the effect of different choices of $\eta$ for for RCEG. In our experiment here, we set $n=40$ consistently. 

Figure~\ref{fig:exp-deterministic-appendix} presents the performance of RCEG when $\alpha=2.0$. We observe that the results are similar to that summarized in Figure~\ref{fig:exp-deterministic}. In particular, the last iterate of RCEG consistently achieves the linearly convergence to an optimal solution in all the settings. In contrast, the average iterate of RCEG converges much slower than the last iterate of RCEG. Figure~\ref{fig:exp-deterministic-stepsize-appendix} summarizes the effect of different choices of $\eta$ in RCEG. We observe that setting $\eta$ as a relatively larger value will speed up the convergence to an optimal solution while all of the choices here lead to the linear convergence. This suggests that the choice of stepsize $\eta$ in RCEG can be aggressive in practice. 
\begin{figure*}[t!]
\centering
\includegraphics[width=0.32\textwidth]{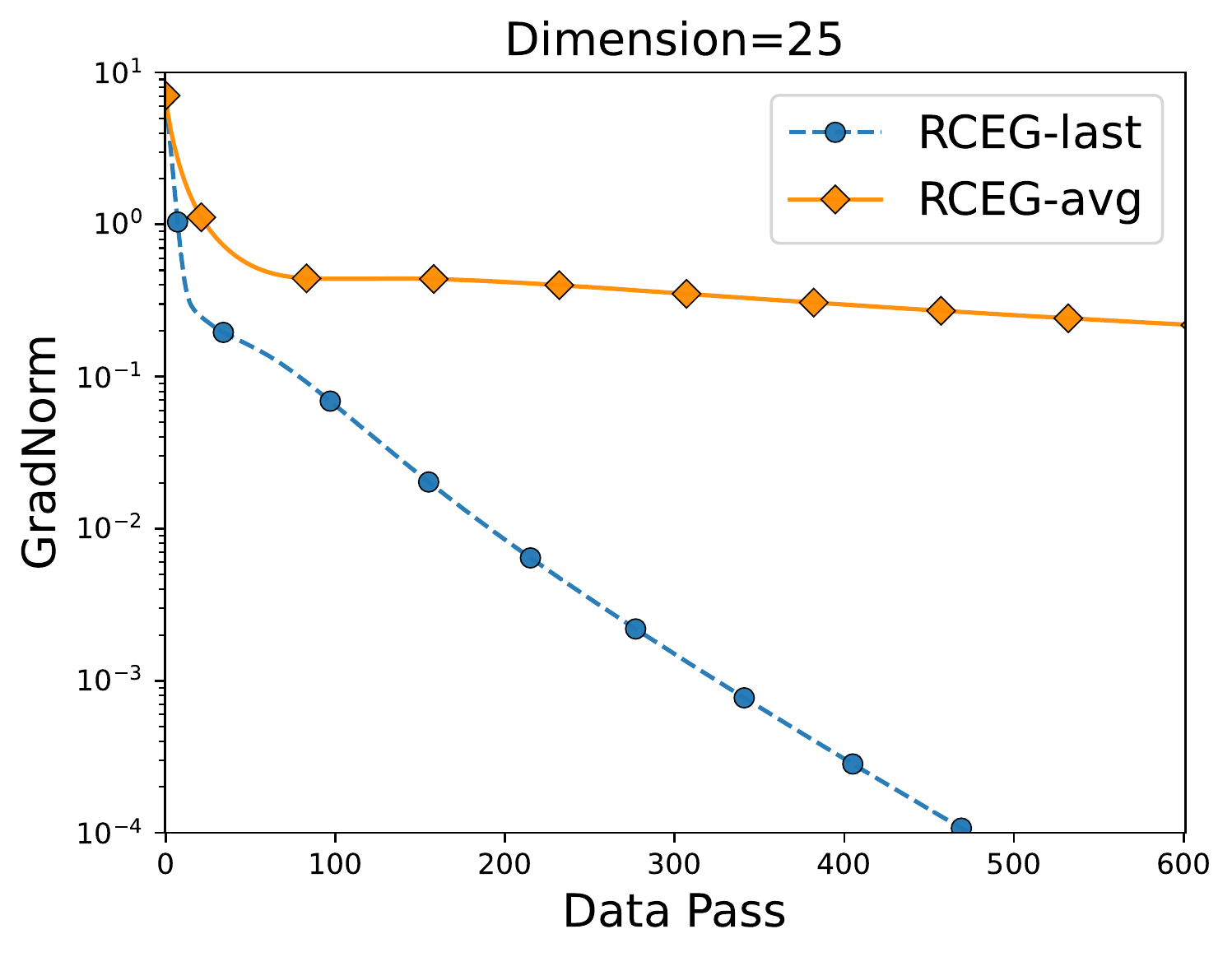}
\includegraphics[width=0.32\textwidth]{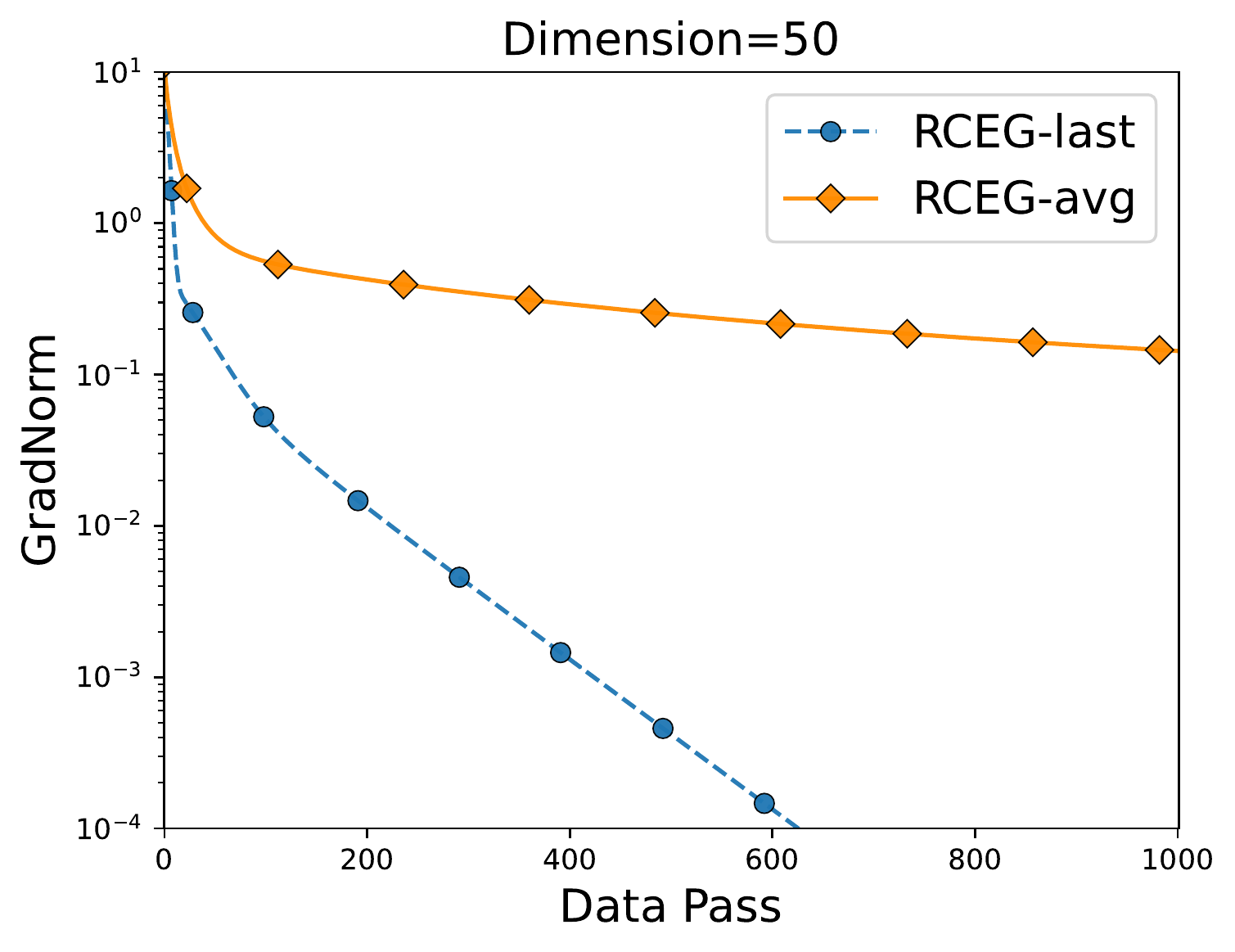}
\includegraphics[width=0.32\textwidth]{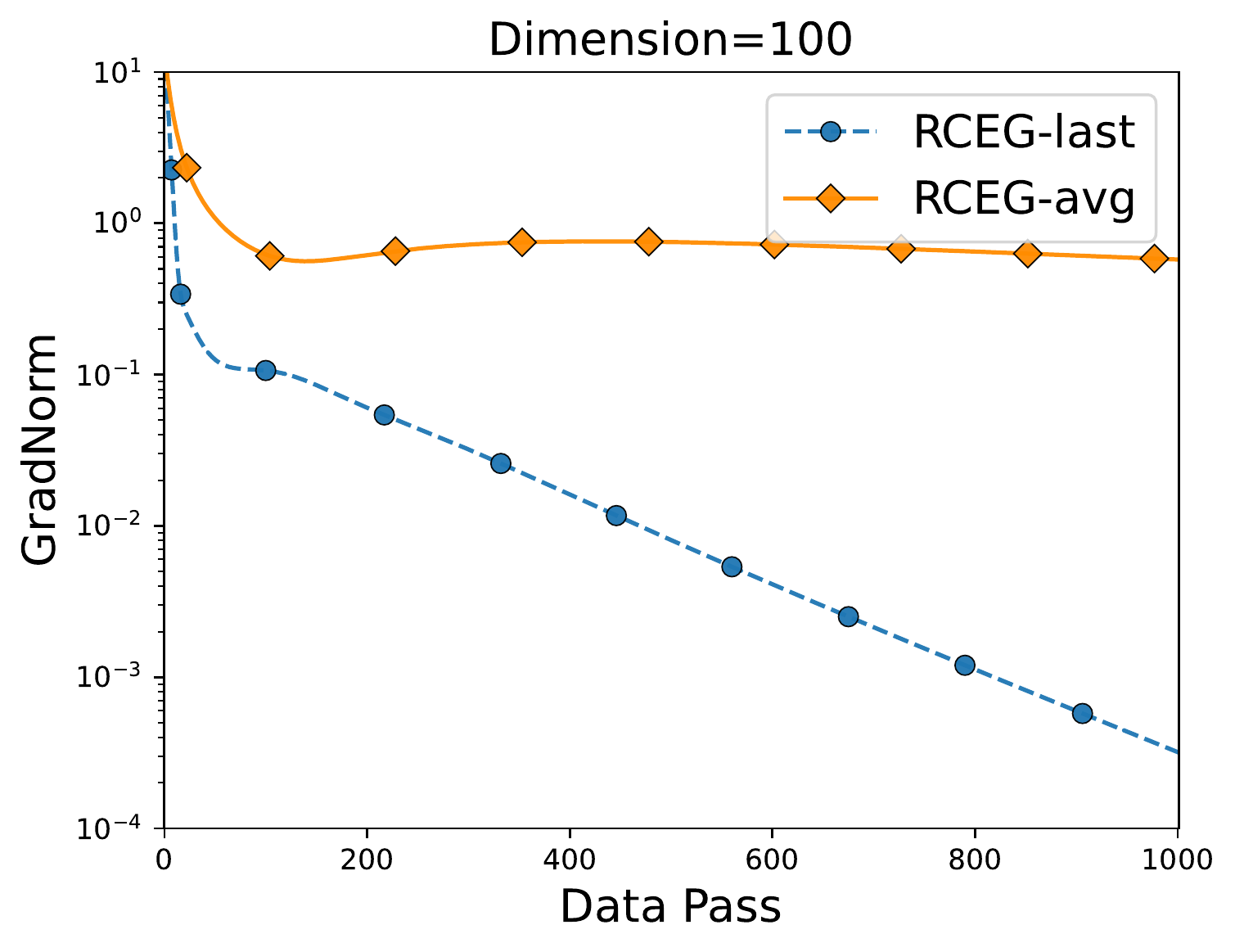}
\caption{\footnotesize{Comparison of last iterate (RCEG-last) and time-average iterate (RCEG-avg) for solving the RPCA problem when $\alpha=2.0$. The horizontal axis represents the number of data passes and the vertical axis represents gradient norm.}}
\label{fig:exp-deterministic-appendix}
\end{figure*}
\begin{figure*}[!t]
\centering
\includegraphics[width=0.45\textwidth]{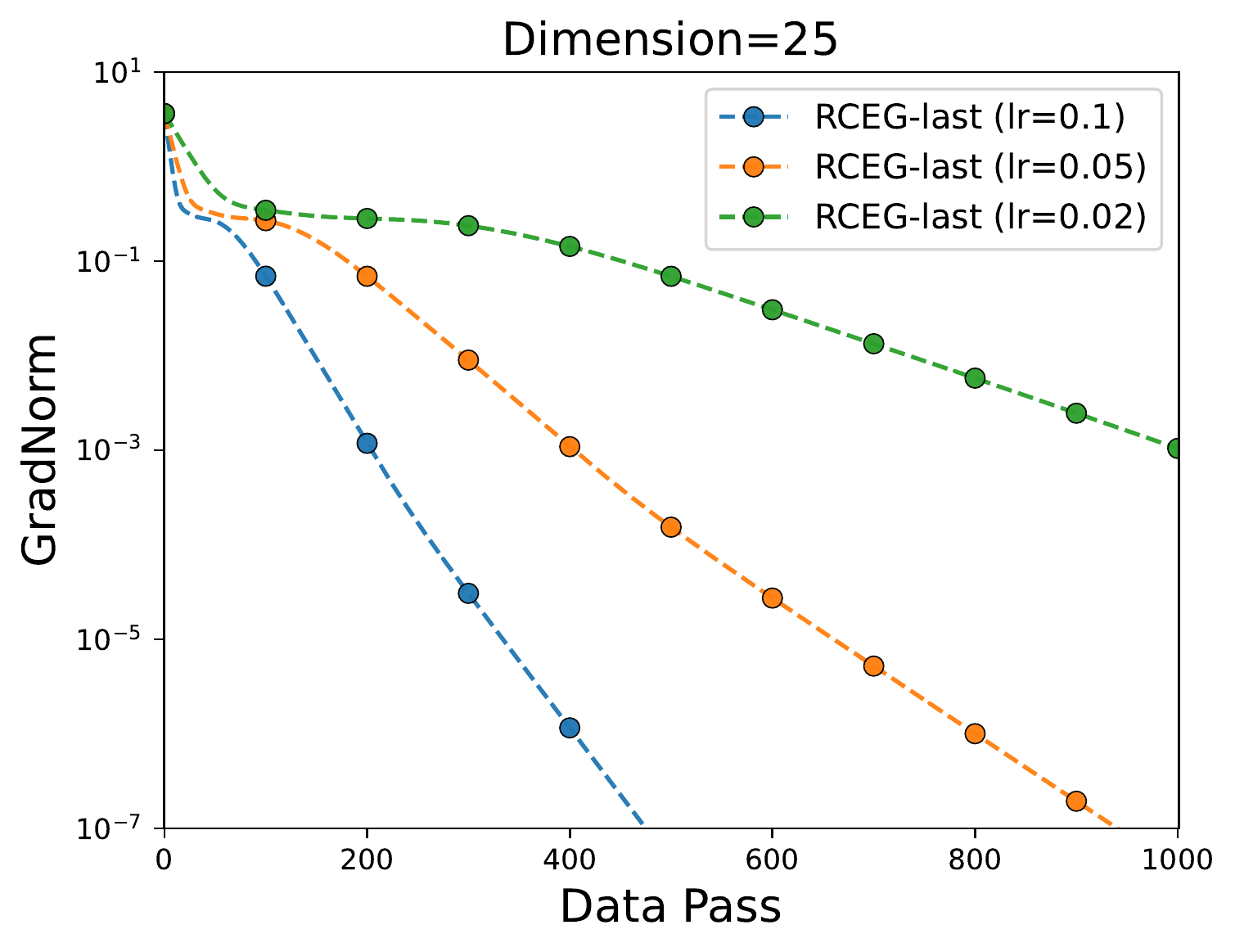}
\includegraphics[width=0.45\textwidth]{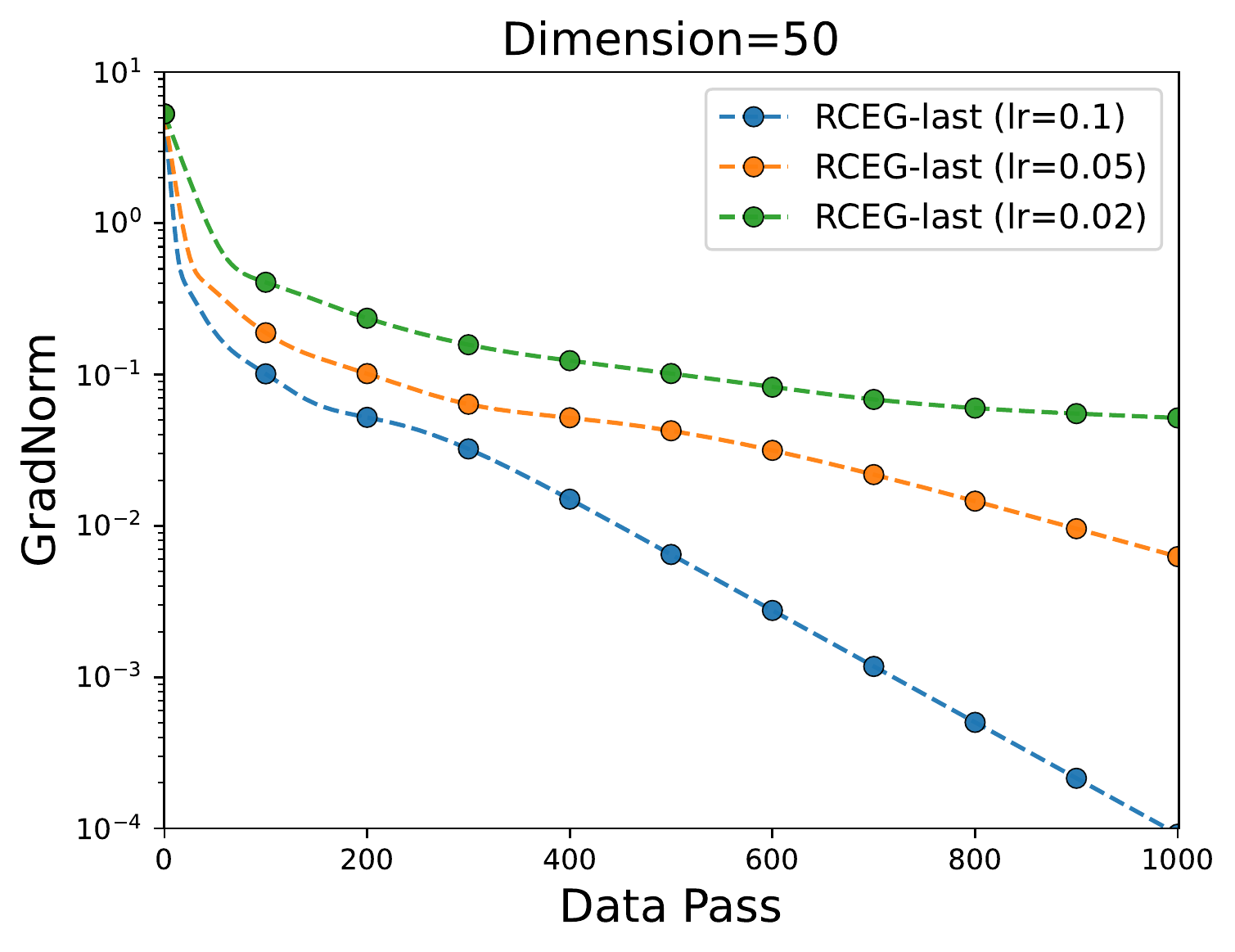}
\caption{\footnotesize{Comparison of different step sizes ($\eta \in \{0.1, 0.05, 0.02\}$) for solving the RPCA problem with different dimensions when $\alpha=2.0$. The horizontal axis represents the number of data passes and the vertical axis represents gradient norm.}}
\label{fig:exp-deterministic-stepsize-appendix}
\end{figure*}

\end{document}